\documentclass[10pt,reqno,a4paper]{amsart}
\usepackage[dvipsnames, table]{xcolor}
\usepackage{amsthm}
\theoremstyle{plain}
\usepackage{amssymb}
\usepackage{marvosym}
\usepackage{bm}
\usepackage{mathrsfs}
\usepackage{enumerate}  
\usepackage{mathtools}
\usepackage{tikz-cd}
\usepackage{tikz}
\usepackage{graphicx}
\usepackage{float}
\usepackage[titletoc,title]{appendix}
\usepackage{marginnote}
\usepackage{todonotes}
\usepackage{etoolbox}
\usepackage[all]{xy}
\makeatletter
\patchcmd{\Ginclude@eps}{"#1"}{#1}{}{}
\makeatother
\usepackage[outdir=./]{epstopdf}
\usepackage[numbers]{natbib}
\usepackage[utf8]{inputenc}
\usepackage[english]{babel}
\usepackage{changepage}
\usepackage{makecell}
\usepackage{enumitem}
\usepackage{comment}

\usepackage[colorlinks,citecolor=blue]{hyperref}

\definecolor{lightblue}{HTML}{1F88CD}
\definecolor{lightgrey}{HTML}{727272}
\definecolor{lightblue2}{HTML}{009EC1}
\definecolor{mypink}{HTML}{FD00B0}
\definecolor{lightred}{HTML}{ff4d4d}

\usepackage{hyperref}
\hypersetup{
	colorlinks=true,
    linkcolor={OliveGreen},
    citecolor={blue},
	urlcolor={black}
}

\newtheorem*{theorem*}{Theorem}
\newtheorem{theorem}{Theorem}[section]
\newtheorem{corollary}[theorem]{Corollary}
\newtheorem{lemma}[theorem]{Lemma}
\newtheorem{conjecture}[theorem]{Conjecture}

\newtheorem{proposition}[theorem]{Proposition}
\theoremstyle{definition}

\theoremstyle{definition}
\newtheorem{definition}[theorem]{Definition}
\theoremstyle{definition}
\newtheorem{remark}[theorem]{Remark}
\theoremstyle{definition}

\theoremstyle{definition}

\theoremstyle{definition}

\theoremstyle{definition}

\theoremstyle{definition}

\theoremstyle{definition}
\newtheorem{question!}[theorem]{Question!}
\theoremstyle{definition}

\makeatletter
\newcommand*\sbt{\mathpalette\sbt@{.75}}
\newcommand*\sbt@[2]{\mathbin{\vcenter{\hbox{\scalebox{#2}{$\m@th#1\bullet$}}}}}
\makeatother

\newcommand{\ra}{\rightarrow}
\newcommand{\xra}{\xrightarrow}

\newcommand{\wt}{\widetilde}

\newcommand{\sst}{\subset}

\newcommand{\bR}{\bm{\mathrm{R}}}
\newcommand{\bL}{\bm{\mathrm{L}}}

\newcommand{\D}{\mathrm{D}}

\newcommand{\KK}{\mathrm{K}}

\newcommand{\ZZ}{\mathbb{Z}}

\newcommand{\QQ}{\mathbb{Q}}

\newcommand{\CC}{\mathbb{C}}
\newcommand{\PP}{\mathbb{P}}

\newcommand{\bA}{\mathbf{A}}
\newcommand{\bB}{\mathbf{B}}

\newcommand{\ch}{\mathrm{ch}}

\newcommand{\perf}{\mathrm{perf}}

\newcommand{\pr}{\mathrm{pr}}

\newcommand{\C}{\mathbb{C}}
\newcommand{\Hilb}{\mathrm{Hilb}}
\newcommand{\Knum}{\mathrm{K}_{\mathrm{num}}}

\newcommand{\3}{\mathcal{K}_1}
\newcommand{\4}{\mathcal{K}_2}

\renewcommand{\Re}{\operatorname{Re}}
\renewcommand{\Im}{\operatorname{Im}}

\DeclareMathOperator{\Aut}{Aut}

\DeclareMathOperator{\identity}{id}

\DeclareMathOperator{\Ext}{Ext}

\DeclareMathOperator{\Hom}{Hom}
\DeclareMathOperator{\RHom}{RHom}

\DeclareMathOperator{\ext}{ext}

\DeclareMathOperator{\Spec}{Spec}

\DeclareMathOperator{\cone}{cone}
\DeclareMathOperator{\Stab}{Stab}
\DeclareMathOperator{\Gr}{Gr}

\newcommand{\GL}{\widetilde{\mathrm{GL}}^+(2,\mathbb{R})}
\newcommand{\cX}{\mathcal{X}}
\newcommand{\cY}{\mathcal{Y}}

\newcommand{\cA}{\mathcal{A}}

\newcommand{\cU}{\mathcal{U}}
\newcommand{\cH}{\mathcal{H}}

\newcommand{\cK}{\mathcal{K}}

\newcommand{\cT}{\mathcal{T}}
\newcommand{\cQ}{\mathcal{Q}}
\newcommand{\Ku}{\mathcal{K}u}
\newcommand{\cP}{\mathcal{P}}

\newcommand{\cD}{\mathcal{D}}

\newcommand{\cM}{\mathcal{M}}
\newcommand{\cV}{\mathcal{V}}

\DeclareMathOperator{\oh}{\mathcal{O}}

\usetikzlibrary{decorations.pathmorphing}
\begin{document}

\title[Lagrangian families of Bridgeland moduli spaces]{Lagrangian families of Bridgeland moduli spaces from Gushel--Mukai fourfolds}

\subjclass[2020]{Primary 14F08; secondary 14J42, 14J45, 14D20, 14D23}
\keywords{Bridgeland moduli spaces, Kuznetsov components, Gushel--Mukai varieties, Cubic fourfolds, Hyperk\"ahler manifolds, Lagrangian subvarieties, Lagrangian covering families}

\address{Department of Mathematics, Imperial College, London SW7 2AZ, United Kingdom}
\email{s.feyzbakhsh@imperial.ac.uk}

\address{Shanghai Center for Mathematical Sciences, Fudan University, Jiangwan Campus, 2005 Songhu Road, Shanghai, 200438, China
}
\email{hanfeiguo23@gmail.com}

\address{School of Mathematical Sciences, Zhejiang University, Hangzhou, Zhejiang Province 310058, P. R. China}
\email{jasonlzy0617@gmail.com}


\address{Sun Yat-Sen University, School of Mathematics, Guangzhou 510275, China}

\address{Center for Geometry and Physics, Institute for Basic Science. 
79, Jigok-ro 127beon-gil, Nam-gu, Pohang-si,
Gyeongsangbuk-do,
Republic of Korea 37673}
\email{shizhuozhang@msri.org, zszmath@ibs.re.kr}

\author{Soheyla Feyzbakhsh, Hanfei Guo, Zhiyu Liu, Shizhuo Zhang}
\address{}
\email{}

\begin{abstract}
Let $X$ be a very general Gushel--Mukai (GM) variety of dimension $n\geq 4$, and let $Y$ be a smooth hyperplane section. There are natural pull-back and push-forward functors between the semi-orthogonal components (known as the Kuznetsov components) of the derived categories of $X$ and $Y$. In this paper, we prove that the Bridgeland stability of objects is preserved by both pull-back and push-forward functors. We then explore various applications of this result, such as constructing an $8$-dimensional smooth family of Lagrangian subvarieties for each moduli space of stable objects in the Kuznetsov component of a general GM fourfold and proving the projectivity of the moduli spaces of semistable objects of any class in the Kuznetsov component of a general GM threefold, as conjectured by Perry, Pertusi, and Zhao.

\end{abstract}

\maketitle

{
\hypersetup{linkcolor=blue}
\setcounter{tocdepth}{1}
\tableofcontents
}

\section{Introduction}

Hyperk\"ahler manifolds play a central role in algebraic geometry and complex geometry. However, very few examples of them are known and it is challenging to construct explicit examples of projective hyperk\"ahler manifolds. 

One of the most studied examples comes from K3 surfaces: the Hilbert schemes of points on K3 surfaces (see \cite{beauville:hilb-on-k3}), or more generally, moduli spaces of stable sheaves on them (see \cite{mukai:moduli-K3-I}). More recently, instead of considering the moduli space of stable sheaves on K3 surfaces, people consider categories that behave homologically as actual K3 surfaces, known as \emph{noncommutative K3 surfaces} or \emph{K3 categories},  and construct projective hyperk\"ahler varieties from them as Bridgeland moduli spaces of stable objects, provided suitable stability conditions are introduced. 

An important source of K3 categories is provided by the \emph{Kuznetsov component} of a cubic fourfold or a Gushel--Mukai fourfold/sixfold $X$, which is a quadric section of the projective cone over the Grassmannian $\mathrm{Gr}(2,5)$. The Kuznetsov component, denoted by $\Ku(X)$, is the right orthogonal complement of a collection of rigid bundles on $X$. 
In \cite{bayer2017stability, BLMNPS21, perry2019stability}, a family of Bridgeland stability conditions on $\Ku(X)$ is constructed and it is shown that the moduli spaces of semistable objects in $\Ku(X)$ are projective hyperk\"ahler manifolds under natural assumptions. 

\medskip

In this paper, we mainly focus on Gushel--Mukai (GM) varieties $X$ of dimension $n \geq 3$ (see Section \ref{sec-GM-1} for a review). There is a semi-orthogonal decomposition
\[\D^b(X)=\langle\Ku(X),\oh_X,\cU^{\vee}_X,\cdots,\oh_X((n-3)H),\cU_X^{\vee}((n-3)H)\rangle,\]
where $\cU_X$ is the pull-back of the tautological subbundle on $\Gr(2, 5)$. 
\begin{itemize}
    \item If $n=3$ or $5$, then $\Ku(X)$ is an Enriques category, i.e.~its Serre functor is of the form $T_X \circ [2]$ for a non-trivial involution $T_X$ of $\Ku(X)$. The numerical Grothendieck group of $\Ku(X)$ is of rank 2, generated by classes $\lambda_1$ and $\lambda_2$, and there is a unique Serre-invariant stability condition\footnote{We always consider them up to $\GL$-action, see Definition \ref{def-serre-invariant}. All known stability conditions on $\Ku(X)$ are Serre-invariant and they share several nice properties, see e.g.~Section \ref{subsec-3.1} and Section \ref{sec-GM-1}.} on $\Ku(X)$.

\item If $n=4$ or $6$, then $\Ku(X)$ is a K3 category, i.e.~its Serre functor is $[2]$. The numerical Grothendieck group of $\Ku(X)$ contains a canonical rank two sublattice, generated by classes $\Lambda_1$ and $\Lambda_2$. All stability conditions on $\Ku(X)$ are Serre-invariant, and there is a set of stability conditions $\Stab^{\circ}(\Ku(X))$ constructed in \cite[Theorem 4.12]{perry2019stability}. Moreover, 
when $X$ is very general, there is a unique stability condition on $\Ku(X)$ (cf.~Proposition \ref{prop-gm4-unique}).
\end{itemize}

\subsection*{Preservation of stability under push-forward and pull-back} Let $j\colon Y \hookrightarrow X$ be a smooth hyperplane section of a GM variety $X$. The restriction of the derived pull-back to $\Ku(Y)$ yields the functor
\begin{equation*}
    j^* \colon \Ku(X) \to \Ku(Y). 
\end{equation*}
However, the image of the push-forward does not always lie in the Kuznetsov component, so we need to project it into $\Ku(X)$. This gives us the functor
\begin{equation*}
    \pr_X \circ j_* \colon \Ku(Y) \to \Ku(X). 
\end{equation*}
See Section \ref{sec-GM-1} for more details. The primary result of our paper proves the preservation of stability under both these functors.

\begin{theorem}[{Theorem \ref{thm-very-general-GM4}}]\label{thm-intro-section}
Let $X$ be a GM variety of dimension $n\geq 4$ and $j \colon Y\hookrightarrow X$ be a smooth hyperplane section. Let $\sigma_Y$ and $\sigma_X$ be Serre-invariant stability conditions on $\Ku(Y)$ and $\Ku(X)$, respectively. We additionally assume that whichever of $X$ and $Y$ has even dimensionality is considered to be very general.
\begin{enumerate}
    \item An object $E\in \Ku(Y)$ is $\sigma_Y$-semistable if and only if $\pr_X(j_*E)\in \Ku(X)$ is $\sigma_X$-semistable.

    \item An object $F\in \Ku(X)$ is $\sigma_X$-semistable if and only if $j^*F\in \Ku(Y)$ is $\sigma_Y$-semistable.
\end{enumerate}
\end{theorem}

We refer to Theorem \ref{thm-very-general-GM4} for a more complete version, describing when $\pr_X(j_*E)$ or $j^*F$ are stable. The above theorem can be viewed as a noncommutative analogue of the stability of push-forward/pull-back objects for the embedding of a curve into a K3 surface.

Via a deformation argument, we can generalize Theorem \ref{thm-intro-section}(1) to general GM fourfolds:

\begin{theorem}[{Theorem \ref{thm-general-gm4}}]\label{thm-intro-section-general}
Let 
$X$ be a general GM fourfold, and $j \colon Y\hookrightarrow X$ be a smooth hyperplane section. Given $\sigma_X\in \Stab^{\circ}(\Ku(X))$ and a Serre-invariant stability condition $\sigma_Y$ on $\Ku(Y)$. If $E\in \Ku(Y)$ is a $\sigma_Y$-semistable object, 
then $\pr_X(j_*E)$ is $\sigma_X$-semistable.
\end{theorem}

For a generalization of Theorem \ref{thm-intro-section}(2), see Theorem \ref{thm-general-gm3}. Once a construction of stability conditions of families of GM fivefolds and sixfolds is known, we can generalize both parts of Theorem \ref{thm-intro-section} to general GM varieties. See Remark \ref{rmk-general-gm-varieties}. We also conjecture that Theorem~\ref{thm-intro-section} holds generically for every GM fourfold and cubic fourfold (cf.~Conjecture~\ref{main-conjecture-1}).


Using Theorem \ref{thm-intro-section} and \ref{thm-intro-section-general}, we can construct families of Lagrangian subvarieties as follows. We denote by $M^X_{\sigma_X}(a,b)$ (resp.~$M^Y_{\sigma_Y}(a,b)$) the moduli space that parameterizes S-equivalence classes of $\sigma_X$-semistable (resp.~$\sigma_Y$-semistable) objects of class $a\Lambda_1 +b\Lambda_2$ (resp.~$a\lambda_1 +b\lambda_2$) in $\Ku(X)$ (resp.~$\Ku(Y)$). Then by \cite{perry2019stability}, for any pair of coprime integers $a,b$ and a generic $\sigma_X$, the space $M^X_{\sigma_X}(a,b)$ is a projective hyperk\"ahler manifold. In this case, Theorem \ref{thm-intro-section} and \ref{thm-intro-section-general} induces a morphism
\[M^Y_{\sigma_Y}(a,b)\to M^X_{\sigma_X}(a,b),\quad E\mapsto \pr_X(j_*E)\]
which is finite and unramified when $Y\subset X$ is a general hyperplane section, and we show that its image is a Lagrangian subvariety of the hyperk\"ahler manifold $M^X_{\sigma_X}(a,b)$ (cf.~Theorem \ref{thm-lagrangian-general}). Therefore, as we vary the hyperplane section $Y\subset X$, we obtain a family of Lagrangian subvarieties of $M^X_{\sigma_X}(a,b)$, which can be arranged into a Lagrangian family as detailed below.

\begin{theorem}[{Theorem \ref{thm-lagrangian-family}}]
Let $a,b$ be a pair of coprime integers and $X$ be a general GM fourfold. Then for any $\sigma_X\in \Stab^{\circ}(\Ku(X))$ generic\footnote{Here, $\sigma_X$ is generic means that $\sigma_X$ does not lie on any wall, which ensures the moduli space $M^X_{\sigma_X}(a,b)$ is a hyperk\"ahler manifold.} with respect to $a\Lambda_1+b\Lambda_2$, there is a Lagrangian family (cf.~Definition \ref{def-lag-family}) of $M^X_{\sigma_X}(a,b)$ over an open dense subset of $|\oh_X(H)|$.
\end{theorem}

For GM sixfolds, there are similar constructions in Corollary \ref{cor-lagrangian-family-6fold}. One may wonder if the Lagrangian families constructed above are Lagrangian covering families (cf.~Definition \ref{def-lag-family}), provided $\dim M^X_{\sigma_X}(a,b)$ is small. We can verify this except in the case when $\dim M^X_{\sigma_X}(a,b)=12$, as discussed in Section \ref{subsec-lag-cover}. In a sequel \cite{FGLZcube}, we demonstrate that all known constructions of Lagrangian covering families of hyperk\"ahler manifolds coming from GM fourfolds or cubic fourfolds can be recovered by the categorical method described in Theorem~\ref{thm-lagrangian-family}. 

\medskip 

In Section \ref{subsec-application}, we discuss two other applications of Theorem \ref{thm-intro-section} and \ref{thm-intro-section-general}. We first confirm a conjecture proposed in \cite[Section 5.4.3]{perry2019stability}.

\begin{corollary}[{Corollary \ref{cor-rational-embedding}}]
Let $X$ be a general GM fourfold. For any $\sigma_X\in \Stab^{\circ}(\Ku(X))$ generic with respect to $\Lambda_1+2\Lambda_2$, the functor $\pr_X$ induces a rational map
\[X\dashrightarrow M^X_{\sigma_X}(1,2)\]
which sends the structure sheaf at a general point to its projection to the Kuznetsov component. Moreover, the map is generically an embedding.
\end{corollary}

As a second application, we prove the projectivity of the moduli spaces of semistable objects in $\Ku(Y)$ of any class for a general GM threefold $Y$, which improves \cite[Theorem 1.3(2)]{ppzEnriques2023} and confirms the expectation in \cite[Remark 1.4]{ppzEnriques2023} for general GM threefolds (which can be either ordinary or special).

\begin{corollary}[{Corollary \ref{cor-projective}}]
Let $Y$ be a general GM threefold. For any pair of integers $a,b$ and any Serre-invariant stability condition $\sigma_Y$ on $\Ku(Y)$, the moduli space of $\sigma_Y$-semistable objects of class $a\lambda_1 +b\lambda_2$ is a projective scheme.
\end{corollary}


\subsection*{Lagrangian constant cycle subvarieties} 
As an application of studying a GM fourfold/sixfold $X$ via its hyperplane sections, we can construct a Lagrangian constant cycle subvariety for each Bridgeland moduli space of stable objects in $\Ku(X)$. Recall that there is an involution $T_X$ defined on $\Ku(X)$ (cf.~Lemma~\ref{lem-involution}). And there is a family of stability conditions $\Stab^{\circ}(\Ku(X))$ on $\Ku(X)$ constructed in \cite[Theorem 4.12]{perry2019stability}. According to \cite[Remark 5.8]{bayer2022kuznetsov} and \cite[Theorem 1.6]{ppzEnriques2023}, the functor $T_X$ 
induces an anti-symplectic involution on $M_{\sigma_X}^X(a,b)$ and the fixed locus $\mathrm{Fix}(T_X)$ is a smooth Lagrangian subvariety of $M_{\sigma_X}^X(a,b)$. 

\begin{corollary}[{Theorem \ref{thm-CH0}}]\label{cor-co}
Let $X$ be a GM fourfold or sixfold and $a,b$ be a pair of coprime integers. For any stability condition $\sigma_X\in \Stab^{\circ}(\Ku(X))$ which is generic with respect to $a\Lambda_1+b\Lambda_2$, the fixed locus $\mathrm{Fix}(T_X)$ is a Lagrangian constant cycle subvariety of $M_{\sigma_X}^X(a,b)$, i.e.~all points of $\mathrm{Fix}(T_X)$ are rationally equivalent in $M_{\sigma_X}^X(a,b)$.
\end{corollary}

The existence of Lagrangian constant cycle subvarieties for any projective hyperk\"ahler manifold is conjectured by Voisin in \cite[Conjecture 0.4]{voisin:remark-coisotropic}. Therefore, Corollary \ref{cor-co} provides an infinite series of locally complete $20$-dimensional families of hyperk\"ahler varieties satisfying this conjecture. By \cite[Theorem 1.8(2)]{ppzEnriques2023}, this result can be regarded as a noncommutative version of \cite[Proposition 4.2]{Beckmann20}. 

\subsection*{Related work}
The idea of constructing Lagrangian subvarieties of hyperk\"ahler varieties arising from cubic or GM fourfolds using their hyperplane sections has been realized in several classical results. For instance, the Fano variety of lines of a cubic fourfold by \cite{voisin1992stabilite} and the LLSvS eightfold associated with a general cubic fourfold due to \cite[Proposition 6.9]{shinder2017geometry}. On the other hand, for a GM fourfold $X$, according to \cite[Proposition 5.2, 5.6]{iliev2011fano}, one can construct concrete Lagrangians of the double EPW sextic either from hyperplane sections of $X$ or fivefolds that contain $X$ as a hyperplane section. 

In the recent paper \cite{ppzEnriques2023}, the authors study the fixed locus of the anti-symplectic involution of hyperk\"ahler varieties constructed from the Kuznetsov components of GM fourfolds and show that they are nonempty, thus Lagrangian. The difference is that their method relies on the relation between the Kuznetsov component of a GM threefold and a special GM fourfold branched along it, so the Lagrangian subvariety can not vary in the family of hyperplane sections.

As for the hyperk\"ahler varieties admitting Lagrangian covering families, 
the work of \cite{voisin:remark-coisotropic}, \cite{voisin2021lefschetz}, and \cite{Bai22} exploit their geometry and establish the Lefschetz standard conjecture of degree 
two and other well-behaved cohomological properties. 


The notion of Lagrangian constant cycle subvarieties was introduced in \cite{Huy14} for the case of K3 surfaces. As a specific case of \cite[Conjecture 0.4]{voisin:remark-coisotropic}, it is predicted that there exists a Lagrangian constant cycle subvariety in any projective hyperk\"ahler manifold. In \cite{lin:constant-cycle-lag-fibration}, the author constructs Lagrangian constant cycle subvarieties for hyperk\"ahler varieties admitting a Lagrangian fibration. The existence of such subvarieties has also been verified for some cases, such as \cite{Huy14, voisin:remark-coisotropic, Beckmann20, zhang2023one, li2023bloch}.

\subsection*{Plan of the paper} 
In Section \ref{sec-preliminay}, we provide the necessary definitions and properties of semi-orthogonal decompositions. Then we introduce the Kuznetsov components of several Fano manifolds and recall the construction of stability conditions on them. In Section \ref{sec-general-criterion}, we prove Theorem \ref{thm-main}, which shows that the stability of objects is preserved under push-forward and pull-back between triangulated categories satisfying additional assumptions. Then in Section \ref{sec-very-general-GM}, we apply Theorem \ref{thm-main} to non-Hodge-special GM varieties of dimension $n\geq 4$ and prove Theorem \ref{thm-very-general-GM4}. Then we obtain Theorem \ref{thm-general-gm4} via a deformation argument. In Section \ref{sec-lag-family}, we use Theorem \ref{thm-very-general-GM4} and \ref{thm-general-gm4} to construct Lagrangian families of those Bridgeland moduli spaces on the Kuznetsov components of non-Hodge-special GM fourfolds in Theorem \ref{thm-lagrangian-family}. Next, in Section \ref{subsec-lag-cover}, we discuss when the family in Theorem \ref{thm-lagrangian-family} is a Lagrangian covering family. 

In Section \ref{subsec-application}, we discuss other applications and prove Corollary \ref{cor-rational-embedding}, Corollary \ref{cor-projective}, and Theorem \ref{thm-CH0}. 
Finally, in Appendix \ref{App-cubic}, we state Conjecture \ref{main-conjecture-1} and explain a general recipe to construct Lagrangian subvarieties of hyperk\"ahler varieties as Bridgeland moduli spaces on GM fourfolds and cubic fourfolds.


\subsection*{Notation and conventions} \leavevmode
\begin{itemize}
 
    
    \item All triangulated categories are assumed to be $\CC$-linear of finite type, i.e.~$\sum_{i\in \ZZ} \dim_{\CC} \Ext^i(E,F)$ is finite for any two objects $E,F$.

    \item We use $\hom$ and $\ext^{i}$ to represent the dimension of the vector spaces $\Hom$ and~$\Ext^{i}$. We denote $\RHom(-,-)=\bigoplus_{i\in \ZZ} \Hom(-,-[i])[-i]$ and $\chi(-,-)=\sum_{i\in \ZZ} (-1)^i \ext^i(-,-)$.
    
    \item For a triangulated category $\cT$, its Grothendieck group and numerical Grothendieck group are denoted by $\mathrm{K}(\cT)$ and $\Knum(\cT):=\KK(\cT)/\ker(\chi)$, respectively.



\end{itemize}

\subsection*{Acknowledgements}
It is our pleasure to thank Arend Bayer, Sasha Kuznetsov, and Qizheng Yin for very useful discussions on the topics of this project. We would like to thank Enrico Arbarello, Marcello Bernardara, Lie Fu, Yong Hu, Grzegorz Kapustka, Chunyi Li, Zhiyuan Li, Laurent Manivel, Kieran O'Grady, Alexander Perry, Laura Pertusi, Richard Thomas, Claire Voisin, Ruxuan Zhang, Yilong Zhang, and Xiaolei Zhao for helpful conversations. We thank Ruxuan Zhang for pointing out a mistake in the statement of Proposition \ref{prop-Marian-Zhao} in the first version of our paper. SF acknowledges the support of EPSRC postdoctoral fellowship EP/T018658/1 and the Royal Society URF/ R1/231191. HG is supported by NSFC grant (12121001, 12171090 and 12425105). SZ is supported by the ERC Consolidator Grant WallCrossAG, no. 819864, ANR project FanoHK, grant ANR-20-CE40-0023 and partially supported by GSSCU2021092. SZ is also supported by the NSF under grant No. DMS-1928930, while he is residence at the Simons Laufer Mathematical Sciences Institute in Berkeley, California. Part of the work was finished during the junior trimester program of SF, ZL, and SZ funded by the Deutsche Forschungsgemeinschaft (DFG, German Research Foundation) under Germany’s Excellence Strategy – EXC-2047/1 – 390685813. We would like to thank the Hausdorff
Research Institute for Mathematics for their hospitality.

\section{Preliminaries}\label{sec-preliminay}

In this section, we recall some basic definitions and properties of semi-orthogonal decompositions and the notion of Bridgeland stability conditions. 
We mainly follow from \cite{bridgeland:stability} and \cite{bayer2017stability}.

\subsection{Semi-orthogonal decompositions} \label{sod}
Let $\cT$ be a triangulated category and $\cD\subset \cT$ a full triangulated subcategory. We define the \emph{right orthogonal complement} of $\cD$ in $\cT$ as the full triangulated subcategory
\[ \cD^\bot = \{ X \in \cT \mid \Hom(Y, X) =0  \text{ for all } Y \in \cD \}.  \]
The \emph{left orthogonal complement} is defined similarly, as 
\[ {}^\bot \cD = \{ X \in \cT \mid \Hom(X, Y) =0  \text{ for all } Y \in \cD \}.  \]
We say a full triangulated subcategory $\cD \sst \cT$ is \emph{admissible} if the inclusion functor $i \colon \cD \hookrightarrow \cT$ has left adjoint $i^*$ and right adjoint $i^!$.
Let $( \cD_1, \dots, \cD_m  )$ be a collection of admissible full subcategories of $\cT$. We say that $\cT = \langle \cD_1, \dots, \cD_m \rangle$ is a \emph{semi-orthogonal decomposition} of $\cT$ if $\cD_j \sst \cD_i^\bot $ for all $i > j$, and the subcategories $(\cD_1, \dots, \cD_m )$ generate $\cT$, i.e.~the category resulting from taking all shifts and cones in the categories $(\cD_1, \dots, \cD_m )$ is equivalent to $\cT$.

Let $i\colon \cD \hookrightarrow \cT$ be an admissible full subcategory. Then the \emph{left mutation functor} $\bL_{\cD}$ through $\cD$ is defined as the functor lying in the canonical functorial exact triangle 
\[  i i^! \ra \identity \ra \bL_{\cD}   \]
and the \emph{right mutation functor} $\bR_{\cD}$ through $\cD$ is defined similarly, by the triangle 
\[ \bR_{\cD} \ra \identity \ra i i^*  .  \]
Therefore, $\bL_{\cD}$ is exactly the left adjoint functor of $\cD^\bot\hookrightarrow \cT$. Similarly, $\bR_{\cD}$ is the right adjoint functor of ${}^\bot \cD\hookrightarrow \cT$. When $E \in \cT$ is an exceptional object, and $F \in \cT$ is any object, the left mutation $\bL_E F:=\bL_{\langle E\rangle} F$ fits into the triangle 
\[ E \otimes \RHom(E, F) \ra F \ra \bL_E F , \]
and the right mutation $\bR_E F:=\bR_{\langle E\rangle} F$ fits into the triangle
\[ \bR_E F \ra F \ra E \otimes \RHom(F, E)^\vee  . \]

Given a semi-orthogonal decomposition $\cT = \langle \cD_1 , \cD_2 \rangle$. Then $\cT \simeq \langle S_{\cT}(\cD_2), \cD_1 \rangle \simeq \langle \cD_2, S^{-1}_{\cT}(\cD_1) \rangle$ are also semi-orthogonal decompositions of $\cT$, where $S_{\cT}$ is the Serre functor of $\cT$. Moreover, \cite[Lemma 2.6]{kuz:fractional-CY} shows that  
\[ S_{\cD_2} = \bR_{\cD_1} \circ S_{\cT}  \, \, \, \text{ and } \, \, \, S_{\cD_1}^{-1} = \bL_{\cD_2} \circ S_{\cT}^{-1} . \]

\subsection{Stability conditions}

Let $\cT$ be a triangulated category and $\KK(\cT)$ be its Grothendieck group. 
Fix a surjective morphism to a finite rank lattice $v \colon \KK(\cT) \ra \Lambda$. 

\begin{definition}
A \emph{stability condition} on $\cT$ is a pair $\sigma = (\cA, Z)$, where $\cA$ is the heart of a bounded t-structure on $\cT$ and $Z \colon \Lambda \ra \CC$ is a group homomorphism such that 
\begin{enumerate}
    \item the composition $Z \circ v : \KK(\cA) \cong \KK(\cT) \ra \CC$ is a stability function on $\cA$, i.e.~for any $E \in \cA$, we have $\Im Z(v(E)) \geq 0$ and if $\Im Z(v(E)) = 0$, $\Re Z(v(E)) < 0$. From now on, we write $Z(E)$ rather than $Z(v(E))$.
\end{enumerate}
For any object $E \in \cA$, we define the slope function $\mu_{\sigma}(-)$ as
\[
\mu_\sigma(E) := \begin{cases}  - \frac{\Re Z(E)}{\Im Z(E)}, & \Im Z(E) > 0 \\
+ \infty , & \text{else}.
\end{cases}
\]
An object $0 \neq E \in \cA$ is called $\sigma$-(semi)stable if for any proper subobject $F \sst E$, we have $\mu_\sigma(F) (\leq) \mu_\sigma(E)$. 
\begin{enumerate}[resume]
    \item Any object $E \in \cA$ has a Harder--Narasimhan filtration in terms of $\sigma$-semistability defined above.
    \item There exists a quadratic form $Q$ on $\Lambda \otimes \mathbb{R}$ such that $Q|_{\ker Z}$ is negative definite  and $Q(E) \geq 0$ for all $\sigma$-semistable objects $E \in \cA$. This is known as the \emph{support property}.
\end{enumerate}
\end{definition}

The \emph{phase} of a $\sigma$-semistable object $E\in \cA$ is defined as
\[\phi_{\sigma}(E):=\frac{1}{\pi}\mathrm{arg}(Z(E))\in (0,1].\]
For $n\in \ZZ$, we set $\phi_{\sigma}(E[n]):=\phi_{\sigma}(E)+n$. Given a non-zero object $E\in \cT$, we will denote by $\phi^+_{\sigma}(E)$ (resp.~$\phi^-_{\sigma}(E)$) the biggest (resp.~smallest) phase of a Harder--Narasimhan factor of $E$ with respect to $\sigma$.

A \emph{slicing} $\mathcal{P}_{\sigma}$ of $\mathcal{T}$ with respect to the stability condition $\sigma$ consists of full additive subcategories $\mathcal{P}_{\sigma}(\phi) \subset \mathcal{T}$ for each $\phi \in \mathbb{R}$ such that the subcategory $\mathcal{P}_{\sigma}(\phi)$ contains the zero object and all $\sigma$-semistable objects whose phase is $\phi$.
For any interval $I\subset \mathbb{R}$, we denote by $\cP_{\sigma}(I)$ the category given by the extension closure of $\{\cP_{\sigma}(\phi)\}_{\phi\in I}$. We will use both notations $\sigma = (\cA,Z)$ and $\sigma = (\cP_{\sigma},Z)$ for a stability condition $\sigma$ with heart $\cA = \cP_{\sigma}((0,1])$. 

We finish this section with the following definition:

\begin{definition}\label{def-serre-invariant}
Let $\cT$ be a triangulated category and $\Phi$ be an auto-equivalence of $\cT$. We say a stability condition $\sigma$ on $\cT$ is \emph{$\Phi$-invariant}\footnote{The action of $\Phi$ and $\GL$ on $\sigma$ is defined in \cite[Lemma 8.2]{bridgeland:stability}.} if
\[\Phi\cdot \sigma=\sigma\cdot\wt{g}\]
for an element $\wt{g}\in \GL$. We say $\sigma$ is \emph{Serre-invariant} if it is $S_{\cT}$-invariant, where $S_{\cT}$ is the Serre functor of $\cT$.
\end{definition}

\section{Stability of push-forward and pull-back objects}\label{sec-general-criterion}

In this section,
we are going to prove Theorem \ref{thm-main}, which will be applied to the Kuznetsov components of GM varieties in Section \ref{sec-very-general-GM}. We consider two triangulated categories $\mathcal{K}_1$ and $\mathcal{K}_2$ with the following properties:  
\begin{enumerate}[label=\textbf{(C\arabic*)}]
    \item\label{c1} $\Knum(\3) = \ZZ \lambda_1 \oplus \ZZ \lambda_2$ whose Euler pairing is 
\begin{equation}
\left[               
\begin{array}{cc}   
-1 & 0 \\  
0 & -1\\
\end{array}
\right].
\end{equation} 
\item\label{c2} $\Knum(\4) = \ZZ \Lambda_1 \oplus \ZZ \Lambda_2$ whose Euler pairing is 
\begin{equation}
\left[               
\begin{array}{cc}   
-2 & 0 \\  
0 & -2\\
\end{array}
\right].
\end{equation}
\item\label{c-T} There exist exact auto-equivalences $T_i$ of $\mathcal{K}_i$ for $i=1, 2$ such that $T_i^2= \mathrm{id}_{\cK_i}$ and $T_i$ acts trivially on $\Knum(\mathcal{K}_i)$.  
    \item\label{c-S} We have the Serre functors $S_{\4} = [2]$ and $S_{\3} = T_1\circ [2]$. 
    \item\label{c-adjoints} We have exact functors
    \[\Phi^*\colon \cK_2\to \cK_1,\quad \Phi_*\colon \cK_1\to \cK_2\]
    with adjoint pairs $\Phi_* \circ T_1 \dashv \Phi^* \dashv \Phi_*$ so that $\Phi_*(\lambda_i) = \Lambda_i$ and $\Phi^*(\Lambda_i) = 2\lambda_i$ for $i=1, 2$.  Moreover, for any object $F \in \4$, we have $\Phi^*T_2(F) \cong T_1(\Phi^*F)$.
    \item \label{c-stability}
    There are stability conditions $\sigma_i =(Z_i, \mathcal{A}_i)$ on $\mathcal{K}_i$ for $i=1, 2$ such that 
    \begin{enumerate}
        \item $\sigma_i$ is $T_i$-invariant, 
        \item $Z_1(\lambda_1) = Z_2(\Lambda_1) =-1$ and $Z_1(\lambda_2) = Z_2(\Lambda_2) =\mathfrak{i}$,
        \item there are infinitely many $\sigma_1$-stable objects of class $\lambda_i$ for each $i=1, 2$,
        
        \item for any $\sigma_1$-stable object $E \in \3$ of class $\lambda_i$ for $i=1, 2$, we know $\Phi_*E$ is $\sigma_2$-stable with $\phi_{\sigma_1}(E)=\phi_{\sigma_2}(\Phi_*E)$.
    \end{enumerate}
    \item\label{c-t1} Any object $E \in \3$ lies in an exact triangle 
    \begin{equation}\label{exact-3}
        T_1(E) \to \Phi^*\Phi_*E \to E. 
    \end{equation}
    \item\label{c-t2} Any object $E \in \4$ lies in an exact triangle
    \begin{equation}\label{exact-4}
        E \to \Phi_*\Phi^*E \to T_2(E). 
    \end{equation}
\end{enumerate}
Moreover, for simplicity, we assume that $\cK_i$ for each $i=1,2$ is a full triangulated subcategory of the derived category of coherent sheaves on a smooth variety\footnote{Lemma \ref{spec-seq} is indeed the only instance where we relied on this assumption.}. We denote the phase function and the slope function of $\sigma_i$ by $\phi_i(-)$ and $\mu_i(-)$, respectively.

The main theorem that we prove in this section is as follows. 
\begin{theorem}\label{thm-main}
    Consider triangulated categories $\3$ and $\4$ satisfying the above conditions \ref{c1} to \ref{c-t2}.

\begin{enumerate}
    \item An object $E \in \3$ is $\sigma_1$-semistable of  phase $\phi$ if and only if $\Phi_*E$ is $\sigma_2$-semistable of phase $\phi$. Moreover, $\Phi_*(E)$ is $\sigma_2$-stable if and only if $E$ is $\sigma_1$-stable and for any $\sigma_2$-stable object $F$ we have $\Phi^*F\neq E$.

    \item An object $F \in \4$ is $\sigma_2$-semistable of  phase $\phi$ if and only if $\Phi^*F$ is $\sigma_1$-semistable of phase $\phi$. Moreover, $\Phi^*(F)$ is $\sigma_1$-stable if and only if $F$ is $\sigma_2$-stable and for any $\sigma_1$-stable object $E$ we have $\Phi_*E\neq F$.
\end{enumerate}

\end{theorem}

\subsection{Collection of lemmas}\label{subsec-3.1}

Before proving the theorem, we investigate some of the further properties of categories $\3$ and $\4$. 

\begin{lemma}\label{lem-T-fix-phase}
Let $\sigma$ be a stability condition on a triangulated category $\cT$. Assume that $T$ is an exact auto-equivalence of $\cT$ with $T^n=\mathrm{id}_{\cT}$ and $\sigma$ is $T$-invariant. Then for any $\sigma$-(semi)stable object $E\in \cT$, we have $T(E)$ is $\sigma$-(semi)stable with $\phi_{\sigma}(E)=\phi_{\sigma}(T(E))$.
\end{lemma}

\begin{proof}
As $\sigma$ is $T$-invariant, $E$ is $\sigma$-(semi)stable if and only if $T(E)$ is $\sigma$-(semi)stable. It remains to show $\phi_{\sigma}(E)=\phi_{\sigma}(T(E))$. Assume that $T\cdot \sigma=\sigma\cdot \wt{g}$ for $\wt{g}=(M, g)\in \GL$, where $M\in \mathrm{GL}^+(2, \mathbb{R})$ and $g\colon \mathbb{R}\to \mathbb{R}$ is an increasing function such that $g(x+1)=g(x)+1$. As $T^n=\mathrm{id}_{\cT}$, we have
\[\phi_{\sigma}(E)=\phi_{\sigma}(T^n(E))=g(\phi_{\sigma}(T^{n-1}(E)))=\cdots=g^n(\phi_{\sigma}(E)).\]
As $g$ is an increasing function, we obtain $\phi_{\sigma}(E)=\phi_{\sigma}(T(E))=g(\phi_{\sigma}(E))$.
\end{proof}

Therefore, by Lemma \ref{lem-T-fix-phase}, \ref{c-t1} and \ref{c-t2}, we have

\begin{itemize}
    \item for any $\sigma_1$-semistable object $E\in \3$, $\Phi^*\Phi_*E$ is also $\sigma_1$-semistable with
    \[\phi_1(E)=\phi_1(\Phi^*\Phi_*E),\]
    and

    \item for any $\sigma_2$-semistable object $F\in \4$, $\Phi_*\Phi^*F$ is also $\sigma_2$-semistable with 
\[\phi_2(F)=\phi_2(\Phi_*\Phi^*F).\]
\end{itemize}

The next lemma simplifies the statement that we need to prove in Theorem \ref{thm-main}.

\begin{lemma}\label{lem-simplify}
    Consider triangulated categories $\3$ and $\4$ satisfying the above conditions \ref{c1} to \ref{c-t2}. Assume that

\begin{enumerate}
    \item if $E \in \3$ is $\sigma_1$-semistable, then $\Phi_*E$ is $\sigma_2$-semistable with $\phi_{1}(E)=\phi_2(\Phi_*E)$, and

    \item if $F \in \4$ is $\sigma_2$-semistable, then $\Phi^*F$ is $\sigma_1$-semistable with $\phi_{2}(F)=\phi_1(\Phi^*F)$.
\end{enumerate}
Then Theorem \ref{thm-main} holds.
\end{lemma}

\begin{proof}
Given an object $E \in \3$ such that $\Phi_*E$ is $\sigma_2$-semistable. If $A\to E\to B$ is an exact triangle such that $\phi^-_1(A)>\phi^+_1(B)$, then we have $\phi^-_2(\Phi_*A)>\phi^+_2(\Phi_*B)$ by (1), which contradicts the semistability of $\Phi_*E$ as we have an exact triangle $\Phi_*A\to \Phi_*E\to \Phi_*B$. Therefore, the semistability of $\Phi_*E$ implies the semistability of $E$ and they have the same phase by (1). Similarly, the semistability of $\Phi^*F$ implies the semistability of $F$ for any object $F\in \4$ and they have the same phase by (2). This proves the semistability part of Theorem \ref{thm-main}.

Now we assume that $E \in \3$ is $\sigma_1$-semistable such that $\Phi_*E$ is $\sigma_2$-stable. Then by (1), it is clear that $E$ is $\sigma_1$-stable. Indeed, if $E$ is strictly $\sigma_1$-semistable, then we can find a $\sigma_1$-stable object $E'$ with $\Hom(E', E)\neq 0$ and $\phi_1(E)=\phi_1(E')$. Applying $\Hom(-,E)$ to the triangle \eqref{exact-3} associated with $E'$ and using $\phi_1(E')=\phi_1(T_1(E'))$, we get
\[\Hom(\Phi^*\Phi_*(E'),E)=\Hom(\Phi_*(E'),\Phi_*(E))\neq 0,\]
which contradicts the stability of $\Phi_*E$ as we have $\phi_2(\Phi_*(E'))=\phi_1(E')=\phi_1(E)=\phi_2(\Phi_*(E))$ by the assumption (1). Now if there is a $\sigma_1$-stable object $F$ such that $\Phi^*F\cong E$, then by \eqref{exact-4} and Lemma \ref{lem-T-fix-phase}, we know that $\Phi_*\Phi^*(F)\cong \Phi_*(E)$ is strictly $\sigma_2$-semistable, a contradiction. 

Next, assume that $E \in \3$ is $\sigma_1$-semistable such that $E$ is $\sigma_1$-stable and for any $\sigma_2$-stable object $F$ we have $\Phi^*F\neq E$, we need to prove the $\sigma_2$-stability of $\Phi_*E$. Up to shift, we can assume that $\Phi_*E\in \cA_2$. If $\Phi_*E$ is strictly $\sigma_2$-semistable, let $\{F_k\}_{k\in I}$ be the set of Jordan--H\"older factors of $\Phi_*E$, where $I$ is a finite index set. By \eqref{exact-3} and Lemma \ref{lem-T-fix-phase}, $\Phi^*\Phi_*E$ is strictly $\sigma_1$-semistable and the Jordan--H\"older factors are $E$ and $T_1(E)$. As $\Phi^*F_k$ is $\sigma_1$-semistable with $\phi_1(\Phi^*F_k)=\phi_1(E)=\phi_2(\Phi_*E)$ for each $k\in I$ by (2), there exists $k\in I$ such that $\Phi^*F_k\cong E$ by the stability of $E$ and the uniqueness of Jordan--H\"older factors, which makes a contradiction. This proves Theorem \ref{thm-main}(1). The proof of Theorem \ref{thm-main}(2) is similar.
\end{proof}

Our approach to Theorem \ref{thm-main} is by induction on $\ext^1$. Therefore, the first step is to control $\ext^1$ in several situations.

\begin{lemma}[{Weak Mukai Lemma}]\label{lem-mukai}
Let $A\to E\to B$ be an exact triangle in a triangulated category $\cT$. If $\Hom(A,B)=\Hom(B,A[2])=0$, then 
\[\ext^1(A,A)+\ext^1(B,B)\leq \ext^1(E,E).\]
\end{lemma}
\begin{proof}
The result follows from the same argument as in \cite[Lemma 2.5]{bayerK3}.
\end{proof}

\begin{lemma}\label{lem-lower bound on ext}
Consider triangulated categories $\3$ and $\4$ and stability conditions $\sigma_1$ and $\sigma_2$ satisfying the above conditions.

\begin{enumerate}
    \item Any non-zero object $E\in \mathcal{K}_i$ satisfies $\ext^1(E,E)\geq 2\,i$ for each $i=1, 2$. 

    \item The heart $\cA_i$ has homological dimension $2$, i.e.~$\Hom(A,B[k])=0$ for any $k>2$ and any two objects $A,B\in \cA_i$.
\end{enumerate}

\end{lemma}

\begin{proof}
As $\sigma_i$ is Serre-invariant for each $i=1, 2$, part (2) follows from \cite[Proposition 3.4]{FeyzbakhshPertusi2021stab}. Now by (2), \ref{c1}, and \ref{c2}, every $\sigma_i$-stable object $E\in \cK_i$ satisfies $\ext^1(E,E)\geq 2i$. Then part (1) follows from applying \cite[Proposition 3.4(b)]{FeyzbakhshPertusi2021stab}.
\end{proof}

\begin{lemma} \label{lem-hn ext}
For each $i=1, 2$, we have the following: 
\begin{enumerate}
    \item Any object $E \in \mathcal{K}_i$ satisfies
    \[\sum_k \ext^1(\cH^k_{\cA_i}(E), \cH^k_{\cA_i}(E))\leq \ext^1(E,E).\]
In particular, if $E$ is not in $\cA_i[m]$ for any $m \in \ZZ$, then $\ext^1(\cH^k_{\cA_i}(E), \cH^k_{\cA_i}(E))\leq \ext^1(E,E)-2\,i$ for each $k$.
    \item Let $\{E_j\}_{j\in I}$ be the set of the Harder--Narasimhan factors of $E$, then
    \[\sum_{j\in I} \ext^1(E_j, E_j)\leq \ext^1(E,E).\]
In particular, if $E$ is not $\sigma_i$-semistable, then $\ext^1(E_j, E_j)\leq \ext^1(E,E)-2\,i$ for each $j\in I$.
    \item If $E$ is strictly $\sigma_i$-semistable, then each of the Jordan--H\"older factors $A$ of $E$ satisfies
    \[\ext^1(A, A)<\ext^1(E,E).\]
\end{enumerate}
\end{lemma}
\begin{proof}

For part (1), we can assume that $E$ has at least two non-zero cohomology objects with respect to $\cA_i$. Let $k_0$ be the minimum value such that $\cH^{k_0}_{\cA_i}(E) \neq 0$. This gives rise to an exact triangle $\cH_{\cA_i}^{k_0}(E)[-k_0] \to E \to E'$, where $E' \in \mathcal{P}_{\sigma_i}(-\infty, -k_0]$. Consequently, according to \ref{c-stability}, we can deduce that $T_1(E') \in \mathcal{P}_{\sigma_i}(-\infty, -k_0]$. As $S_{\3} = T_1\circ [2]$, the conditions of Lemma \ref{lem-mukai} are satisfied for this exact triangle, yielding:
\begin{equation*}
\ext^1(E', E') + \ext^1(\cH^{k_0}_{\cA_i}(E), \cH^{k_0}_{\cA_i}(E)) \leq \ext^1(E, E).
\end{equation*}
Thus, the first claim in part (1) follows by induction on the number of non-zero cohomologies, and the second claim in part (1) follows by combining it with Lemma \ref{lem-lower bound on ext}. Similarly, Part (2) can be proved via induction on the length of the Harder--Narasimhan filtration. 

Finally, we prove (3). If all Jordan--H\"older factors of $E$ have the same class $v$, then for any Jordan--H\"older factor $A$ of $E$, we have $\chi(A,A)\leq 2-\ext^1(A,A)$ and $[E]=m[A]$ for an integer $m\geq 2$. Hence, $\ext^1(E,E)=\hom(E,E)+\ext^2(E,E)-\chi(E,E)\geq 1-m^2\chi(A,A)$. Then
\[\ext^1(A,A)\leq 2-\chi(A,A)< 1-m^2\chi(A,A)\]
as $m^2\geq 4$ and $\chi(A,A)\leq -1$. If $E$ has at least two non-isomorphic Jordan--H\"older factors, then we can find an exact triangle $E_1\to E\to E_2$ such that all Jordan--H\"older factors of $E_1$ are isomorphic with $\phi_{\sigma_i}(E_1)=\phi_{\sigma_i}(E_2)$ and $\Hom(E_1, E_2)=0$. 

When $i=2$, we have $\Hom(E_2, E_1[2])=0$, which gives $\ext^1(E_1, E_1)+\ext^1(E_2, E_2)\leq \ext^1(E,E)$ by Lemma \ref{lem-mukai}. By the first case, we know that (3) holds for $E_1$, hence, we only need to prove (3) for $E_2$. But as the number of Jordan--H\"older factors of $E_2$ is less than $E$ and $\ext^1(E_2, E_2)<\ext^1(E, E)$ by Lemma \ref{lem-mukai} and Lemma \ref{lem-lower bound on ext}, the statement for $E_2$ follows from the induction on the number of Jordan--H\"older factors.

Now we assume that $i=1$. If $\Hom(E_2, E_1[2])=\Hom(E_1, T_1(E_2))=0$, then the argument is the same as above. If $\Hom(E_1, T_1(E_2))\neq 0$, let $A$ be the Jordan--H\"older factor of $E_1$, then $E_2$ has $T_1(A)$ as a subobject and a Jordan--H\"older factor. Then we have an exact triangle $E_3\to E_2\to E_4$ such that all Jordan--H\"older factors of $E_3$ are isomorphic to $T_1(A)$ with $\phi_{1}(E_3)=\phi_{1}(E_4)$ and $\Hom(E_3, E_4)=0$. Let $E_5:=\cone(E\to E_4)[-1]$. Then $E_5$ is an extension of $E_1$ and $E_3$, hence, every Jordan--H\"older factor of $E_5$ is isomorphic to $A$ or $T_1(A)$. Moreover, $\Hom(E_5, E_4)=0$ and $\Hom(E_4, E_5[2])=\Hom(E_5,T_1(E_4))=0$ since $A$ and $T_1(A)$ are not Jordan--H\"older factors of $E_4$ by construction. Then by Lemma \ref{lem-mukai}, we get $\ext^1(E_5, E_5)+\ext^1(E_4, E_4)\leq \ext^1(E, E)$ and the remaining argument is similar to the $i=2$ case.
\end{proof}

Using exact triangles \eqref{exact-3} and \eqref{exact-4}, we are able to bound $\ext^1$ after acting by $\Phi_*$ and $\Phi^*$. 

\begin{lemma}\label{lem-rhom-pushfor}
Let $E \in \3$ be a $\sigma_1$-stable object with $\Ext^1(E, E) = \CC^n$. If $T_1(E) \neq E$, then 
\begin{equation*}
    \RHom(\Phi_*E, \Phi_*E) = \CC \oplus \CC^{2n}[-1] \oplus \CC[-2]. 
\end{equation*}
Otherwise, either 

\begin{enumerate}[label=(\roman*)]
    \item $\RHom(\Phi_*E, \Phi_*E) = \CC \oplus \CC^{2n-2}[-1] \oplus \CC[-2]$, or

    \item $ \RHom(\Phi_*E, \Phi_*E) = \CC^2 \oplus \CC^{2n}[-1] \oplus \CC^2[-2]$.
\end{enumerate}
Moreover, case (ii) happens if and only if $T_1(E)=E$ and the triangle \eqref{exact-3} is splitting.
\end{lemma}
\begin{proof}
    We know $\Hom(E, E[2]) = \Hom(E, T_1(E))^{\vee}$, which is $\CC$ if $T_1(E) \cong E$, and zero otherwise. Hence, $\chi(\Phi_*E, \Phi_*E) = 2\chi(E, E)$ is equal to $2(n-1)$ if $T_1(E) \neq E$, and $2(n-2)$ otherwise. Thus, the claim follows from applying $\Hom(-, E)$ to the exact triangle \eqref{exact-3} as $\RHom(\Phi_*E, \Phi_*E) = \RHom(\Phi^*\Phi_*E, E)$. Note that $\Hom(\Phi_*E, \Phi_*E[2]) = \Hom(\Phi_*E, \Phi_*E)^{\vee}$ as $\4$ is $2$-Calabi--Yau. 
\end{proof}

Applying a similar argument and \ref{c-t2} imply the following: 
\begin{lemma}\label{lem-rhom-back}
Let $E \in \4$ be a $\sigma_2$-stable object with $\Ext^1(E, E) = \CC^{2n}$, then either
\begin{enumerate}[label=(\roman*)]
    \item $\RHom(\Phi^*E, \Phi^*E) = \CC \oplus \CC^{4n-3+\delta}[-1] \oplus \CC^{\delta}[-2]$, or

    \item $\RHom(\Phi^*E, \Phi^*E) = \CC^2 \oplus \CC^{4n-1+\delta}[-1] \oplus \CC^{1+\delta}[-2]$
\end{enumerate}
for $\delta = 0$ or $1$. Moreover, case (ii) happens if and only if $E\cong T_2(E)$ and the triangle \eqref{exact-4} is splitting. 
\end{lemma}

In the following, for $i=1, 2$, let $A_i$ denote any $\sigma_1$-stable object of class $\lambda_i$ in the heart $\cA_1$. By \ref{c-stability}(d), $\Phi_*A_i \in \cA_2$ is $\sigma_2$-stable with $\phi_{\sigma_2}(\Phi_*A_i) = \phi_{\sigma_1}(A_i)$.

\begin{lemma}\label{lem-phase of push - back}
    Let $E \in \cA_1$ (resp.~$F \in \cA_2$) be a semistable object with respect to $\sigma_1$ (resp.~$\sigma_2$) of phase $\phi$, then  
    \begin{equation*}
        a< \phi^+_2(\Phi_*E) \ , \ \phi^-_2(\Phi_*E)<b \qquad \text{and} \qquad
        a< \phi^+_1(\Phi^*F) \ , \ \phi^-_1(\Phi^*F)<b
    \end{equation*}
    where
    \begin{enumerate}
        \item $a = -\frac{1}{2}$ and $b=1$ if $0<\phi<\frac{1}{2}$,
        \item $a = -\frac{1}{2}$ and $b= \frac{3}{2}$ if $\phi= \frac{1}{2}$,
        \item $a =0$ and $b= \frac{3}{2}$ if $\frac{1}{2}<\phi <1$, and
        \item $a =0$ and $b= 2$ if $\phi =1$. 
    \end{enumerate}
\end{lemma}
\begin{proof}
    Firstly, we consider a $\sigma_1$-semistable object $E \in \cA_1$. If $0< \phi_1(E) = \phi <\frac{1}{2}$, we have
\begin{enumerate}[label=(\roman*)]
    \item $\chi(\lambda_1, E)=\chi(E, \lambda_1)=\Re[Z_1(E)] >0$, and

    \item $\chi(\lambda_2, E)=\chi(E, \lambda_2)=-\Im[Z_1(E)]  <0$.
\end{enumerate}
    As $\phi_{1}(A_1)=1\geq \phi_1(E)$, by the finiteness of Jordan--H\"older factors, there are only at most finitely many $\sigma_1$-stable objects $A_1\in \cA_1$ such that $\Hom(E, A_1[2])=\Hom(T_1(A_1), E)\neq 0$. Then by Lemma \ref{lem-lower bound on ext}(2), (i) gives $\Hom(E, A_1) \neq 0$ for every but finitely many $\sigma_1$-stable object $A_1\in \cA_1$ with $[A_1]=\lambda_1$. Thus applying $\Hom(-, A_1)$ to \eqref{exact-3} implies $\Hom(\Phi^*\Phi_*E, A_1) = \Hom(\Phi_*E, \Phi_*A_1) \neq 0 $. By \ref{c-stability}(c), (d) and the finiteness of Jordan--H\"older factors, we obtain $\phi^-_2(\Phi_*E) < 1$. At the same time, (ii) implies that $\Hom(A_2, E[1]) \neq 0.$ Since $\Hom(T_1(A_2), E) = 0$ by $\phi<\frac{1}{2}$, applying $\Hom(-, E)$ to \eqref{exact-3}, we get $\Hom(\Phi^*\Phi_*A_2, E[1]) = \Hom(\Phi_*A_2, \Phi_*E[1]) \neq 0$, which gives $-\frac{1}{2} < \phi^+_2(\Phi_*E)$. 

If $\frac{1}{2}< \phi_1(E) =\phi <1$, we have
\begin{enumerate}[label=(\roman*)]
    \item $\chi(\lambda_1, E)=\chi(E, \lambda_1)=\Re[Z_1(E)] <0$, and

    \item $\chi(\lambda_2, E)=\chi(E, \lambda_2)=-\Im[Z_1(E)]  <0$.
\end{enumerate}
Then a similar argument as above gives $0< \phi_2^+(\Phi_*E)$ and $\phi_2^-(\Phi_*E) < \frac{3}{2}$. If $\phi_1(E) = \frac{1}{2}$, then we have $\chi(\lambda_2, E) =\chi(E, \lambda_2) < 0$. Thus $\Hom(E, A_2[1]) \neq 0$ for every $A_2$. But from the finiteness of Jordan--H\"older factors, $\Hom(T_1(E), A_2) = 0$ for every but finitely many $A_2$, so $\Hom(\Phi^*\Phi_*E, A_2[1]) \neq 0$ for infinitely many $A_2$, which gives $\phi_2^-(\Phi_*E) < \frac{3}{2}$. Also we have $\Hom(A_2, E[1]) \neq 0$ and $\Hom(T_1(A_2), E)=0$ for infinitely many $A_2$, which implies $-\frac{1}{2} < \phi_2^+(\Phi_*E)$. If $\phi_1(E) =1$, then $\chi(\lambda_1, E) = \chi(E, \lambda_1) <0$. Then a similar argument as above gives $\phi_2^-(\Phi_*E) < 2$ and $0< \phi_2^+(\Phi_*E)$. 

The second part for $\phi_1^{\pm}(\Phi^*F)$ where $F \in \cA_2$ is $\sigma_2$-semistable follows via the same argument as above. We only need to use $\Phi_*A_i$ for $i=1, 2$ and the exact triangle \eqref{exact-4} instead.  

\end{proof}

\subsection{Induction argument}

Now we are prepared to prove the main theorem of this section.

\begin{proof}[Proof of Theorem \ref{thm-main}]
    By Lemma \ref{lem-simplify}, we only need to show the following two claims for any $n>0$:

\textbf{$\bA_n$:} If $E \in \3$ is $\sigma_1$-stable with $\ext^1(E, E) \leq n$, then $\Phi_*E$ is $\sigma_2$-semistable with the same phase $\phi_{2}(\Phi_*E) = \phi_{1}(E)$. 

\textbf{$\bB_n$:} If $F \in \4$ is $\sigma_2$-stable with $\ext^1(F, F) \leq n$, then $\Phi^*F$ is $\sigma_1$-semistable with the same phase $\phi_{1}(\Phi^*F) = \phi_{2}(F)$. 

We use a double induction argument to prove the above two statements. According to \ref{c-stability}(d), the assertions $\textbf{A}_n$ are automatically satisfied when $n = 3$, as any $\sigma_1$-stable object $A$ with $\ext^1(A,A)\leq 3$ is of class $\lambda_i$. Then the extension to arbitrary $n$ follows from Lemma \ref{part-1} and Lemma \ref{part-2}. 
\end{proof}

It remains to prove Lemma \ref{part-1} and Lemma \ref{part-2}, which is the goal of the rest of this section. The proof of each lemma is divided into two steps. For example, in the first step of Lemma \ref{part-1}, we prove $\Phi_*(E)\in\cA_2$ for any $\sigma_1$-stable object $E\in \cA_1$. Then in the second step, the $\sigma_2$-semistability of $\Phi_*(E)$ will be deduced from a slope-comparison argument.

\begin{remark}\label{rmk-A-B}
Using Lemma \ref{lem-hn ext}(2) and (3), we have the following equivalent forms of $\bA_n$ and $\bB_n$ by reducing to the case when $E$ and $F$ are stable:

\textbf{$\bA_n$:} If $E \in \3$ is a non-zero object with $\ext^1(E, E) \leq n$, then $\Phi_*$ preserves the Harder--Narasimhan filtration of $E$.

\textbf{$\bB_n$:} If $F \in \4$ is a non-zero object with $\ext^1(F, F) \leq n$, then $\Phi^*$ preserves the Harder--Narasimhan filtration of $F$.

Here, \emph{preserving the Harder--Narasimhan filtration} through a functor entails that the Harder--Narasimhan filtration of the image corresponds to the image of the Harder--Narasimhan filtration of the original object, and the components retain the same phases. 
\end{remark}

\begin{lemma}\label{part-1}
For any $n > 1$, the validity of $\bB_{n-1}$ implies $\bA_n$. 
\end{lemma} 
\begin{proof}
Take a $\sigma_1$-stable object $E$ with $\ext^1(E, E) \leq n$. Up to shift, we can assume that $E\in \cA_1$. Firstly, we claim that we can furthermore assume $0 < \phi_1(E) \leq \frac{1}{2}$. If $\frac{1}{2}<\phi_1(E) \leq 1$, then we can find $\wt{g}\in \GL$ such that $\phi_{\sigma_1\cdot \wt{g}}(A_1)=\phi_{\sigma_2\cdot \wt{g}}(\Phi_*A_1)=\frac{1}{2}$ and $\phi_{\sigma_1\cdot \wt{g}}(A_2[1])=\phi_{\sigma_2\cdot \wt{g}}(\Phi_*A_2[1])=1$, hence, $0 < \phi_{\sigma_1\cdot \wt{g}}(E) \leq \frac{1}{2}$. Therefore, if we replace $\sigma_i$ by $\sigma_i\cdot \wt{g}$ in \ref{c-stability}, $\lambda_2$ (resp.~$\Lambda_2$) by $-\lambda_2$ (resp.~$-\Lambda_2$) in \ref{c1} (resp.~\ref{c2}), all conditions \ref{c1} to \ref{c-t2} will not change. Thus we only need to deal with the case $0 < \phi_1(E) \leq \frac{1}{2}$.

Write $[E]=a\lambda_1+b\lambda_2$ for integers $a,b\in \ZZ$. By our assumption on $\phi_1(E)$, we have  
$a\leq 0$ and $b>0$. We know $[\Phi_*E]=a\Lambda_1+b\Lambda_2$ and
\begin{equation}\label{central charge}
Z_1(E)=Z_2(\Phi_*E)=-a+b\mathfrak{i}=\chi(\lambda_1, E)-\chi(\lambda_2, E)\mathfrak{i}.    
\end{equation}

\bigskip
 
\textbf{Step 1.} The first step is to show that $\Phi_*E \in \cA_2$. Suppose $\Phi_*E$ has non-zero cohomology objects $E^{x_i} := \cH^{x_i}_{\cA_2}(\Phi_*E)[-x_i]$ for $0\leq i \leq m$, where
\begin{equation}\label{sec}
x_0<x_1<...<x_m.     
\end{equation}
If $x_0=x_m$, then $\Phi_*E \in \cA[-x_0]$. Then \eqref{central charge} shows that $x_0 \in 2\ZZ$, and thus, by Lemma \ref{lem-phase of push - back}, it must be zero as required. Assume, for a contradiction, that $x_0 \neq x_m$. We claim that there exists $l \in \{0, m\}$ such that 
\begin{equation}\label{ex-cond}
    \ext^1(E^{x_l}, E^{x_l})\leq \frac{1}{2}\ext^1(\Phi_*E, \Phi_*E)
\end{equation}
and $\Phi^*$ preserves the Harder--Narasimhan filtration of $E^{x_l}$. From Lemma \ref{lem-hn ext}(1), we know that \eqref{ex-cond} holds for at least one $l\in \{0,m\}$. If the inequality in \eqref{ex-cond} is strict for $l=0$ or $m$, then the claim follows from $\bB_{n-1}$. Thus, by Lemma \ref{lem-hn ext}(1), the only remaining case is when we have equality in \eqref{ex-cond} for both $l=0$ and $m$. In this case, $E^{x_0}$ and $E^{x_m}$ cannot both be $\sigma_2$-stable at the same time by Lemma \ref{lem-two-factor-stable-push}. Hence, at least one $E^{x_l}$ must be either strictly $\sigma_2$-semistable or not $\sigma_2$-semistable. Then by Lemma \ref{lem-hn ext}(2) and (3), any Jordan--H\"older factor $E^{x_l '}$ of any Harder--Narasimhan factor of $E^{x_l}$ satisfies $\ext^1(E^{x_l '},E^{x_l '})<\ext^1(E^{x_l}, E^{x_l})$ and the claim follows from applying $\bB_{n-1}$ to each $E^{x_l '}$. Therefore, we divide the proof into the following two situations, and we will rule out them case-by-case.

\bigskip

\textbf{Case I.} \emph{$\Phi^*$ preserves the Harder--Narasimhan filtration of $E^{x_0}$ and \eqref{ex-cond} holds for $l=0$.}

In this case, we get $\Phi^*(E^{x_0}[x_0]) \in \cA_1$. Let $E^{x_0}_{\max}$ denote the factor with the highest phase in the Harder--Narasimhan filtration of $E^{x_{0}}$. It is also the factor with the maximum phase in the Harder--Narasimhan filtration of $\Phi_*E$. Thus $\Hom(\Phi^*E^{x_0}_{\max}, E) =\Hom(E^{x_0}_{\max}, \Phi_*E) \neq 0$, so we have 
\begin{equation}\label{eq-case1-1}
    -x_0<\phi^-_{\sigma_1}(\Phi^*E^{x_0}) \leq \phi^+_{\sigma_1}(\Phi^*E^{x_0})=\phi_{\sigma_1}(\Phi^*E^{x_0}_{max}) \leq
    \phi_{\sigma_1}(E) \leq \frac{1}{2},
\end{equation}
which implies $-x_0<1$. Hence, combined with Lemma \ref{lem-phase of push - back}, we obtain $ x_0\in \{0,1\}$. Thus $\Phi_*E$ lies in an exact triangle $E_{> -1} \to \Phi_*E \to E_{\leq -1}$ where $E_{>-1} \in \mathcal{P}_{\sigma_2}(-1, 1]$ and $E_{\leq -1} \in \mathcal{P}_{\sigma_2}(-\infty, -1]$. We claim
\begin{equation}\label{f-claim}
    \Im[Z_2(E_{\leq -1})] = -\chi(\Lambda_2, E_{\leq -1}) >0. 
\end{equation}
If $x_0 =1$, then the claim is trivial as $\Im[Z_2(E)] >0$ and $E_{>-1}=E^{x_0}\in \cA_2[-1]$, so we may assume $x_0 = 0$. In this case, we define $E':=\cone(E^{x_0} \to \Phi_*E)$. As $E'$ is an extension of $E_{\leq -1}$ and $E^1$, we only need to show 
\begin{equation}\label{l-claim}
\Im[Z_2(E')] > 0.     
\end{equation}
As $x_0=0$, we know that $\Re[Z_2(E^{x_0})]\geq 0$ and
\begin{equation}\label{slope-1}
    \mu_2(E^{x_0})\leq \mu_2(\Phi_*E)=\mu_1(E)\leq 0.
\end{equation}
So if $\Re[Z_2(E^{x_0})] = 0$, then $\Re[Z_2(\Phi_*E)] = 0$ as well. Then we have 
\begin{align}\label{comparing im}
    2(\Im[Z_2(E^{x_0})])^2 = -\chi(E^{x_0}, E^{x_0}) \leq & -2+ \ext^1(E^{x_0}, E^{x_0}) \\
    < & -2 + \frac{1}{2}\ext^1(\Phi_*E, \Phi_*E) \leq -2 + \frac{1}{2} (4-\chi(\Phi_*E, \Phi_*E)) = (\Im[Z_2(E)])^2 \nonumber
\end{align}
and so $0 < \Im[Z_2(E^{x_0})] < \Im[Z_2(E)]$, which proves the claim \eqref{l-claim}. Hence, we may assume $$\Re[Z_2(E^{x_0})] >0.$$ Using \eqref{slope-1}, to prove \eqref{l-claim}, it suffices to show 
\begin{equation}\label{claim-1.1}
    \Re[Z_2(E')] \geq 0 
\end{equation}
as $Z_2(E') +Z_2(E^{x_0}) = Z_2(\Phi_*E)$.
To prove \eqref{claim-1.1}, we investigate homomorphisms to $A_1$. By \eqref{eq-case1-1}, 
\begin{equation}\label{eq-vanish-A1}
\Hom(\Phi^*E^{x_0}, A_1[k])=\Hom(\Phi^*\Phi_*E, A_1[k])=0,~k\notin\{0,1\}.
\end{equation}
Hence, applying $\Hom(-,\Phi_*A_1)$ to the exact triangle $E^{x_0} \to \Phi_*E \to E'$ gives
\[\Hom(\Phi^*E', A_1[k])=\Hom(E', \Phi_*A_1[k])=0\]
for any $k\notin \{0,1,2\}$. Thus, if $\Re[Z_2(E')]=\chi(\Lambda_1, E')<0$, we have $\Hom(\Phi_*A_1, E'[1])\neq 0$, which is not possible\footnote{Note that $A_1$ can be any $\sigma_1$-stable object of class $\lambda_1$.} as $E' \in \mathcal{P}_{\sigma_2}(-\infty , 0]$. Hence, $\Re[Z_2(E')] = \chi(\Lambda_1, E') \geq 0$ as claimed in \eqref{claim-1.1}. This ends the proof of \eqref{f-claim}.

Finally, we investigate morphisms to $A_2$. We have 
\[\Hom(E_{> -1}, \Phi_*A_2[k_1])= 0 =\Hom(\Phi^*\Phi_*E, A_2[k_2])\]
for $k_1 \notin\{-1, 0,1\}$ and $k_2 \notin \{0, 1\}$ for infinitely many $A_2$. Note that the vanishing for $k_1=2$ follows from the phase ordering \eqref{eq-case1-1}. Thus applying $\Hom(\Phi_*A_2, -)$ to $E_{>-1}\to \Phi_*E\to E_{\leq -1}$, the claim  \eqref{f-claim} gives $\Hom(\Phi_*A_2, E_{\leq -1}[1]) \neq 0$, which is not possible as $E_{\leq -1} \in \mathcal{P}_{\sigma_2}(-\infty, -1]$.

\bigskip

\textbf{Case II.} \emph{$\Phi^*$ preserves the Harder--Narasimhan filtration of $E^{x_m}$ and \eqref{ex-cond} holds for $l=m$.}

By Lemma \ref{lem-phase of push - back}, either

\begin{enumerate}[label=(\roman*)]
    \item $\phi^-_{\sigma_2}(\Phi_*E) <1$, or

    \item $\phi^-_{\sigma_2}(\Phi_*E) <\frac{3}{2}$ and $[E]$ is a multiple of $\lambda_2$.
\end{enumerate}

Let $E^{x_m}_{\max}$ denote the factor with the smallest phase in the Harder--Narasimhan filtration of $E^{x_{m}}$. It is also the factor with the minimal phase in the Harder--Narasimhan filtration of $\Phi_*E$. From \ref{c-adjoints}, we know $\Hom(T_1(E), \Phi^*E^{x_m}_{min})=\Hom(\Phi_*E, E^{x_m}_{min})\neq 0$, so
\begin{equation}\label{eq-case2-11}
    0<\phi_{\sigma_1}(E)\leq \phi_{\sigma_1}(\Phi^*E^{x_m}_{min})=\phi^-_{\sigma_1}(\Phi^*E^{x_m})\leq -x_m+1
\end{equation}
which gives $x_m\leq 0$. Thus $x_m =0$ in case (i) and $x_m \in \{0, -1\}$ in case (ii). We define $$E':=\cone(\Phi_*E\to E^{x_m})[-1].$$

If $x_m=0$, we have $E' \in \mathcal{P}_{\sigma_2}(1, +\infty)$, so $\Hom^{\leq 0}(E', \Phi_*A_i) = 0$ for $i=1, 2$. Since
\[\Hom(\Phi^*E^{x_m}, A_i[k_1])= 0 = \Hom(\Phi^*\Phi_*E, A_i[k_2])\]
for any $k_1\notin \{0,1, 2\}$ and $k_2 \notin \{0, 1\}$, applying $\Hom(-, \Phi_*A_i)$ to $E'\to \Phi_*E\to E^{x_m}$, we get $$\Hom(E', \Phi_*A_i[k]) = 0$$ for $i=1, 2$ and $k\neq 1$. Thus $\Re[Z_2(E')] \leq 0$ and $\Im[Z_2(E')] \geq 0$. Then $\Re[Z_2(E^{x_m})] \geq 0$ as $\Re[Z_2(\Phi_*E)] \geq 0$. So if $[E]$ is not a multiple of $[\lambda_2]$, then the slope $\mu_2(E')$ is bigger than $\mu_2(\Phi_*E)$, which is not possible by \eqref{eq-case2-11}. Thus the only possibility is when both $[\Phi_*E]$ and $[E^{x_m}]$ are multiples of $\Lambda_2$. But then the same argument as in \eqref{comparing im} gives $\Im[Z_2(E^{x_m})] < \Im[Z_2(\Phi_*E)]$ which implies $\Im[Z_2(E')] > 0$. As $E'$ lies in the triangle $E'_{>2} \to E' \to \cH^{-1}_{\cA_2}(E')[1]$ such that $E'_{>2}\in \cP_{\sigma_2}(2,+\infty)$, we see $-\chi(E'_{>2}, \Lambda_2)=2\Im[Z_2(E'_{>2})] > 0$. Since we have already seen that $\Hom(E', \Phi_*A_2[k]) = 0$ for $k\neq 1$ and $\Hom(\cH^{-1}_{\cA_2}(E')[1], \Phi_*A_2[k])=0$ for $k\notin \{1,2,3\}$, applying $\Hom(-,\Phi_*A_2)$ to $E'_{>2}\to E'\to \cH^{-1}_{\cA_2}(E')[1]$ and using $\chi(E'_{>2}, \Lambda_2)<0$ imply $\ext^1(E'_{>2}, \Phi_*A_2) \neq 0$, a contradiction.

If $x_m=-1$, there is an exact sequence in $\cA_1$:
\[0\to \Phi^*E^{x_m}[-1]\to \Phi^*E'\to \Phi^*\Phi_*E\to 0.\]
Hence, $\chi(\Phi^*E', \lambda_2)<0$ and so $\Hom(\Phi^*E', A_2[1])=\Hom(E', \Phi_*A_2[1])\neq 0$, contradicting $E' \in \mathcal{P}_{\sigma_2}(2, +\infty)$. 
\bigskip 

\textbf{Step 2.} Now we know $\Phi_*E \in \cA_2$ and it remains to show that $\Phi_*E$ is $\sigma_2$-semistable. Assume not, and let $E^{\max}$ (resp.~$E^{\min}$) be the factor with the maximum (resp.~minimum) phase in the Harder--Narasimhan filtration of $\Phi_*E$. Then we have $E^{\max}, E^{\min}\in \cA_2$ and 
\begin{equation}\label{eq-slope}
\mu_2(E^{\max})>\mu_2(\Phi_*E)>\mu_2(E^{\min}).
\end{equation}
Using Lemma \ref{lem-mukai}, one can show via the same argument as in Step 1. that there is $A \in \{E^{\min}, E^{\max}\}$ such that $\Phi^*A$ is $\sigma_1$-semistable of the same phase $\phi_{\sigma_1}(\Phi^*A) = \phi_{\sigma_2}(A)$. However, both $$\Hom(E^{\max}, \Phi_*E) = \Hom(\Phi^*E^{\max}, E) \neq 0$$ and $$\Hom(\Phi_*E, E^{\min}) = \Hom(E, T_1(\Phi^*E^{\min})) \neq 0$$ contradict \eqref{eq-slope}.      
\end{proof}

To roll out the maximal dimensional scenario in Lemma \ref{lem-two-factor-stable-pull}, we utilize the following standard spectral sequence, see e.g.~\cite[Lemma 2.27]{pirozhkov2020admissible}.
\begin{lemma}\label{spec-seq}
Let $X$ be a smooth algebraic variety, and let $A \to B\to C$ be an exact triangle in $\D^b(X)$. Then there exists a spectral sequence which degenerates at $E_3$ and converges to $\Ext^*_X(C, C)$, with $E_1$-page
\[E^{p,q}_1= \left\{
	\begin{aligned}
	&\mathrm{Ext}_X^q(B,A),  ~~~~~~~~~ &p=-1 \\
	&\mathrm{Ext}_X^q(A,A)\oplus \mathrm{Ext}_X^q(B,B),  ~~~~~~~~~   ~ &p=0 \\
	&\mathrm{Ext}_X^q(A, B),  ~~~~~~~~~ &p=1 \\
	&0, ~~~~~~~~~  &\text{otherwise}
	\end{aligned}
	\right.\]
with differentials $d_r^{p,q}\colon E^{p,q}_1\to E^{p+r, q-r+1}_1$. Moreover, $d_1^{p,q}$ is given by composition with the morphism $A\to B$.
\end{lemma}

\begin{lemma}\label{lem-two-factor-stable-push}
Let $E\in \3$ be a $\sigma_1$-stable object with $\ext^1(E,E)=n$. If $\Phi_*E$ is not $\sigma_2$-semistable with the Harder--Narasimhan filtration
\begin{equation}\label{hn}
    E_1\to \Phi_*E\to E_2
\end{equation}
such that $E_1$ and $E_2$ are $\sigma_2$-stable, then there exists $i\in \{1,2\}$ such that $\ext^1(E_i, E_i)<n$.
\end{lemma}

\begin{proof}
As $\phi_{\sigma_2}(E_1)>\phi_{\sigma_2}(E_2)$, we see that
 \begin{equation}\label{eq-ext=0-two-factor}
   \Hom(E_1, E_2[i])=\Hom(E_2, E_1[2-i])=0  
 \end{equation}
for $i\leq 0$. By Lemma \ref{lem-mukai} and \ref{lem-rhom-pushfor},  we have $$\ext^1(E_1, E_1)+\ext^1(E_2, E_2)\leq \ext^1(\Phi_*E, \Phi_*E)\leq 2n.$$ Thus, it is enough to show that the case $\ext^1(E_1, E_1)=\ext^1(E_2, E_2)=n$ cannot happen. Suppose this happens. Applying Lemma \ref{spec-seq} to the exact triangle $E_2[-1]\to E_1\to \Phi_*E$, we get a spectral sequence with the first page
 $$ E^{p,q}_1= \begin{array}{cc|cc}
	0 & \Ext^3(E_1,E_2) & 0 & 0\\
	0 & \Ext^2(E_1,E_2) & 0 & 0\\
	0 & \Ext^1(E_1,E_2) & \Ext^2(E_1,E_1)\oplus \Ext^2(E_2,E_2) & 0\\
	0 & 0 & \Ext^1(E_1,E_1)\oplus \Ext^1(E_2,E_2) & 0 \\
	0 & 0 & \Hom(E_1,E_1)\oplus \Hom(E_2,E_2) & \Ext^1(E_2,E_1) \\ \hline
	0 & 0 &  0 & \Hom(E_2,E_1)
	\end{array}$$
 which degenerates at $E_3$ and converges to $\Ext^*(\Phi_*E, \Phi_*E)$. Note that the differential $$d_1^{0,0}\colon \Hom(E_1, E_1)\oplus \Hom(E_2, E_2)=\CC^{2}\to \Ext^1(E_2, E_1)$$ is given by compositions, hence has a non-zero kernel. Moreover, as $$E^{0,1}_1=E^{0,1}_{\infty}=\Ext^1(\Phi_*E, \Phi_*E)=\CC^{2n} = \Ext^1(E_1,E_1)\oplus \Ext^1(E_2,E_2) ,$$ we know that $d_1^{0,0}$ is surjective. By Lemma \ref{lem-rhom-pushfor}, we have the following two cases:

\medskip

\textbf{Case I.} \emph{$\RHom(\Phi_*E, \Phi_*E)=\CC\oplus \CC^{2n}[-1]\oplus \CC[-2]$.} In this case, we have
\begin{align}\label{chi}
    2-2n = \chi(\Phi_*E, \Phi_*E) = & \chi(E_1, E_1) +\chi(E_2, E_2) +2 \chi(E_1, E_2) \\
    = & (2-n) + (2 -n) + 2\chi(E_1, E_2)  \nonumber
\end{align}
and so $\chi(E_1, E_2) =-1$, which leads to a contradiction since the Euler pairing in $\4$ is an even number.

\medskip

\textbf{Case II.} \emph{$\RHom(\Phi_*E, \Phi_*E)=\CC^2\oplus \CC^{2n}[-1]\oplus \CC^2[-2]$.} In this case, $T_1(E) \cong E$ and $\Phi^*\Phi_*E \cong E \oplus E$ by Lemma \ref{lem-rhom-pushfor}. Since $d_1^{0, 0}$ is surjective and has a non-zero kernel, we must have either $\Ext^1(E_2, E_1) = 0$ or $\CC$. We first show that the Harder--Narasimhan sequence \eqref{hn} does not split so the former case cannot happen. Otherwise, $\Phi_*E \cong E_1 \oplus E_2$ and so $\Phi^*\Phi_*E \cong \Phi^*E_1 \oplus \Phi^*E_2$. Since $E$ is $\sigma_1$-stable, we get $\Phi^*E_1 \cong \Phi^*E_2 \cong E$ and so $\Phi_*\Phi^*E_1 \cong \Phi_*\Phi^*E_2$, which is not possible as $\phi_2(\Phi_*\Phi^*E_1) = \phi_2(E_1) > \phi_2(E_2) = \phi_2(\Phi_*\Phi^*E_2)$. Thus, we may assume 
\begin{equation}\label{assumption}
    \Ext^1(E_2, E_1) = \C = \Hom(E_2, E_1) 
\end{equation}
as $\chi(E_2, E_1) = 0$ via the same argument as in \eqref{chi}. Applying $\Phi^*$ to \eqref{hn} gives the exact triangle 
\begin{equation*}
    \Phi^*E_1 \xrightarrow{h} E\oplus E \xrightarrow{h'} \Phi^*E_2.
\end{equation*}
The next step is to show that $h$ and $h'$ are non-zero by showing that 
\begin{equation}\label{cl}
    \Hom(E, \Phi^*E_2) = \C = \Hom(E, \Phi^*E_1). 
\end{equation}
Since $\Hom(E_1[k], E_2) = 0$ for $k\geq 0$, taking $\Hom(-, E_2)$ from the sequence \eqref{hn} gives $$ \C = \Hom(\Phi_*E, E_2) = \Hom(E, \Phi^*E_2)$$ as the left-hand claim in \eqref{cl}. Similarly, by taking $\Hom(-, E_1)$ and using assumption \eqref{assumption}, we obtain the long exact sequence 
\begin{equation*}
    ... \to \Hom(E_1[1], E_1) = 0 \to \Hom(E_2, E_1) \cong \C \to \Hom(\Phi_*E, E_1) \to \Hom(E_1, E_1) \cong \C \to ...
\end{equation*}
Since the sequence \eqref{hn} does not split, we get $\Hom(\Phi_*E, E_1) = \C$. Then the right-hand claim in \eqref{cl} follows by adjunction in \ref{c-adjoints}. Similarly, we have $\Hom(\Phi^*E_2, E)=\Hom(E_2, \Phi_*E)=\CC$ as \eqref{hn} does not split.

Therefore, there is a map $s\colon E\oplus E\to E$ such that $s\circ h$ is non-zero. Since $E$ is $\sigma_1$-stable, we know that $\cone(s)[-1]\cong E$, and denote the natural map $\cone(s)[-1]\cong E\to E\oplus E$ by $s'$. Define $F \coloneqq \cone(s \circ h)[-1]$, then we have the following commutative diagram: 
\[\begin{tikzcd}
	F & E & {\Phi^*E_2} \\
	{\Phi^*E_1} & {E\oplus E} & {\Phi^*E_2} \\
	E & E & 0
	\arrow[no head, from=3-1, to=3-2]
	\arrow[from=2-2, to=3-2]
	\arrow["s\circ h", from=2-1, to=3-1]
	\arrow["h"', from=2-1, to=2-2]
 \arrow["h'"', from=2-2, to=2-3]
 \arrow["s'"', from=1-2, to=2-2]
 \arrow["s"', from=2-2, to=3-2]
	\arrow[shift right, no head, from=3-1, to=3-2]
	\arrow[from=3-2, to=3-3]
	\arrow[from=2-2, to=2-3]
	\arrow[from=2-3, to=3-3]
	\arrow[no head, from=1-3, to=2-3]
	\arrow[shift right, no head, from=1-3, to=2-3]
	\arrow[from=1-2, to=2-2]
	\arrow[from=1-1, to=2-1]
	\arrow[from=1-1, to=1-2]
	\arrow["h'\circ s'"', from=1-2, to=1-3]
\end{tikzcd}\]
At first, we show that $h'\circ s'\neq 0$. As $s\circ h\neq 0$, we know that the induced map $$\Hom(E, E)=\CC\to \Hom(\Phi^*E_1, E)$$ is non-zero, hence it is injective. Then from $\Hom(\Phi^*E_1, E[-1])=\Hom(E_1, \Phi_*E[-1])=0$, we get $\Hom(F, E[-1])=0$. Therefore, if we apply $\Hom(-, E)$ to the exact triangle in the first row, we see the induced map $\Hom(\Phi^*E_2, E)=\CC\to \Hom(E, E)$ is injective as well. This implies that $h'\circ s'$ is non-zero.

Now, taking $\Hom(E, -)$ from the first row and applying \eqref{cl} show that $\Hom(E, F) = 0$ because $\Hom(E, \Phi^*E_2[-1]) = \Hom(\Phi_*E, E_2[-1]) = 0$ as $\phi_2^{-}(\Phi_*E_2) = \phi_2(E_2)$. Thus, applying $\Hom(E, -)$ to the first column gives the isomorphism $\C = \Hom(E, \Phi^*E_1) \to \C = \Hom(E, E)$. This shows that the first column is splitting, so by \eqref{cl}, we have 
\begin{equation*}
    \Hom(E_1, \Phi_*\Phi^*E_2) = \Hom(\Phi^*E_1, \Phi^*E_2) \cong \Hom(E \oplus F , \Phi^*E_2) \neq 0,
\end{equation*}
 which is not possible as $\phi_2(E_1) > \phi_2(j_*j^*E_2)=\phi_2(E_2)$ and both are $\sigma_2$-semistable.
\end{proof}



\begin{lemma}\label{part-2}
    For any $n>1$, the validity of $\bA_{n-1}$ implies that $\bB_n$ holds.
\end{lemma} 
\begin{proof}
The proof closely resembles that of Lemma \ref{part-1}. Therefore, we will only provide an outline of the key steps. For any object $F \in \4$, we know $\ext^1(F, F) = -\chi(F, F) + 2\hom(F, F) \in 2\ZZ$. Thus we only need to prove the claim when $n= 2k$ for $k\in \ZZ$. Therefore, we take a $\sigma_2$-stable object $F$ with $\ext^1(F, F) =2k$. As in the proof of Lemma \ref{part-1}, we may assume that $F \in \cA_2$ and $0< \phi_2(F) \leq \frac{1}{2}$.

\bigskip

\textbf{Step 1.} The first step is to show that $\Phi^*F \in \cA_1$. If this is not true, we have non-zero cohomology objects $F^{x_i} := \cH^{x_i}_{\cA_1}(\Phi^*F)[-x_i]$ for $0 \leq i \leq m$, where $x_0 <x_1 < ...< x_m$. If $x_0 = x_m$, then since $Z_1(\Phi^*F) = 2Z_2(F)$, we have $x_0 \in 2 \ZZ$ and so $x_0 =0$ by Lemma \ref{lem-phase of push - back}. Thus, we may assume $x_0 < x_m$. Then $\textbf{A}_{2k-1}$ and the next Lemma \ref{lem-two-factor-stable-pull} show that there exists $l \in \{0, m\}$ such that $\ext^1(F^{x_l}, F^{x_l}) \leq 2k$ and $\Phi_*$ preserves the Harder--Narasimhan filtration of both $E^{x_{l}}$ and $T_1(E^{x_{l}})$ as in the first step of Lemma \ref{part-1}.

\bigskip

\textbf{Case I.} Firstly, we assume that $l= 0$.  Let $F^{x_0}_{\max}$ denote the factor with the highest phase in the Harder--Narasimhan filtration of $F^{x_{0}}$. Since $\Hom(\Phi_*T_1(F^{x_0}_{\max}), F) =\Hom(F^{x_0}_{\max}, \Phi^*F) \neq 0$, we have 
\begin{equation}\label{eq-case1-11}
    -x_0<\phi^{-}_{\sigma_1}(F^{x_0}) \leq \phi^{+}_{\sigma_1}(F^{x_0}) = \phi^+_{\sigma_2}(\Phi_*T_1(F^{x_0}_{\max})) \leq 
    \phi_{\sigma_2}(F) \leq \frac{1}{2},
\end{equation}
which implies $-x_0<1$. Hence, combined with Lemma \ref{lem-phase of push - back}, we obtain $ x_0\in \{0,1\}$. Thus $\Phi^*F$ lies in an exact triangle $F_{> -1} \to \Phi^*F \to F_{\leq -1}$ where $F_{>-1} \in \mathcal{P}_{\sigma_1}(-1, \frac{1}{2}]$ and $F_{\leq -1} \in \mathcal{P}_{\sigma_1}(-\infty, -1]$. We claim
\begin{equation}\label{f-claim-1}
    \Im[Z_1(F_{\leq -1})] = -\chi(\lambda_2, F_{\leq -1}) >0. 
\end{equation}
If $x_0 =1$, then the claim is trivial as $\Im[Z_1(\Phi^*F)]>0$, so we may assume $x_0 = 0$. We define $$F':=\cone(F^{x_0} \to \Phi^*F).$$ Because $F'$ is an extension of $F_{\leq -1}$ and $F^1=\cH^1_{\cA_1}(\Phi^*F)[-1]$, we only need to show $\Im[Z_1(F')] > 0$. Since the slope of $Z_1(F^{x_0})$ is smaller than or equal to the slope of $Z_1(\Phi^*F)$ by \eqref{eq-case1-11}, we can assume that $\Re[Z_1(F^{x_0})]>0$ via the same argument as \eqref{slope-1} and \eqref{comparing im}. Then as \eqref{claim-1.1}, it suffices to prove 
\begin{equation}\label{claim-1.11}
    \Re[Z_1(F')] \geq 0.  
\end{equation}
If \eqref{claim-1.11} does not hold, then $\chi(\lambda_1, F')<0$. On the other hand, applying $\Hom(A_1, -)$ to the exact triangle $F^{x_0}\to \Phi^*F \to F'$ yields $\Hom(A_1, F'[k]) = 0$ for $k \notin \{0, 1, 2\}$ for infinitely many $A_1$. Consequently, $\Hom(A_1, F'[1]) \neq 0$ for infinitely many $A_1$, which is not possible since $F' \in \mathcal{P}_{\sigma_1}(-\infty, 0]$. Therefore, \eqref{claim-1.11} holds, and so \eqref{f-claim-1} holds. 

Next, we investigate morphisms from $A_2$ by taking $\Hom(A_2, -)$ from the exact triangle $$F_{>-1}\to \Phi^*F \to F_{\leq -1}.$$ We have 
\[\Hom(A_2, F_{>-1}[1+k])= 0 =\Hom(A_2, \Phi^*F[k]) = \Hom(\Phi_*T_1(A_2), F[k])\]
for $k \notin\{0,1, 2\}$ and infinitely many $A_2$, where the first vanishing follows from $F_{>-1} \in \mathcal{P}_{\sigma_1}(-1, \frac{1}{2}]$. Thus $\Hom(A_2, F_{\leq -1}[k])=0$ for $k\notin \{0,1,2\}$. Then  \eqref{f-claim-1} gives $\Hom(A_2, F_{\leq -1}[1]) \neq 0$, which contradicts $F_{\leq -1} \in \mathcal{P}_{\sigma_1}(-\infty, -1]$.

\bigskip

\textbf{Case II.} Now assume $l=m$. Let $F^{x_m}_{\min}$ be the object with the minimum phase in the Harder--Narasimhan filtration of $F^{x_m}$. We have $\Hom(F, \Phi_*F^{x_m}_{\min})=\Hom(\Phi^*F, F^{x_m}_{\min})\neq 0$, so
\begin{equation}\label{eq-case2-1}
    0<\phi_{\sigma_2}(F)\leq \phi_{\sigma_2}(\Phi_*F^{x_m}_{min})=\phi^-_{\sigma_1}(F^{x_m})\leq \phi_{\sigma_1}^+(F^{x_m}) \leq  -x_m+1,
\end{equation}
which gives $x_m\leq 0$. Hence, $x_m \in \{0, -1\}$ by Lemma \ref{lem-phase of push - back}. We define $F':=\cone(\Phi^*F\to F^{x_m})[-1]$.

If $x_m=0$, we have $F' \in \mathcal{P}_{\sigma_1}(1, +\infty)$, so $\Hom^{\leq 0}(F', A_i) = 0$ for $i=1, 2$. Since for each $i=1,2$
\[\Hom(F^{x_m}, A_i[k_1])= 0 = \Hom(\Phi^*F, A_i[k_2])=\Hom(F, \Phi_*A_i[k_2])\]
for any $k_1\notin \{0,1, 2\}$ and $k_2 \notin \{0, 1\}$ and infinitely many $A_i$, applying $\Hom(-, A_i)$ to $F'\to \Phi^*F\to F^{x_m}$, we get $\Hom(F', A_i[k]) = 0$ for $i=1, 2$ and $k\neq 1$. Thus $\Re[Z_1(F')] \leq 0$ and $\Im[Z_1(F')] \geq 0$. Then $\Re[Z_1(F^{x_m})] \geq 0$ as $\Re[Z_1(\Phi^*F)] \geq 0$. So if $[F]$ is not a multiple of $[\Lambda_2]$, then the slope $\mu_1(F')$ is bigger than $\mu_1(\Phi^*F)$, which is not possible by \eqref{eq-case2-1}. Thus the only possibility is when both $[\Phi^*F]$ and $[F^{x_m}]$ are multiples of $\lambda_2$. But then the same argument as in \eqref{comparing im} gives $\Im[Z_1(F^{x_m})] < \Im[Z_1(\Phi^*F)]$, which implies $\Im[Z_1(F')] > 0$. As $F'$ lies in the triangle $F'_{>2} \to F' \to \cH^{-1}_{\cA_1}(F')[1]$ such that $F'_{>2}\in \cP_{\sigma_1}(2,+\infty)$, we see  $-\chi(F'_{>2}, \lambda_2)=\Im[Z_1(F'_{>2})] > 0$. Since we have already seen that $\Hom(F', A_2[k]) = 0$ for $k\neq 1$ and $\Hom(\cH^{-1}_{\cA_1}(F')[1], A_2[k])=0$ for $k\notin \{1,2,3\}$, applying $\Hom(-, A_2)$ to $F'_{>2}\to F'\to \cH^{-1}_{\cA_1}(F')[1]$ and using $\chi(F'_{>2}, \lambda_2)<0$ imply $\ext^1(F'_{>2}, A_2) \neq 0$, a contradiction.

If $x_m = -1$, there is an exact sequence in $\cA_2$:
\[0\to \Phi_*F^{x_m}[-1]\to \Phi_*F'\to \Phi_*\Phi^*F\to 0.\]
Hence, $\chi(\Phi_*F', \Lambda_2)<0$ and so $$\Hom(\Phi_*A_2, \Phi_*F'[1])=\Hom(\Phi^*\Phi_*A_2, F'[1]) = \Hom(F',T_1(\Phi^*\Phi_*A_2)[1])\neq 0,$$  contradicting $F' \in \mathcal{P}_{\sigma_1}(2, +\infty)$ as $T_1(\Phi^*\Phi_*A_2)$ is $\sigma_1$-semistable of phase $\phi_1(A_2)=\frac{1}{2}$. 
\bigskip

\textbf{Step 2.} Now we know $\Phi^*F \in \cA_1$ and we only need to show that it is $\sigma_1$-semistable. Assume not, and let $F^{\max}$ (resp.~$F^{\min}$) be the factor with the maximum (resp.~minimum) phase in the Harder--Narasimhan filtration of $\Phi^*F$, so $F^{\max}, F^{\min}\in \cA_1$ and 
\begin{equation}\label{eq-slope-1}
\mu_1(F^{\max})>\mu_1(\Phi^*F)>\mu_1(F^{\min}).
\end{equation}
Using Lemma \ref{lem-mukai}, one can show via the same argument as in Step 1 of the proof of Lemma \ref{part-1} that there exists $A \in \{F^{\min}, F^{\max}\}$ such that $\Phi_*A$ and $\Phi_*T_1(A)$ are $\sigma_2$-semistable of the same phase $$\phi_{\sigma_2}(\Phi_*A)=\phi_{\sigma_2}(\Phi_*T_1(A)) = \phi_{\sigma_1}(A).$$ 
However, $0 \neq \Hom(F^{\max}, \Phi^*F) = \Hom(\Phi_*T_1(F^{\max}), F)$ and $0 \neq \Hom(\Phi^*F, F^{\min}) = \Hom(F, \Phi_*F^{\min})$, which are not possible by \eqref{eq-slope-1}.

\end{proof}

\begin{lemma}\label{lem-two-factor-stable-pull}
Let $F\in \4$ be a $\sigma_2$-stable object with $\ext^1(F,F)=2k$. If $\Phi^*F\in \3$ is not $\sigma_1$-semistable with the Harder--Narasimhan filtration
\begin{equation}\label{hn-f}
    F_1\to \Phi^*F\to F_2
\end{equation}
such that $F_1$ and $F_2$ are $\sigma_1$-stable, then there exists $i\in \{1,2\}$ such that $\ext^1(F_i, F_i)<2k$.
\end{lemma}
\begin{proof}
    The proof is similar to the proof of Lemma \ref{lem-two-factor-stable-push}, so we only outline the main steps. By Lemma \ref{lem-mukai} and \ref{lem-rhom-pushfor}, we only need to consider the case that $\ext^1(F_1, F_1) = \ext^1(F_2, F_2) = 2k$ and $\ext^1(\Phi^*F, \Phi^*F) = 4k$. Then Lemma \ref{lem-rhom-back} implies that $\Hom(\Phi^*F, \Phi^*F) = \C^2$, $T_2(F) = F$ and $\Phi_*\Phi^*F \cong F \oplus F$. By \ref{c-adjoints}, we have $T_1(\Phi^*F)\cong \Phi^*(T_2(F))\cong \Phi^*F$.

    As in Lemma \ref{lem-two-factor-stable-push}, we get a spectral sequence with the first page
 $$ E^{p,q}_1= \begin{array}{cc|cc}
	0 & 0 & 0 & 0\\
	0 & \Ext^2(F_1,F_2) & 0 & 0\\
	0 & \Ext^1(F_1,F_2) & \Ext^2(F_1,F_1)\oplus \Ext^2(F_2,F_2) & 0\\
	0 & 0 & \CC^{2n} & 0 \\
	0 & 0 & \CC^2 & \Ext^1(F_2,F_1) \\ \hline
	0 & 0 &  0 & \Hom(F_2,F_1)
	\end{array}$$
 which degenerates at $E_3$ and converges to $\Ext^*(\Phi^*F, \Phi^*F)$. Moreover, $d^{0,0}_1$ has a non-zero kernel and is surjective. We know that \eqref{hn-f} does not split, otherwise,  the $\sigma_2$-stability of $F$ forces $\Phi_*F_1 \cong \Phi_*F_2 \cong F$, and so $\Phi^*\Phi_*F_1 \cong \Phi^*\Phi_*F_2$, which is not possible by the phase ordering of $F_1$ and $F_2$. Therefore,
 \begin{equation}\label{h-c}
     \Ext^1(F_2, F_1) = \C = \Hom(F_2, F_1). 
 \end{equation}
 Taking $\Phi_*$ from \eqref{hn-f} gives the exact triangle $\Phi_*F_1 \xrightarrow{h} F \oplus F \xrightarrow{h'} \Phi_*F_2$ so that both $h$ and $h'$ are non-zero because 
 \begin{equation}\label{hh-c}
     \Hom(F, \Phi_*F_2) = \C = \Hom(F, \Phi_*F_1). 
 \end{equation}
 Therefore, we can find an exact triangle $F \xrightarrow{s'} F \oplus F \xrightarrow{s} F$ so that both $s \circ h$ and $h' \circ s'$ are non-zero. Define $G \coloneqq \cone(s \circ h)[-1]$, then we have the following commutative diagram: 
\[\begin{tikzcd}
	G & F& {\Phi_*F_2} \\
	{\Phi_*F_1} & {F\oplus F} & {\Phi_*F_2} \\
	F & F & 0
	\arrow[no head, from=3-1, to=3-2]
	\arrow[from=2-2, to=3-2]
	\arrow["f", from=2-1, to=3-1]
	\arrow["h"', from=2-1, to=2-2]
 \arrow["h'"', from=2-2, to=2-3]
 \arrow["s'"', from=1-2, to=2-2]
 \arrow["s"', from=2-2, to=3-2]
	\arrow[shift right, no head, from=3-1, to=3-2]
	\arrow[from=3-2, to=3-3]
	\arrow[from=2-2, to=2-3]
	\arrow[from=2-3, to=3-3]
	\arrow[no head, from=1-3, to=2-3]
	\arrow[shift right, no head, from=1-3, to=2-3]
	\arrow[from=1-2, to=2-2]
	\arrow[from=1-1, to=2-1]
	\arrow[from=1-1, to=1-2]
	\arrow["g"', from=1-2, to=1-3]
\end{tikzcd}\]
Taking $\Hom(F, -)$ from the first row and applying \eqref{hh-c} show that $\Hom(F, G) = 0$. Thus taking $\Hom(F, -)$ from the first column shows that the first column is splitting. Then we have $\Phi_*F_1 \cong G \oplus F$, which implies $\Hom(\Phi_*F_1, \Phi_*F_2) \neq 0$ and contradicts  $\phi_1(F_1)=\phi_1(\Phi^*\Phi_*F_1) > \phi_1(F_2)$. 
\end{proof}

\section{Gushel--Mukai varieties}\label{sec-very-general-GM}

This section aims to utilize Theorem \ref{thm-main} from the preceding section to establish Theorem \ref{thm-intro-section}, or more precisely, Theorem \ref{thm-very-general-GM4}. At the end of this section, we will use a deformation argument to generalize a part of Theorem \ref{thm-very-general-GM4} from very general cases to general cases (cf.~Theorem \ref{thm-general-gm4}). We begin by investigating some properties of the Kuznetsov components of Gushel--Mukai varieties and then explore the required conditions mentioned in the earlier section.

\subsection{Gushel--Mukai varieties}\label{sec-GM-1} Recall that a Gushel--Mukai (GM) variety $X$ of dimension $n$ is a smooth intersection $$X=\mathrm{Cone}(\Gr(2,5))\cap Q,$$ 
where $\mathrm{Cone}(\Gr(2,5))\subset \PP^{10}$ is the projective cone over the \text{Plücker} embedded Grassmannian $\Gr(2,5)\subset \PP^9$, and $Q\subset \PP^{n+4}$ is a quadric hypersurface. Then $n\leq 6$ and we have a natural morphism $\gamma_X\colon X\to \Gr(2,5)$. We say $X$ is \emph{ordinary} if $\gamma_X$ is a closed immersion, and \emph{special} if $\gamma_X$ is a double covering onto its image. 

\begin{definition}
Let $X$ be a GM variety of dimension $n=4$ or $6$. We say $X$ is \emph{Hodge-special} if 
\[\mathrm{H}^{\frac{n}{2}, \frac{n}{2}}(X)\cap \mathrm{H}_{\mathrm{van}}^n(X, \QQ) \neq 0,\]
where $\mathrm{H}_{\mathrm{van}}^n(X, \QQ):=\mathrm{H}_{\mathrm{van}}^n(X, \ZZ)\otimes \QQ$ and $\mathrm{H}_{\mathrm{van}}^n(X, \ZZ)$ is defined as the orthogonal complement of $$\gamma_X^*\mathrm{H}^n(\Gr(2,5), \ZZ)\subset \mathrm{H}^n(X, \ZZ)$$ with respect to the intersection form.
\end{definition}

By \cite[Corollary 4.6]{debarre2015special}, $X$ is non-Hodge-special when $X$ is very general among all ordinary GM varieties of the same dimension or very general among all special GM varieties of the same dimension.

The semi-orthogonal decomposition of $\D^b(X)$ for a GM variety $X$ of dimension $n\geq 3$ is given by 
\[\D^b(X)=\langle\Ku(X),\oh_X,\cU^{\vee}_X,\cdots,\oh_X((n-3)H),\cU_X^{\vee}((n-3)H)\rangle,\]
where $\cU_X$ is the pull-back of the tautological subbundle via $\gamma_X$. We refer to $\Ku(X)$ as the Kuznetsov component of $X$. We define the projection functors $$\mathrm{pr}_X:=\bR_{\cU_X}\bR_{\oh_X(-H)}\bL_{\oh_X}\bL_{\cU^{\vee}_X}\cdots \bL_{\oh_X((n-4)H)}\bL_{\cU^{\vee}_X((n-4)H)}\colon \D^b(X)\to \Ku(X)$$
when $n\geq 4$ and $\pr_X:=\bL_{\oh_X}\bL_{\cU^{\vee}_X}$ when $n=3$.


When $n=3$, according to the proof of \cite[Proposition 3.9]{kuznetsov2009derived}, $\Knum(\Ku(X))$ is a rank two lattice generated by $\lambda_1$ and $\lambda_2$, where
\begin{equation}\label{lambda}
    \ch(\lambda_1)=-1+\frac{1}{5}H^2,\quad \ch(\lambda_2)=2-H+\frac{1}{12}H^3,
\end{equation}
with the Euler pairing
\begin{equation}\label{eq-matrix-odd}
\left[               
\begin{array}{cc}   
-1 & 0 \\  
0 & -1\\
\end{array}
\right].
\end{equation} 

When $n=4$, there is a rank two sublattice in~$\Knum(\Ku(X))$ generated by $\Lambda_1$ and $\Lambda_2$, where
\begin{equation}\label{Lambda}
    \ch(\Lambda_1)=-2+(H^2-\gamma^*_X \sigma_2)-\frac{1}{20}H^4,
\quad \ch(\Lambda_2)=4-2H+\frac{1}{6}H^3,
\end{equation}
whose Euler pairing is 
\begin{equation}\label{eq-matrix-even}
\left[               
\begin{array}{cc}   
-2 & 0 \\  
0 & -2\\
\end{array}
\right],
\end{equation}
where $\gamma^*_X \sigma_2$ is the pull-back of the Schubert cycle $\sigma_2\in \mathrm{H}^4(\Gr(2, 5), \ZZ)$. When $X$ is non-Hodge-special, by \cite[Proposition 2.25]{kuznetsov2018derived}, we have $\Knum(\Ku(X))=\ZZ\Lambda_1\oplus \ZZ\Lambda_2$. 

When $n=5$, by \cite[Corollary 6.5]{kuznetsov2019categorical}, we can find a smooth GM threefold $X'$ with an equivalence $\Ku(X')\simeq \Ku(X)$. Hence, $\Knum(\Ku(X))$ is also a rank two lattice with the Euler pairing \eqref{eq-matrix-odd}. 

When $n=6$, by \cite[Corollary 6.5]{kuznetsov2019categorical} again, we can find a smooth GM fourfold $X'$ with an equivalence $\Ku(X')\simeq \Ku(X)$. Hence, $\Knum(\Ku(X))$ contains a rank two lattice whose Euler pairing is the same as \eqref{eq-matrix-even}, which is the whole numerical Grothendieck group when $X$ is non-Hodge-special by \cite[Corollary 4.6]{debarre2015special} and \cite[Proposition 2.25]{kuznetsov2018derived}.

Now let $Y$ (resp.~$X$) be a GM variety of dimension $n-1$ (resp.~$n$) where $n\geq 4$, and $j \colon Y \hookrightarrow X$ represents a hyperplane section. For any object $E\in \Ku(Y)$, we have
\begin{equation}\label{pr}
 \pr_X(j_*E)=\bR_{\cU_X}\bR_{\oh_X(-H)}(j_*E).   
\end{equation}
This can be easily deduced from the adjunction of $j^*$ and $j_*$. For example, when $X$ is a GM fourfold, we have
    \[\RHom_X(\cU^{\vee}_X,j_*E)=\RHom_Y(\cU^{\vee}_Y,E)=0,\quad \RHom_X(\oh_X,j_*E)=\RHom_Y(\oh_Y,E)=0\]
    since $E\in \Ku(Y)$.



\begin{lemma} \label{pull-back-in-ku}
    For any object $F\in \Ku(X)$, we have $j^*F\in \Ku(Y)$.
\end{lemma}

\begin{proof}
    By Serre duality and adjunction, the functor $j^!=j_*\circ -\otimes \oh_Y(H)[-1]$ is the left adjoint of $j^*$. Thus, in the GM fourfold case,  we need to show 
    \[\RHom_X(j_*(\cU^{\vee}_Y(H))[-1], F)=0,\quad \RHom_X(j_*(\oh_Y(H))[-1], F)=0.\]
Then the result follows from the fact that $F\in \Ku(X)$ and applying $\Hom_X(-, F)$ to the exact triangles $\oh_X\to\oh_X(H)\to j_*\oh_Y(H)$ and $\cU^{\vee}_X\to \cU^{\vee}_X(H)\to j_*\cU^{\vee}_Y(H)$. The other cases are similar.
\end{proof}

Therefore, we have well-defined functors
\[\pr_X\circ j_* \colon \Ku(Y) \to \Ku(X) \qquad \text{and} \qquad j^*\colon \Ku(X)\to \Ku(Y).\]

\begin{lemma}\label{lem-involution}
Let $X$ be a GM variety of dimension $n \geq 3$. We define 
    \[T_X:=\bR_{\cU_X} \bR_{\oh_X(-H)}\circ (-\otimes \oh_X(-H))[1]\colon \Ku(X)\to \Ku(X).\]
\begin{enumerate}
    \item If $n$ is odd, then $T_X=S_{\Ku(X)}[-2]$ is an involution on $\Ku(X)$ that acts trivially on $\Knum(\Ku(X))$.

    \item If $n$ is even and $X$ is non-Hodge-special, then $T_X$ is an involution on $\Ku(X)$ that acts trivially on $\Knum(\Ku(X))$. Moreover, $\Ku(X)$ has the Serre functor $S_{\Ku(X)}=[2]$
\end{enumerate}

\end{lemma}

\begin{proof}
Note that $T_X$ is the inverse of $\bL_{\oh_X}\bL_{\cU^{\vee}_X}\circ (-\otimes \oh_X(H))[-1]|_{\Ku(X)}$, then both of them are involutions on $\Ku(X)$ by \cite[Theorem 4.15]{bayer2022kuznetsov}. Thus, (2) follows from \cite[Proposition 5.7]{bayer2022kuznetsov}. And (1) can be deduced from a direct computation using the Euler pairing, as 
\[\chi(v, T_X(w))=\chi(v, S_{\Ku(X)}(w))=\chi(w, v)=\chi(v, w)\]
for any two classes $v,w\in \Knum(\Ku(X))$.
\end{proof}

\begin{lemma}\label{lem-adjoint-functors}

\begin{enumerate}
    \item If $n=4,6$, we have adjoint pairs 
    \[\pr_X\circ j_*\circ T_Y\dashv j^* \dashv \pr_X\circ j_*\]
    and isomorphisms of functors $j^*\circ T_X\cong T_Y\circ j^*$ and $\pr_X\circ j_* \circ T_Y\cong T_X\circ \pr_X\circ j_*$.

    \item If $n=5$, we have adjoint pairs
    \[j^* \dashv \pr_X\circ j_* \dashv j^*\circ T_X\]
    and isomorphisms of functors $j^*\circ T_X\cong T_Y\circ j^*$ and $\pr_X\circ j_* \circ T_Y\cong T_X\circ \pr_X\circ j_*$.
\end{enumerate}

\end{lemma}

\begin{proof}
Firstly, we have the adjunction $j^* \dashv \pr_X\circ j_*$ because for any object $F\in \Ku(X)$ and $E\in \Ku(Y)$, we have
\[\Hom_Y(j^*F, E)=\Hom_X(F, j_*E)=\Hom_X(F, \pr_X(j_*E)),\]
where the last equality follows from adjunctions of right mutations as $\pr_X(j_*E)=\bR_{\cU_X}\bR_{\oh_X(-H)}(j_*E)$ by \eqref{pr}. Then the existence of adjoint pairs in (1) and (2) follows from the description of Serre functors on $\Ku(Y)$ and $\Ku(X)$.

As decompositions of $\D^b(Y)$ and $\D^b(X)$ fit into the situation of \cite[Corollary 4.19(1)]{kuznetsov:serre-dimension}, the isomorphism $j^*\circ T_X\cong T_Y\circ j^*$ in (1) follows from \cite[Lemma 4.12]{kuznetsov:serre-dimension}. 
Then for any pair of objects $E,F\in \Ku(X)$, we have functorial isomorphisms
\[\Hom_X(E, \pr_X(j_*T_Y(F)))\cong\Hom_Y(j^*E, T_Y(F))\cong\Hom_Y(T_Y(j^*E), F)\]
\[\cong \Hom_Y(j^*(T_Y(E)), F)\cong \Hom_X(T_X(E), \pr_X(j_*F))\cong \Hom_X(E, T_X(\pr_X(j_*F))),\]
where the first and the fourth isomorphism follow from the adjunction of functors, the second and the last one follow from Lemma \ref{lem-involution} that $T_X$ and $T_Y$ are involutions, and the third one is given by $j^*\circ T_X\cong T_Y\circ j^*$. Therefore, the last statement of (2) follows from the Yoneda Lemma. The last statement of (1) is similar.
\end{proof}



The following will be useful later in our computations. 
\begin{proposition} \label{prop-pushforward}
For any object $E\in \D^b(Y)$, we have
\[\pr_X(j_*E)\cong\pr_X(j_*\pr_Y(E)).\]
\end{proposition}

\begin{proof}
As the cone of the natural map $E\to \pr_Y(E)$ is contained in ${}^\bot\Ku(Y)$, we only need to show $\pr_X(j_*E)\cong 0$ for any 
$E\in \langle \oh_Y, \cU^{\vee}_Y,\cdots,\oh_X((n-3)H),\cU_X^{\vee}((n-3)H)\rangle$.

First of all, we assume that $X$ is a GM fourfold. Recall that $\pr_X=\bR_{\cU_X}\bR_{\oh_X(-H)}\bL_{\oh_X}\bL_{\cU_X^{\vee}}$. Then from the exact sequence $0\to \cU_X\to \cU_X^{\vee}\to j_*\cU_Y^{\vee}\to 0$, we see $\pr_X(j_*\cU^{\vee}_Y)=\pr_X(\cU_X[1])$. Since we have  $\RHom_X(\cU^{\vee}_X, \cU_X)=0$ by \cite[Lemma 5.4(2)]{GLZ2021conics}, then $\bL_{\oh_X}\bL_{\cU_X^{\vee}}(\cU_X)=\cU_X$, hence, we deduce that $\pr_X(\cU_X)=\bR_{\cU_X}\bR_{\oh_X(-H)}(\cU_X)=0$. Similarly, we have $\pr_X(j_*\oh_Y)=0$ and the result follows. The arguments for other cases are similar.
\end{proof}




\subsubsection{Stability conditions on the Kuznetsov components}\label{sec-stab-ku}


In \cite{bayer2017stability}, the authors provide a way to construct stability conditions on a semi-orthogonal component from weak stability conditions on a larger category. For GM varieties, we have the following. 

\begin{theorem}[{\cite{perry2019stability}}]\label{blms-induce}
Let $X$ be 
a GM variety of dimension $n \geq 3$. Then there exists a family of stability conditions on $\Ku(X)$.
\end{theorem}

When $X$ is a GM fourfold, we denote by $\Stab^{\circ}(\Ku(X))$ the family of stability conditions on $\Ku(X)$ constructed in \cite[Theorem 4.12]{perry2019stability}.

When $Y$ is a GM threefold, it is proved in \cite{pertusiGM3fold} that stability conditions on $\Ku(Y)$ constructed in \cite{bayer2017stability} are \emph{Serre-invariant}. Furthermore, they all belong to the same $\GL$-orbit.

\begin{theorem}[{\cite{JLLZ2021gushelmukai,FeyzbakhshPertusi2021stab}}]\label{thm-unique-threefold}
Let $Y$ be a GM threefold, then all Serre-invariant stability conditions on $\Ku(Y)$ are contained in the same $\GL$-orbit.
\end{theorem}

\subsection{Stability of pull-back and push-forward} One of the main goals of this subsection is to prove the following Theorem. 


\begin{theorem}\label{thm-very-general-GM4}
Let $X$ be a GM variety of dimension $n\geq 4$ and $j \colon Y\hookrightarrow X$ be a smooth hyperplane section. Furthermore, we assume that the one between $X$ and $Y$ with even dimension is non-Hodge-special. Let $\sigma_Y$ and $\sigma_X$ be Serre-invariant stability conditions on $\Ku(Y)$ and $\Ku(X)$, respectively. 

\begin{enumerate}
    \item An object $E\in \Ku(Y)$ is $\sigma_Y$-semistable if and only if $\pr_X(j_*E)\in \Ku(X)$ is $\sigma_X$-semistable. Moreover, $\pr_X(j_*E)$ is $\sigma_X$-stable if and only if $E$ is $\sigma_Y$-stable and there is no $\sigma_X$-stable object $F \in \Ku(X)$ such that $j^*F \cong E$.

    \item An object $F\in \Ku(X)$ is $\sigma_X$-semistable if and only if $j^*F\in \Ku(Y)$ is $\sigma_Y$-semistable. Moreover, $j^*F$ is $\sigma_Y$-stable if and only if $F$ is $\sigma_X$-stable and there is no $\sigma_Y$-stable object $E$ such that $\pr_X(j_*E) \cong F$.
\end{enumerate}
\end{theorem}

\begin{remark}
In Section \ref{subsec-general-gm}, we will explain how to generalize this theorem from very general cases to general cases.
\end{remark}

To prove Theorem \ref{thm-very-general-GM4}, we only need to verify conditions \ref{c1} to \ref{c-t2} in the earlier section for the functors $j^*$ and $\pr \circ j_*$. To accomplish this, we need to further analyze these two adjoint functors.

\begin{lemma}\label{lem-exact-triangle}
Let $X$ be a GM variety of dimension $n$ and $j\colon Y\hookrightarrow X$ be a smooth hyperplane section.

\begin{enumerate}
    \item If $n=4,6$, we have exact triangles
    \begin{equation}\label{eq-triangle-functor-1}
    T_Y\to j^*\circ (\pr_X\circ j_*) \to \mathrm{id}_{\Ku(Y)}
    \end{equation}
and
\begin{equation}\label{eq-triangle-functor-2}
    \mathrm{id}_{\Ku(X)}\to (\pr_X\circ j_*)\circ j^*\to T_X.
\end{equation}
    \item If $n=5$, we have exact triangles
    \begin{equation}\label{eq-triangle-functor-3}
        T_X\to (\pr_X\circ j_*)\circ (j^*\circ T_X) \to \mathrm{id}_{\Ku(X)}
    \end{equation}
    and
    \begin{equation}\label{eq-triangle-functor-4}
        \mathrm{id}_{\Ku(Y)}\to (j^*\circ T_X)\circ (\pr_X\circ j_*) \to T_Y.
    \end{equation}
\end{enumerate}

\end{lemma}

\begin{proof}
By \cite[Corollary 4.19]{kuznetsov:serre-dimension}, the functor $j^*\colon \Ku(X)\to \Ku(Y)$ is a spherical functor in the sense of \cite[Definition 2.1]{kuznetsov:serre-dimension}. Then applying Lemma \ref{lem-adjoint-functors} and \cite[Corollary 2.3]{kuznetsov:serre-dimension}, the functor $\pr_X\circ j_*$ is spherical as well. Thus according to Lemma \ref{lem-adjoint-functors} and \cite[Equation (2.5)]{kuznetsov:serre-dimension}, we have an exact triangle
\[j^*\circ \pr_X\circ j_* \to \mathrm{id}_{\Ku(Y)} \to \mathbf{T}_{j^*, \pr_X\circ j_*},\]
where the functor $\mathbf{T}_{j^*, \pr_X\circ j_*}$ is defined in \cite[Definition 2.1]{kuznetsov:serre-dimension}. 

When $n=4,6$, to show the existence of \eqref{eq-triangle-functor-1}, we only need to determine $\mathbf{T}_{j^*, \pr_X\circ j_*}$. From \cite[Corollary 4.19(1)]{kuznetsov:serre-dimension}, we see $S_{\Ku(Y)}\cong \mathbf{T}_{j^*, \pr_X\circ j_*}\circ [1]$, which gives \eqref{eq-triangle-functor-1}. 

For \eqref{eq-triangle-functor-2}, since $j^*E\in \Ku(Y)$ for any object $E\in \Ku(X)$, by \eqref{pr} we have a functorial isomorphism $\pr_X(j_*j^*E)=\bR_{\cU_X} \bR_{\oh_X(-H)}(j_*j^*E)$. Hence,  \eqref{eq-triangle-functor-2} follows from composing $\bR_{\cU_X} \bR_{\oh_X(-H)}$ with  the standard exact triangle  $\mathrm{id}_{\D^b(X)}\to j_*\circ j^*\to (-\otimes \oh_X(-H))[1]$.

Now assume that $n=5$. As in the case of odd dimension, we have exact triangles
\begin{equation}\label{eq-triangle-functor-5}
    T_Y\to j^*\circ (\pr_X\circ j_*) \to \mathrm{id}_{\Ku(Y)}
\end{equation}
and
\begin{equation}\label{eq-triangle-functor-6}
    \mathrm{id}_{\Ku(X)}\to (\pr_X\circ j_*)\circ j^*\to T_X.
\end{equation}
Therefore, \eqref{eq-triangle-functor-4} follows from composing $T_Y$ with \eqref{eq-triangle-functor-5} and using the isomorphism
\[T_Y\circ j^*\circ (\pr_X\circ j_*)\cong j^*\circ T_X\circ (\pr_X\circ j_*)\]
given in Lemma \ref{lem-adjoint-functors}(2). Similarly, \eqref{eq-triangle-functor-3} follows from composing $T_X$ with \eqref{eq-triangle-functor-6} and using isomorphisms
\[T_X\circ (\pr_X\circ j_*)\circ j^*\cong (\pr_X\circ j_*)\circ T_Y \circ j^*\cong (\pr_X\circ j_*)\circ j^*\circ T_X\]
given in Lemma \ref{lem-adjoint-functors}(2).
\end{proof}

Now we investigate how our functors act on the numerical Grothendieck groups. 

\begin{lemma}\label{lem-compute-class}
Let $X$ be a GM variety of dimension $n$ and $j\colon Y\hookrightarrow X$ be a smooth hyperplane section. 
\begin{enumerate}
    \item If $n=4$, then $\pr_X(j_*\lambda_i)=\Lambda_i$ and $j^*\Lambda_i=2\lambda_i$ for each $i=1,2$.
    
    \item If $n=5$, we define $\lambda_i$ to be the unique numerical class satisfying $j^*\lambda_i=\Lambda_i$ for each $i=1,2$. Then $\pr_X(j_*\Lambda_i)=2\lambda_i$ for each $i=1,2$ and $\Knum(\Ku(X))=\ZZ\lambda_1\oplus \ZZ \lambda_2$ with the Euler pairing \eqref{eq-matrix-odd}.
    
    \item If $n=6$, we define $\Lambda_i:=\pr_X(j_*\lambda_i)$. Then the restriction of the Euler pairing to the sublattice $\ZZ\Lambda_1\oplus \ZZ\Lambda_2\subset \Knum(\Ku(X))$ is given by \eqref{eq-matrix-even}.
\end{enumerate}

\end{lemma}

\begin{proof}
When $n=4$, as each class has been described in \eqref{lambda} and \eqref{Lambda}, it is straightforward to verify that $\pr_X(j_*\lambda_i)=\Lambda_i$ and $j^*\Lambda_i=2\lambda_i$ for each $i=1,2$. 

When $n=5$, by Lemma \ref{lem-involution}, Lemma \ref{lem-adjoint-functors}, and \eqref{eq-triangle-functor-3}, we have 
\[\chi(j^*v, j^*w)=\chi(v, \pr_X(j_*j^*w))=\chi(v, \pr_X(j_*j^*T_X(w)))=2\chi(v,w)\]
for $i=1,2$ and any classes $v,w\in \Knum(\Ku(X))$. Thus, $\chi(\lambda_i, \lambda_j)=-\delta_{ij}$. Hence, $\lambda_1$ and $\lambda_2$ generate $\Knum(\Ku(X))$ from the description of the Euler pairing \eqref{eq-matrix-odd}. Therefore, if we take any class $v\in \Knum(\Ku(X))$, then we can write $v=a_1\lambda_1+a_2\lambda_2$, which gives $j^*v=a_1\Lambda_1+a_2\Lambda_2$ and
\[\chi(v, \pr_X(j_*\Lambda_i))=\chi(j^*v, \Lambda_i)=-2a_i=\chi(v,2\lambda_i).\]
This shows that $\pr_X(j_*\Lambda_i)=2\lambda_i$.

When $n=6$, by Lemma \ref{lem-involution}, Lemma \ref{lem-adjoint-functors}, and \eqref{eq-triangle-functor-1}, we have
\[\chi(\pr_X(j_*v), \pr_X(j_*w))=\chi(j^*\pr_X(j_*v), w)=2\chi(v,w)\]
for each $i=1,2$ and classes $v,w\in \Knum(\Ku(Y))$. Hence,  $\Lambda_1$ and $\Lambda_2$ generate a rank two lattice, and the description of the Euler pairing follows from (2).
\end{proof}

The final step in proving Theorem \ref{thm-very-general-GM4} is to explore further stability conditions on $\Ku(X)$ and $\Ku(Y)$. We begin by examining the stability of specific objects in these categories. 

\begin{lemma}\label{lem-mukai-4fold}
Let $X$ be a GM variety of dimension $n \geq 3$. 
    \begin{enumerate}
        \item If $n=3,5$, any object $E \in \Ku(X)$ with $\ext^1_X(E, E) \leq 3$ is stable with respect to every Serre-invariant stability condition on $\Ku(X)$.  
        \item If $n=4, 6$ and $X$ is non-Hodge-special, any object $E\in\Ku(X)$ with $\mathrm{ext}_X^1(E,E)<8$ is stable with respect to every stability condition on $\Ku(X)$.
    \end{enumerate}
\end{lemma}
\begin{proof}
    Part (1) follows from \cite[Proposition 3.4(c)]{FeyzbakhshPertusi2021stab}. For part (2), one can easily check that Lemma \ref{lem-lower bound on ext}(1) holds for $\Ku(X)$. Then by Lemma \ref{lem-mukai}, $E$ is $\sigma$-semistable for any stability condition $\sigma$ on $\Ku(X)$. If $E$ is strictly $\sigma$-semistable and has at least two non-isomorphic Jordan--H\"older factors, then by looking at the Jordan--H\"older filtration of $E$, we can also find an exact triangle $A\to E\to B$ such that $A$ is semistable with all Jordan--H\"older factors being a Jordan--H\"older factor $A'$ of $E$ and $B$ does not have $A'$ as a factor (see also \cite[Lemma 4.1]{liu:stability-manifold-of-ku}). This gives $$\Hom_X(A, B)=\Hom_X(B, A[2])=0.$$ However, this also leads to a contradiction as above by $\ext^1_X(E, E)<8$, Lemma \ref{lem-mukai}, and Lemma \ref{lem-lower bound on ext}(1). If $E$ is strictly $\sigma$-semistable such that all Jordan--H\"older factors are isomorphic, then $[E]\in \Knum(\Ku(X))$ is not primitive. As
\[\chi(E,E)\geq 2-\ext^1_X(E,E)>-6,\]
from \eqref{eq-matrix-even} we know that $\chi(E,E)=-2$ or $-4$. But in each case, $[E]$ is a primitive class, which makes a contradiction.
\end{proof}

\begin{proposition}\label{prop-gm4-unique}
Let $X$ be a non-Hodge-special GM fourfold or sixfold. Then all stability conditions on $\Ku(X)$ are in the same $\GL$-orbit.
\end{proposition}

\begin{proof}
By \cite[Corollary 6.5]{kuznetsov2019categorical}, we only need to prove the statement when $\dim X=4$. Since $X$ is non-Hodge-special, $\Knum(\Ku(X))$ is generated by $\Lambda_1$ and $\Lambda_2$. Let $j\colon Y\hookrightarrow X$ be a smooth GM threefold and $F_1, F_1'$ and $F_2$ be objects in $\Ku(Y)$ defined in \cite[Lemma A.8]{JLLZ2021gushelmukai}. Denote by $\sigma_Y$ a Serre-invariant stability condition on $\Ku(Y)$. Then by \cite[Lemma A.8]{JLLZ2021gushelmukai}, we have
$$\phi_{\sigma_Y}(F_2)-1<\phi_{\sigma_Y}(F_1)=\phi_{\sigma_Y}(F_1')<\phi_{\sigma_Y}(F_2)$$
with $\Hom_Y(F_2, F_1[1])\neq 0$ and $\Hom_Y(F_1', F_2)\neq 0$. Moreover, we have 
\[\RHom_Y(F_1, F_1)=\RHom_Y(F_1', F_1')=\RHom_Y(F_2, F_2)=\CC\oplus \CC^3[-1]\]
as in the proof of \cite[Lemma A.7]{JLLZ2021gushelmukai}. 

The same argument as in Lemma \ref{lem-rhom-pushfor} implies that for any stable object $E\in \Ku(Y)$ with respect to a Serre-invariant stability condition on $\Ku(Y)$ with $\Ext^1_Y(E, E)=\CC^n$, we have $$\ext^1_X(\pr_X(j_*E), \pr_X(j_*E))\leq 2n.$$ Thus by Lemma \ref{lem-mukai-4fold}, $\pr_X(j_*F_1), \pr_X(j_*F_1')$ and $\pr_X(j_*F_2)$ are stable with respect to any stability condition on $\Ku(X)$.

Since $\phi_{\sigma_Y}(F_1')=\phi_{\sigma_Y}(T_Y(F_1'))>\phi_{\sigma_Y}(F_2)-1$, we have $\Hom(T_Y(F_1'), F_2[-1])=0$. After applying $\Hom(-,F_2)$ to the exact triangle 
\[T_Y(F_1')\to j^*\pr_X(j_*F_1')\to F_1'\]
in \eqref{exact-3}, we obtain an injection $0\neq \Hom_Y(F_1', F_2)\hookrightarrow \Hom_Y(j^*\pr_X(j_*F_1'), F_2)$, which implies
\[\Hom_Y(j^*\pr_X(j_*F_1'), F_2)=\Hom_X(\pr_X(j_*F_1'), \pr_X(j_*F_2))\neq 0.\]
Similarly, since $\phi_{\sigma_Y}(F_2)=\phi_{\sigma_Y}(T_Y(F_2))>\phi_{\sigma_Y}(F_1)$, applying $\Hom_Y(-,F_1)$ to the exact triangle 
\[T_Y(F_2)\to j^*\pr_X(j_*F_2)\to F_2\]
in \eqref{exact-3}, we get an injection $0\neq \Hom_Y(F_2, F_1[1])\hookrightarrow \Hom_Y(j^*\pr_X(j_*F_2), F_1[1])$, which implies
\[\Hom_Y(j^*\pr_X(j_*F_2), F_1[1])=\Hom_X(\pr_X(j_*F_2), \pr_X(j_*F_1)[1])\neq 0.\]
By the same argument, we have $\Hom_Y(j^*\pr_X(j_*F_1), F_1'[3])=\Hom_X(\pr_X(j_*F_1), \pr_X(j_*F_1')[3])=0$ as well. Therefore, if we define $Q_2:=\pr_X(j_*F_1')$, $Q_2':=\pr_X(j_*F_1)$ and $Q_1:=\pr_X(j_*F_2)$, then they satisfy the assumptions in \cite[Lemma 3.7]{FeyzbakhshPertusi2021stab}. Thus the result follows from \cite[Theorem 3.2]{FeyzbakhshPertusi2021stab}.
\end{proof}

Now we are prepared to prove the main theorem in this subsection. 

\begin{proof}[{Proof of Theorem \ref{thm-very-general-GM4}}]
We start by fixing the following data:  
\begin{itemize}\label{set}
    \item Categories
    \[\cK_1:=\Ku(Y),\quad \cK_2:=\Ku(X)\]
    when $n$ is even, and 
    \[\cK_1:=\Ku(X),\quad \cK_2:=\Ku(Y)\]
    when $n$ is odd, 

    \item classes $\lambda_i$ and $\Lambda_i$ as in \eqref{lambda}, \eqref{Lambda} and Lemma \ref{lem-compute-class} for each $i=1,2$,

    \item auto-equivalences $T_1$ (resp.~$T_2$) are $T_Y$ (resp.~$T_X$) as in Lemma \ref{lem-involution} when $n$ is even and $T_X$ (resp.~$T_Y$) when $n$ is odd,
    \item functors
    \[\Phi^*:=j^*|_{\Ku(X)},\quad \Phi_*:=\pr_X\circ j_*\]
    when $n$ is even, and
    \[\Phi^*:=\pr_X\circ j_*,\quad \Phi_*:=j^*\circ T_X\]
    when $n$ is odd.
\end{itemize}

We fix Serre-invariant stability conditions $\sigma_1 = (Z_1, \cA_1)$ and $\sigma_2 = (Z_2, \cA_2)$ on $\cK_1$ and $\cK_2$, respectively. By Proposition \ref{prop-gm4-unique} and Theorem \ref{thm-unique-threefold}, up to $\GL$-action, we can assume that
$$Z_1(\lambda_1) = Z_2(\Lambda_1) = -1  \ , \qquad Z_1(\lambda_2) = Z_2(\Lambda_2) = \mathfrak{i}.$$

\textbf{Claim.} We can furthermore assume that for any $\sigma_1$-stable object $A\in \cK_1$ with $[A]\in \{\pm\lambda_1, \pm\lambda_2\}$, the object $\Phi_*(A)$ is $\sigma_2$-stable with $\phi_{\sigma_1}(A)=\phi_{\sigma_2}(\Phi_*(A))$.

\begin{proof}[{Proof of the Claim}]
Since $\ext^1(A,A)\leq 3$, we have $\ext^1(\Phi_*(A),\Phi_*(A))\leq 6$ via the same argument as in Lemma \ref{lem-rhom-pushfor} 
and so $\Phi_*(A)$ is $\sigma_2$-stable by Lemma \ref{lem-mukai-4fold}. Since $Z_1=Z_2$, we know that
\begin{equation}\label{eq-phase-in-2Z}
    \phi_{\sigma_1}(A)-\phi_{\sigma_2}(\Phi_*(A))\in 2\ZZ.
\end{equation}

Let $D_1$ and $D_2$ be two $\sigma_1$-stable objects defined in \cite[Lemma 5.6]{liu:stability-manifold-of-ku} when $\cK_1$ is the Kuznetsov component of a GM threefold. When $\cK_1$ is the Kuznetsov component of a GM fivefold, we denote the image of $D_1$ and $D_2$ under the equivalence in \cite{kuznetsov2019categorical} by $D_1$ and $D_2$ as well. Up to shift and relabeling the subscript, we can fix the hearts $\cA_1$ and $\cA_2$ so that $D_1, D_2\in \cA_3$, $[D_i]=\lambda_i$, and $$\phi_{\sigma_1}(D_1)=\phi_{\sigma_2}(\Phi_*(D_1))=1.$$ As in \cite[Lemma 5.6]{liu:stability-manifold-of-ku}, we have $\Hom(D_2, D_1)\neq 0$. 

Firstly, we claim that
\[\phi_{\sigma_1}(D_2)=\phi_{\sigma_2}(\Phi_*(D_2))=\frac{1}{2}.\]
Indeed, applying $\Hom(-, D_1)$ to the exact triangle in \eqref{exact-3}, we get $$\Hom(\Phi^*\Phi_*(D_2), D_1)=\Hom(\Phi_*(D_2), \Phi_*(D_1))\neq 0.$$ Thus $$\phi_{\sigma_2}(\Phi_*(D_1))-2<\phi_{\sigma_2}(\Phi_*(D_2))<\phi_{\sigma_2}(\Phi_*(D_1))$$ since $$\Hom(\Phi_*(D_2), \Phi_*(D_1))=\Hom(\Phi_*(D_1), \Phi_*(D_2)[2])\neq 0.$$
As $\phi_{\sigma_2}(\Phi_*(D_1))=1$, we get
\[-1<\phi_{\sigma_2}(\Phi_*(D_2))<1,\]
which implies $\phi_{\sigma_2}(\pr_X(j_*D_2))=\phi_{\sigma_1}(D_2)=\frac{1}{2}$ by \eqref{eq-phase-in-2Z} and the claim follows.

Now let $A\in \cA_3$ be an arbitrary $\sigma_1$-stable object with $[A]=\lambda_i$ for $i\in \{1,2\}$. From $\chi(D_i, A)<0$ and the fact that $\cA_1$ has homological dimension $2$, we get $$\Hom(D_i, A[1])=\Hom(A, T_1(D_i)[1])\neq 0.$$ Hence, applying $\Hom(-, A)$ to the exact triangle in \eqref{exact-3}, we see
\[\Hom(\Phi^*\Phi_*(D_i), A[1])=\Hom(\Phi_*(D_i), \Phi_*(A)[1])\neq 0,\]
which implies $\phi_{\sigma_2}(\Phi_*(D_i))-1<\phi_{\sigma_2}(\Phi_*(A))$. Similarly, we have $\phi_{\sigma_2}(\Phi_*(A))<\phi_{\sigma_2}(\Phi_*(D_i))+1$ as well. Since $\phi_{\sigma_1}(A)=\phi_{\sigma_1}(D_i)=\phi_{\sigma_2}(\Phi_*(D_i))$, we get
\[\phi_{\sigma_1}(A)-1<\phi_{\sigma_2}(\Phi_*(A))<\phi_{\sigma_1}(A)+1\]
and the result follows from \eqref{eq-phase-in-2Z}.
\end{proof}

Finally, by Theorem \ref{thm-unique-threefold} and Proposition \ref{prop-gm4-unique}, $\sigma_i$ is $T_i$-invariant for each $i=1, 2$. Moreover, \cite[Theorem 7.12, Theorem 8.9]{JLLZ2021gushelmukai} shows that there is a two-dimensional family of stable objects of class $\lambda_i$ for each $i=1,2$ in $\mathcal{K}_1$. Therefore, Lemma \ref{lem-involution}, Lemma \ref{lem-adjoint-functors}, Lemma \ref{lem-exact-triangle}, and Lemma \ref{lem-compute-class} demonstrate that the above given data satisfies all conditions \ref{c1} to \ref{c-t2} in Section \ref{sec-general-criterion}, thus the claim follows from Theorem \ref{thm-main}.
\end{proof}

As a corollary, we have:

\begin{corollary}\label{cor-push-stable}
Let $X$ be a non-Hodge-special GM variety of dimension $n=4,6$ and $j: Y\hookrightarrow X$ be a smooth hyperplane section. If $E\in \Ku(Y)$ is a stable object with respect to Serre-invariant stability conditions such that $[E]\in \Knum(\Ku(Y))$ is primitive, then $\pr_X(j_*E)$ is stable with respect to any stability condition on $\Ku(X)$.
\end{corollary}

\begin{proof}
By Theorem \ref{thm-very-general-GM4}, $\pr_X(j_*E)$ is semistable with respect to any stability condition on $\Ku(X)$. As $[E]$ is primitive, we know that $[\pr_X(j_*E)]$ is primitive as well by Lemma \ref{lem-compute-class}(1). Hence, $\pr_X(j_*E)$ is stable.
\end{proof}

Similarly, we have:

\begin{corollary}\label{cor-push-stable-5fold}
Let $X$ be a non-Hodge-special GM fourfold and $j: X\hookrightarrow Y$ realize $X$ as a smooth hyperplane section of a GM fivefold $Y$. If $E\in \Ku(Y)$ is an object stable with respect to Serre-invariant stability conditions such that $[E]\in \Knum(\Ku(Y))$ is primitive, then $j^*E$ is stable with respect to any stability condition on $\Ku(X)$.
\end{corollary}

\subsection{General cases}\label{subsec-general-gm}
In this section, we aim to extend Theorem \ref{thm-very-general-GM4}(1) to the general case, and prove the following. 

\begin{theorem}\label{thm-general-gm4}
Let $a,b$ be a pair of integers. Then for a general GM fourfold $X$ and its smooth hyperplane section $j \colon Y\hookrightarrow X$, if $E\in \Ku(Y)$ is a $\sigma_Y$-semistable object of class $a\lambda_1+b\lambda_2$, then $\pr_X(j_*E)$ is $\sigma_X$-semistable, where $\sigma_Y$ is a Serre-invariant stability condition on $\Ku(Y)$ and $\sigma_X\in \Stab^{\circ}(\Ku(X))$.
\end{theorem}
Note that the generality assumption in Theorem \ref{thm-general-gm4} means that the statement holds when $X$ is in an open dense substack of the moduli stack of smooth ordinary GM fourfolds. Moreover, as we will see in the proof, such a substack depends on the choice of integers $a,b$.

\begin{proof}

Let $\cM^{\mathrm{GM}}_4$ be the moduli stack of smooth ordinary GM fourfolds, which is a smooth irreducible Deligne--Mumford stack of finite type and separated over $\CC$ (cf.~\cite[Proposition A.2]{kuznetsov2018derived} and \cite[Corollary 5.12]{debarre:gm-moduli}). For a smooth projective morphism $\cX\to S$ over a smooth scheme $S$ over $\CC$ and a $S$-linear semi-orthogonal component $\cD\subset \D^b(\cX)$, we denote by $\cM_{\mathrm{pug}}(\cD/S)$ the moduli stack of universally gluable objects in $\cD$ over $S$, defined in \cite[Definition 9.1]{BLMNPS21}. 
According to \cite[Lemma 21.12]{BLMNPS21}, when $X$ is a GM fourfold, $\cM^X_{\sigma_X}(a,b)$ is an open substack of $\cM_{\mathrm{pug}}(\Ku(X)/\CC)$. And by Lemma \ref{lem-rhom-pushfor}, for any smooth hyperplane section $j\colon Y\hookrightarrow X$, we have a morphism
\[\gamma\colon \cM^Y_{\sigma_Y}(a,b)\to \cM_{\mathrm{pug}}(\Ku(X)/\CC)\]
given by $E\mapsto \pr_X(j_*E)$ at the level of $\CC$-point. Therefore, given a $\CC$-point $s\colon \Spec(\CC)\to \cM^Y_{\sigma_Y}(a,b)$ corresponding to a $\sigma_Y$-semistable object $E\in \Ku(Y)$, to show $\pr_X(j_*E)\in \Ku(X)$ is $\sigma_X$-semistable, we need to show the composition
\[\Spec(\CC)\xra{s}\cM^Y_{\sigma_Y}(a,b)\xra{\gamma}\cM_{\mathrm{pug}}(\Ku(X)/\CC)\]
factor through the natural open immersion $\cM^X_{\sigma_X}(a,b)\subset \cM_{\mathrm{pug}}(\Ku(X)/\CC)$. By \cite[Theorem 12.17(3)]{BLMNPS21}, this is equivalent to finding an extension $\kappa$ of $\CC$ such that the composition
\[\Spec(\kappa)\xra{s_{\kappa}}\cM^Y_{\sigma_Y}(a,b)\xra{\gamma}\cM_{\mathrm{pug}}(\Ku(X)/\CC)\]
factors through $\cM^X_{\sigma_X}(a,b)\subset \cM_{\mathrm{pug}}(\Ku(X)/\CC)$. We denote by $|\cM|$ the associated topological space of an algebraic stack $\cM$. From the definition of topological spaces of algebraic stacks (cf.~\cite[\href{https://stacks.math.columbia.edu/tag/04XG}{Tag 04XG}]{stacks-project}), it is also equivalent to say 
$$
s_{\kappa}\in |\cM^X_{\sigma_X}(a,b)|\cap \gamma(|\cM^Y_{\sigma_Y}(a,b)|)\subset |\cM_{\mathrm{pug}}(\Ku(X)/\CC)|
$$ 
Thus, to prove that $\pr_X(j_*E)$ is $\sigma_X$-semistable for any $\sigma_Y$-semistable object $E\in \Ku(Y)$, we only need to show the inclusion
\[\gamma(|\cM^Y_{\sigma_Y}(a,b)|)\subset |\cM^X_{\sigma_X}(a,b)|\]
in $|\cM_{\mathrm{pug}}(\Ku(X)/\CC)|$. Then the result can be deduced from Theorem \ref{thm-very-general-GM4} and Chevalley's theorem on constructible subsets in algebraic stacks as follows.

First, we claim that there exists a smooth connected scheme $S$ of finite type over $\CC$ and a family of smooth ordinary GM fourfolds $\pi_4\colon \cX\to S$ with a closed subscheme $\cY\subset \cX$ such that $\pi_3\colon \cY\hookrightarrow \cX\to S$ is a family of smooth GM threefolds and for each general GM fourfold $X$ and its smooth hyperplane section $j\colon Y\hookrightarrow X$, there exists a point $s\in S$ such that $\cY_s\hookrightarrow \cX_s$ is isomorphic to $j$.

Indeed, let $\pi_1\colon \cX'\to W$ be a family of smooth GM fourfolds such that $W$ is smooth connected and of finite type over $\CC$ and any general smooth GM fourfold occurs as a fiber of $\pi_1$. Such a family exists by \cite[Proposition A.2]{kuznetsov2018derived}, since we can take $W$ to be a connected component of a smooth atlas of $\cM_4^{\mathrm{GM}}$ which dominants the moduli stack $\cM_4^{\mathrm{GM}}$ of smooth ordinary GM fourfolds, and $\pi_1$ be the pull-back of the universal family to $W$. Let $S$ be the open subscheme of the relative Hilbert scheme $\Hilb_{\cX'/W}^{\oh(1)}$ of hyperplane sections of $\cX'$ over $W$, parameterizing smooth hyperplane sections. We denote by $q\colon S\to W$ the natural morphism, which is of finite type. We define $\cX:=\cX'\times_W S$ and $\cY\subset \cX$ be the universal closed subscheme. As $S$ is an open subset of a $\PP^8$-bundle over $W$, we see $S$ is smooth connected and of finite type over $\CC$. Then the other statements of the claim follow from the construction. By replacing $W$ with its open dense subscheme and its covering, we can assume that $\pi_4\colon \cX\to S$ is a family of ordinary GM fourfolds satisfying two assumptions in \cite[Proposition 5.3]{perry2019stability}. 

By \cite[Lemma 5.9]{bayer2022kuznetsov}, there is an $S$-linear semi-orthogonal component $\Ku(\cX)\subset \mathrm{D}_{\perf}(\cX)$ such that $\Ku(\cX)_s\simeq \Ku(\cX_s)$ for any $s\in S$. Therefore, using the construction in \cite[Section 4]{perry2019stability}, there is a stability condition $\underline{\sigma}$ on $\Ku(\cX)$ over $S$ such that $\underline{\sigma}|_s\in \Stab^{\circ}(\Ku(\cX_s))$ for each $s\in S$. Similarly, by \cite[Theorem 23.1, Proposition 26.1]{BLMNPS21}, there exists a stability condition $\underline{\sigma'}$ on $\Ku(\cY)$ over $S$ in the sense of \cite[Definition 21.15]{BLMNPS21} such that $\underline{\sigma'}|_s$ is Serre-invariant for each $s\in S$. 

We denote by $\cM_{\mathrm{pug}}(\Ku(\cX)/S)$ the moduli stack of universally gluable objects in $\Ku(\cX)$ over $S$, defined in \cite[Definition 9.1]{BLMNPS21}. By Lemma \ref{lem-rhom-pushfor}, we have a morphism 
\[\gamma'\colon \cM^{\cY}_{\underline{\sigma'}}(a,b)\to \cM_{\mathrm{pug}}(\Ku(\cX)/S)\]
induced by the push-forward along the embedding $\cY\hookrightarrow \cX$ and the projection functor $\mathrm{D}_{\mathrm{perf}}(\cX)\to \Ku(\cX)$. We set $\cU:=\gamma'^{-1}(\cM^{\cX}_{\underline{\sigma}}(a,b))$, which is an open substack of $\cM^{\cY}_{\underline{\sigma'}}(a,b)$ by \cite[Lemma 21.12]{BLMNPS21}. Let $\mathcal{Z}$ be the complement of $|\cU|$ in $|\cM^{\cY}_{\underline{\sigma'}}(a,b)|$ and $p\colon \cM^{\cY}_{\underline{\sigma'}}(a,b)\to S$ be the natural morphism, which is of finite type by \cite[Theorem 21.24]{BLMNPS21}. Then we define $\mathcal{V}:=|\cM^{\cY}_{\underline{\sigma'}}(a,b)|\setminus p^{-1}(p(\mathcal{Z}))$, which is a constructible subset by \cite[Corollaire (5.9.2), Th\'eor\`eme (5.9.4)]{laumon:stack-book}. Let $V:=p(\cV)$, which is a constructible subset of $S$. As $p^{-1}(V)\cap \mathcal{Z}=\varnothing$, we know that for any $\CC$-point $$s\colon \Spec(\CC)\to V\subset S \subset  \Hilb_{\cX'/W}^{\oh(1)},$$ we have a GM fourfold $\cX_s$ and a smooth hyperplane section $j_s\colon \cY_s\hookrightarrow \cX_s$ such that $\pr_{\cX_s}(j_{s*}(E))$ is $\underline{\sigma}|_s$-semistable for any $\underline{\sigma'}|_s$-semistable object $E\in \Ku(\cY_s)$ of class $a\lambda_1+b\lambda_2$. 

We define a new constructible subset $T:=q(S\setminus q^{-1}(q(S\setminus V)))\subset W$. Since $q^{-1}(T)\subset V$, we know that for any $s\colon \Spec(\CC)\to T$, we have a GM fourfold $\cX'_s$ satisfying the statement of the theorem. By Theorem \ref{thm-very-general-GM4}, any very general GM fourfold represents a point in $T$, hence $T$ is dense in $W$. As $T$ is constructible and dense in $W$, it contains an open dense subscheme of $W$ by \cite[\href{https://stacks.math.columbia.edu/tag/005K}{Tag 005K}]{stacks-project} and the result follows.
\end{proof}

Using the same idea in the proof of Theorem \ref{thm-general-gm4}, but just replacing the roles of $X,Y$, and $\pr_X\circ j_*$ by $Y,X$, and $j^*$, respectively, we have:

\begin{theorem}\label{thm-general-gm3}
Let $a,b$ be a pair of integers. Then for a general GM threefold $Y$ and a general GM fourfold $X$ with an embedding $j \colon Y\hookrightarrow X$, if $F\in \Ku(X)$ is a $\sigma_X$-semistable object of class $a\Lambda_1+b\Lambda_2$, then $j^*F$ is $\sigma_Y$-semistable, where $\sigma_Y$ is a Serre-invariant stability condition on $\Ku(Y)$ and $\sigma_X\in \Stab^{\circ}(\Ku(X))$.
\end{theorem}

\begin{remark}\label{rmk-general-gm-varieties}
Once a construction of stability conditions for families of GM fivefolds and sixfolds is known, all results in Theorem \ref{thm-very-general-GM4} can be generalized to the general case by the above argument in Theorem \ref{thm-general-gm4}.
\end{remark}

\section{First application: Lagrangian families}\label{sec-lag-family}

In this section, we explain how to construct Lagrangian families of Bridgeland moduli spaces via Theorem \ref{thm-very-general-GM4} and \ref{thm-general-gm4}.

\subsection{Properties of moduli spaces}

Firstly, we prove some basic properties of moduli spaces of semistable objects in $\Ku(Y)$ for a GM threefold $Y$.

\begin{proposition}\label{prop-structure-3fold-moduli}
Let $Y$ be a GM threefold and $\sigma_Y$ be a Serre-invariant stability condition on $\Ku(Y)$. Let $a,b$ be a pair of coprime integers.

\begin{enumerate}
    \item The moduli space $M_{\sigma_Y}^Y(a,b)$ is normal of pure dimension $a^2+b^2+1$.

    \item If $Y$ is a general\footnote{Here, the generality assumption means that the statement holds for all $Y$ in an open dense subset of the moduli space of ordinary GM threefolds.} GM threefold, then $M_{\sigma_Y}^Y(a,b)$ is smooth and projective.
\end{enumerate}
\end{proposition}

\begin{proof}
Let $[F]\in M^Y_{\sigma_Y}(a,b)$. As $a,b$ are coprime, we know that $F$ is $\sigma_Y$-stable. If $F\neq S_{\Ku(Y)}(F)[-2]$, then $\Ext_Y^2(F,F)=0$ and $[F]\in M^Y_{\sigma_Y}(a,b)$ is a smooth point. If $F= S_{\Ku(Y)}(F)[-2]$, then following the notations in \cite[Proposition 6.1]{sacca:singular-K3-surface}, the same argument in the proof of \cite[Proposition 6.1]{sacca:singular-K3-surface} shows that there is a quiver $\overline{Q}$ and a dimension vector $\mathbf{n}$ such that the moment map
\[\mu\colon \mathrm{Rep}(\overline{Q}, \mathbf{n})\to  \mathfrak{gl}(\mathbf{n})^{\vee}\]
corresponds to the second-order term of the Kuranishi map $\kappa_2\colon \Ext^1(F,F)\to \Ext^2(F,F)$ under $\Aut(F)$-equivariant isomorphisms
\[\mathrm{Rep}(\overline{Q}, \mathbf{n})\cong \Ext^1(F,F), \quad \mathfrak{gl}(\mathbf{n})^{\vee}\cong \Ext^2(F,F).\]
And by \cite[Proposition 3.11]{zhao:remark-formality}, there is a local analytic $\Aut(F)$-equivariant isomorphism $$(M^Y_{\sigma_Y}(a,b), [F])\cong (\kappa^{-1}_2(0)\mathbin{/\mkern-5mu/}\Aut(F), 0)$$
of germs of analytic spaces. Therefore, we get a local 
$\Aut(F)$-equivariant isomorphism
$$(M^Y_{\sigma_Y}(a,b), [F])\cong (\mu^{-1}(0)\mathbin{/\mkern-5mu/}\Aut(F), 0)$$
and (1) follows from \cite[Corollary 1.4]{crawley:moment-map-quiver} and \cite[Theorem 1.1]{crawley:normality}.

Next, we prove (2). By \cite[Theorem 1.3(2)]{perry2019stability}, if $Y$ is the branch divisor of a non-Hodge-special GM fourfold, then (2) holds. Now consider a family $\pi\colon \cY\to S$ of smooth GM threefolds such that $S$ is a smooth connected scheme, of finite type over $\CC$ and each smooth ordinary GM threefold appears as a fiber in this family. Such $S$ exists by \cite[Proposition A.2]{kuznetsov2018derived}. By \cite[Lemma 5.9]{bayer2022kuznetsov}, there is a $S$-linear semi-orthogonal decomposition 
\[\mathrm{D}_{\mathrm{perf}}(\cY)=\langle \Ku(\cY), \pi^*(\mathrm{D}_{\mathrm{perf}}(S))\otimes \oh_{\cY}, \pi^*(\mathrm{D}_{\mathrm{perf}}(S))\otimes \cU^{\vee}_{\cY}\rangle\]
such that $\Ku(\cY)_s\simeq \Ku(\cY_s)$. From \cite[Corollary 26.2]{BLMNPS21}, or more precisely, \cite[Theorem 23.1, Proposition 26.1]{BLMNPS21}, there is a stability condition $\underline{\sigma}$ on $\Ku(\cY)$ over $S$ such that for any $s\in S$, the restriction $(\underline{\sigma})|_s$ is a stability condition on $\Ku(\cY_s)$ constructed in \cite{bayer2017stability}, hence Serre-invariant. 

Let $v\in \Knum(\Ku(\cY)/S)$ be the class such that $v_s=a\lambda_1+b\lambda_2\in \Knum(\Ku(\cY_s))$ for any $s\in S$. We define $$p\colon M_{\underline{\sigma}}(\Ku(\cY)/S, v)\to S$$ to be the good moduli space of the moduli stack $\cM_{\underline{\sigma}}(\Ku(\cY)/S, v)$ of families of geometrically $\underline{\sigma}$-stable objects of class $v$ over $S$ in the sense of \cite[Definition 21.11(1)]{BLMNPS21}. According to \cite[Theorem 21.24(3)]{BLMNPS21}, such moduli space exists and $p$ is proper. Then $p^{-1}(s)\cong M^{\cY_s}_{\underline{\sigma}|_s}(a,b)$ for each $s\in S$. By \cite[\href{https://stacks.math.columbia.edu/tag/052A}{Tag 052A}]{stacks-project}, after replacing $S$ with an open dense subset, we can assume that $p$ is flat. Therefore, we can further shrink $S$ to assume that $p$ is smooth as the generic fiber of $p$ is smooth by \cite[Theorem 1.3(2)]{ppzEnriques2023}. This proves the smoothness part of (2). Now the projectivity follows from \cite[Theorem 21.15]{BLMNPS21} and \cite[Corollary 3.4]{david:projective-criteria}.
\end{proof}

\subsection{Lagrangian subvarieties}
Let $X$ be a GM variety of dimension $4$ or $6$ and $j \colon Y \hookrightarrow X$ be a smooth hyperplane section. As before, let $\sigma_Y$ be a Serre-invariant stability condition on $\Ku(Y)$ and $\sigma_X\in \Stab^{\circ}(\Ku(X))$. For any pair of integers $a,b$, we denote by $M^X_{\sigma_X}(a,b)$ (resp.~$M^Y_{\sigma_Y}(a,b)$) the moduli space of $\sigma_X$ (resp.~$\sigma_Y$)-semistable objects in $\Ku(X)$ (resp.~$\Ku(Y)$) of class $a\Lambda_1 +b\Lambda_2$ (resp.~$a\lambda_1 +b\lambda_2$).
By Theorem \ref{thm-very-general-GM4} and \ref{thm-general-gm4}, we have a morphism $M_{\sigma_Y}^Y(a,b)\to M^X_{\sigma_X}(a,b)$ induced by $\pr_X\circ j_*$ when $X$ is general and $\dim X=4$, or $X$ is non-Hodge-special and $\dim X=6$. The following lemma describes the local properties of this morphism.

\begin{lemma} \label{prop_general-2}
    Let $E\in \Ku(Y)$ be a $\sigma_Y$-semistable object. 
\begin{enumerate}
    \item \emph{(Injectivity of tangent maps)} If $E$ is $\sigma_Y$-stable and $E\neq T_Y(E)$, then the natural map
    \[\Ext^1_Y(E,E)\to \Ext^1_X(\pr_X(j_*E), \pr_X(j_*E))\]
    is injective.

    \item \emph{(Fibers)} Let $E'\neq E\in \Ku(Y)$ be another $\sigma_Y$-semistable object with $\phi_{\sigma_Y}(E)=\phi_{\sigma_Y}(E')$. 

    \begin{enumerate}

    \item If $\pr_X(j_*E)\cong \pr_X(j_*E')$, then $E\oplus T_Y(E)$ is S-equivalent to $E'\oplus T_Y(E')$.

    \item If $E$ and $E'$ are $\sigma_Y$-stable, then
    \[\Hom_X(\pr_X(j_*E), \pr_X(j_*E'))\neq 0\]
if and only if $E'\cong T_Y(E)$ and the triangle 
\begin{equation}\label{split}
    T_Y(E) \to j^*\pr_X(j_*E) \to E
\end{equation}
is splitting. 
    \end{enumerate}

\end{enumerate}
    
\end{lemma}

\begin{proof}
Note that the exact triangle \eqref{split} can be obtained from applying the triangle \eqref{eq-triangle-functor-1} on $E$. Then applying $\Hom_Y(-, E)$ to \eqref{split}, we get an exact sequence
\[\Hom_Y(T_Y(E), E)\to \Ext^1_Y(E,E)\to \Ext^1_Y(j^*\pr_X(j_*E), E)=\Ext^1_X(\pr_X(j_*E), \pr_X(j_*E)).\]
If $E$ is $\sigma_Y$-stable, then $T_Y(E)$ is $\sigma_Y$-stable and has the same phase as $E$ by Lemma \ref{lem-T-fix-phase}, so part (1) follows.  


For part (2)(b), by Lemma \ref{lem-adjoint-functors} we have $\Hom_X(\pr_X(j_*E), \pr_X(j_*E'))=\Hom_Y(j^*\pr_X(j_*E), E')$, then the result also follows from applying $\Hom_Y(-, E')$ to the triangle \eqref{split}. For part (2)(a), as there is an isomorphism $\pr_X(j_*E)\cong \pr_X(j_*E')$, we get $j^*\pr_X(j_*E)\cong j^*\pr_X(j_*E')$. Then the result follows from the exact triangle \eqref{split}. 
\end{proof}

By \cite{perry2019stability}, 
the moduli space $M^X_{\sigma_X}(a,b)$ is a projective hyperk\"ahler manifold when $\gcd(a,b)=1$ and $\sigma_X\in \Stab^{\circ}(\Ku(X))$ is generic with respect to $a\Lambda_1+b\Lambda_2$. Moreover, as $\Ku(X)$ is 2-Calabi--Yau, the holomorphic two-form on $M^X_{\sigma_X}(a,b)$ at a point $[E]$ is given by the Yoneda pairing as in \cite{mukai:moduli-K3-I} (see also \cite{ku:symplectic} and \cite[Lemma 32.5]{BLMNPS21}):
\[\Ext^1_X(E,E)\times \Ext^1_X(E,E)\to \Ext^2_X(E,E)\xra{\mathrm{\cong }} \CC.\]

\begin{theorem} \label{thm-lagrangian-general}
Let $a,b$ be a pair of integers and $X$ be a general GM fourfold or non-Hodge-special GM sixfold with a given smooth hyperplane section $j \colon Y \hookrightarrow X$. Let $\sigma_Y$ be a Serre-invariant stability condition on $\Ku(Y)$ and $\sigma_X\in \Stab^{\circ}(\Ku(X))$. Then the functor $\pr_X\circ j_*$ induces a finite morphism
\[r\colon M^Y_{\sigma_Y}(a,b) \to M^X_{\sigma_X}(a,b).\]
Moreover, if $\gcd(a,b)=1$ and $\sigma_X\in \Stab^{\circ}(\Ku(X))$ is generic with respect to $a\Lambda_1+b\Lambda_2$, then 
\begin{enumerate}
    \item $r$ is generically unramified and its image is a Lagrangian subvariety in $M^X_{\sigma_X}(a,b)$, and $r$ is unramified when $Y$ is a general hyperplane section, and

    \item the restriction $r|_M$ of $r$ to any irreducible component $M$ of $M^Y_{\sigma_Y}(a,b)$ is either a birational map onto its image or a finite morphism of degree $2$ and \'etale on the smooth locus of $M$.

\end{enumerate}
\end{theorem}

\begin{proof}
By Theorem \ref{thm-very-general-GM4}, we have a morphism $r\colon M^Y_{\sigma_Y}(a,b) \to M^X_{\sigma_X}(a,b)$. Moreover, $r$ is quasi-finite by Lemma \ref{prop_general-2}(2)(a). As moduli spaces on both sides are proper by \cite[Theorem 21.24(3)]{BLMNPS21}, $r$ is finite (cf.~\cite[\href{https://stacks.math.columbia.edu/tag/0B1N}{Tag 0A4X}]{stacks-project}).

Now assume that $\gcd(a,b)=1$. As the space $M^Y_{\sigma_Y}(a,b)$ is reduced by Proposition \ref{prop-structure-3fold-moduli}(1), it is generically smooth, or in other words, $E\neq T_Y(E)$ for a general point $[E]\in M^Y_{\sigma_Y}(a,b)$. Then $r$ is generically unramified by Lemma \ref{prop_general-2}(1). Moreover, when $Y$ is a general hyperplane section, $M^Y_{\sigma_Y}(a,b)$ is smooth by Proposition \ref{prop-structure-3fold-moduli}, hence $r$ is unramified.

Let $U \subset M^Y_{\sigma_Y}(a,b)$ be the open dense locus of stable objects $E$ satisfying $E\neq T_Y(E)$, and let $L$ be the image of $r$. First of all, $L$ is reduced because $M^Y_{\sigma_Y}(a,b)$ is reduced. Hence, $L$ is generically smooth since we work over $\CC$. Moreover, Lemma \ref{prop_general-2}(2) implies that $r|_U$ is of relative dimension zero, thus $\dim \overline{U}=\dim L$. Furthermore, by Lemma \ref{prop_general-2}(1), the induced morphism $r|_U$ is injective at the level of tangent maps. Therefore, shrinking $U$ if necessary, we can assume $r|_U\colon U\to r(U)$ is \'etale. Hence, the tangent space at any point in $r(U)$ can be naturally identified with $\Ext^1_Y(E,E)$ for some $[E]\in U$.

Since $\dim \overline{U}=\dim L=\frac{1}{2}\dim M^X_{\sigma_X}(a,b)$, to show that $L$ is Lagrangian, it suffices to prove that the canonical holomorphic symplectic two-form of $M_{\sigma_X}^X(a,b)$ 
becomes zero after restricting to $L$. Because $r(U)$ is dense in $L$, this is equivalent to proving that the restricting two-form vanishes at every point of $r(U)$.

For $[F]\in M_{\sigma_X}^X(a,b)$, $\mathrm{Ext}^1_X(F,F)$ is identified with the tangent space $T_{[F]} M_{\sigma_X}^X(a,b)$. Then the holomorphic symplectic two-form of $M_{\sigma_X}^X(a,b)$ is naturally given by the Yoneda pairing
$$\mathrm{Ext}^1_X(F,F)\times\mathrm{Ext}^1_X(F,F)\xrightarrow{\phi }\mathrm{Ext}^2_X(F,F)\cong\mathrm{Hom}_X(F,F)^{\vee}\cong \CC.$$ 
This means that for any $x, y\in \mathrm{Ext}^1_X(F,F)$, the map $\phi(x,y): F\to F[2]$ is the composition map $x[1]\circ y$.

Now we assume that $[F]=r([E])$ for an object $[E]\in U$. We know that the functor $\pr_X\circ j_*$ induces an embedding $d_{[E]}\colon \Ext^1_Y(E,E)\hookrightarrow \Ext^1_X(F, F)$ by Lemma \ref{prop_general-2}(1), mapping an element $a \colon E\to E[1]$ to $\pr_X\circ j_*(a)\colon F\to F[1]$. For any two elements $x, y\in d_{[E]}(\Ext^1_Y(E, E))=T_{[F]}L$, we can express $x=\pr_X\circ j_*(a)$ and $y=\pr_X\circ j_*(b)$, where $a,b \in \Ext^1_Y(E,E)$. Then the restriction of $\phi$ on $d_{[E]}(\Ext^1_Y(E, E))=T_{[F]}L$ is given by $$\phi(x, y)=x[1]\circ y=\pr_X\circ j_*(a[1]\circ b)\in \pr_X\circ j_*(\Ext^2_Y(E, E)).$$ 
Since $\Ext^2_Y(E, E)=0$, $\phi$ vanishes at every point of $r(U)$, and the desired result in part (1) follows.

Next, let $M$ be an irreducible component of $M^Y_{\sigma_Y}(a,b)$ and $M_{sm}$ be its smooth locus. Recall that the natural involution $T_X$ on $\Ku(X)$ induces an involution on the moduli space $M^X_{\sigma_X}(a,b)$, which is also denoted by $T_X\colon M^X_{\sigma_X}(a,b)\to M^X_{\sigma_X}(a,b)$. Let $\mathrm{Fix}(T_X)\subset M^X_{\sigma_X}(a,b)$ be the fixed locus, which is a smooth subvariety by \cite[Theorem 1.6]{ppzEnriques2023}. If $r(M)$ is contained in $\mathrm{Fix}(T_X)$, then $F:=r(M)$ is a connected component of $\mathrm{Fix}(T_X)$ as $\dim (r(M))=M=\mathrm{Fix}(T_X)$ and $M$ is reduced. By Lemma \ref{prop_general-2}(2), we know that $r|_M\colon M\to r(M)$ has degree $2$. Now we aim to show that $r|_{M_{sm}}$ is \'etale. Note that $r(M_{sm})\subset F$ as $M_{sm}$ is already reduced. Hence, we have a morphism $r|_{M_{sm}}\colon M_{sm}\to F$. As $M_{sm}$ and $F$ are both smooth, from Lemma \ref{prop_general-2}(1), we conclude that $r|_{M_{sm}}$ is \'etale.

If $r(M)$ is not contained in $\mathrm{Fix}(T_X)$, then there is a smooth open dense subset $$U':=(r|_M)^{-1}(r(M)\setminus \mathrm{Fix}(T_X))\cap M_{sm}\subset M$$ such that $r(U')\cap \mathrm{Fix}(T_X)=\varnothing$. Using Lemma \ref{prop_general-2}, we know that $r|_{U'}$ is an immersion, then $r$ is birational onto its image.
\end{proof}

Therefore, we can construct a family of Lagrangian subvarieties for each moduli space $M^X_{\sigma_X}(a,b)$ using hyperplane sections $Y\subset X$. The next proposition demonstrates that this construction coincides with the one presented in \cite{ppzEnriques2023} when $X$ is special and $Y$ is the branch divisor of $X$.

\begin{proposition}
Let $X$ be a special GM fourfold and $j\colon Y\hookrightarrow X$ be its branch divisor. Let $\sigma_Y$ be a Serre-invariant stability condition on $\Ku(Y)$ and $\sigma_X\in \Stab^{\circ}(\Ku(X))$. Then for any pair of integers $a,b$, the functor $\pr_X\circ j_*$ induces a finite morphism
\[r\colon M^Y_{\sigma_Y}(a,b) \to M^X_{\sigma_X}(a,b)\]
such that $r$ is an \'etale double cover onto $\mathrm{Fix}(T_X)$.
\end{proposition}

\begin{proof}
Recall that in \cite[Theorem 1.8]{ppzEnriques2023}, the authors show that $\bL_{\oh_X}\bL_{\cU^{\vee}_X}(j_*E(H))$ is $\sigma_X$-semistable for any $\sigma_Y$-semistable object $E\in \Ku(Y)$ (cf.~\cite[Equation (8.3)]{ppzEnriques2023}). Then it induces a morphism $r'\colon M^Y_{\sigma_Y}(a,b) \to M^X_{\sigma_X}(a,b)$, which maps onto $\mathrm{Fix}(T_X)$. Moreover, $r'([E])=[\bL_{\oh_X}\bL_{\cU^{\vee}_X}(j_*E(H))]$. Using Lemma \ref{lem-involution}, it is easy to check that
\[T_X(\bL_{\oh_X}\bL_{\cU^{\vee}_X}(j_*E(H)))\cong \pr_X(j_*E)[1],\]
hence $T_X\circ r'=r$, where $T_X\colon M^X_{\sigma_X}(a,b)\to M^X_{\sigma_X}(a,b)$ is the induced involution on the moduli space.
\end{proof}

Using Theorem \ref{thm-very-general-GM4} and Lemma \ref{lem-exact-triangle}, the above results can be easily generalized to GM fivefolds: 

\begin{corollary}\label{cor-lagrangian-very-general-5fold}
Let $X$ be a non-Hodge-special GM fourfold and $Y$ be a GM fivefold containing $X$ as a hyperplane section with the embedding $j \colon X\hookrightarrow Y$. Let $\sigma_X$ be a stability condition on $\Ku(X)$ and $\sigma_Y$ be a Serre-invariant stability condition on $\Ku(Y)$. Then for any pair of integers $a,b$, the functor $j^*$ induces a finite morphism
\[r \colon M^Y_{\sigma_Y}(a,b) \to M^X_{\sigma_X}(a,b).\]
Moreover, if $\gcd(a,b)=1$, then
\begin{enumerate}
    \item $r$ is generically unramified and its image is a Lagrangian subvariety in $M^X_{\sigma_X}(a,b)$, and

    \item the restriction $r|_M$ of $r$ to any irreducible component $M$ of $M^Y_{\sigma_Y}(a,b)$ is either a birational map onto its image or a finite morphism of degree $2$ and \'etale on the smooth locus of $M$.
\end{enumerate}

\end{corollary}

\subsection{Lagrangian families}

Using Theorem \ref{thm-lagrangian-general}, we obtain infinitely many Lagrangian subvarieties of $M^X_{\sigma_X}(a,b)$, parameterized by an open subset of $|\oh_X(H))|$ when $\dim X$ is even. These subvarieties can form a family using stability conditions in families and relative moduli spaces. We start with a general definition of Lagrangian families, which can be viewed as a generalization of Lagrangian fibrations.

\begin{definition}\label{def-lag-family}
A \emph{Lagrangian family} of a projective hyperk\"ahler manifold $M$ is a diagram 
\[\begin{tikzcd}
	L & M \\
	B
	\arrow["p"', from=1-1, to=2-1]
	\arrow["q", from=1-1, to=1-2]
\end{tikzcd}\]
where $p$ is flat and projective, $L$ and $B$ are quasi-projective manifolds, such that for a general point $b\in B$, $q|_{p^{-1}(b)}$ is generically finite to its image, which is a Lagrangian subvariety of $M$. The manifold $B$ is called the \emph{base} of a Lagrangian family. A Lagrangian family is called a \emph{Lagrangian covering family} if $q$ is dominant.
\end{definition}

\begin{remark}
There is another definition of Lagrangian covering families in \cite[Definition 0.4]{voisin2021lefschetz}, which does not assume $p$ to be flat but requires $B$ to be projective. It is easy to see that the existence of a Lagrangian family in the sense of Definition \ref{def-lag-family} implies the existence of one in the sense of \cite[Definition 0.4]{voisin2021lefschetz} by compactifying $p$ in Definition \ref{def-lag-family} to $\overline{p}\colon \overline{L}\to \overline{B}$, resolving singularities of $\overline{L}$ by $\wt{L}$, and then resolving the indeterminacy locus of $\wt{L}\dashrightarrow M$.
\end{remark}

Now let $\PP_0\subset |\oh_X(H)|$ be the open locus parameterizing smooth divisors. Using relative moduli spaces and Theorem \ref{thm-general-gm4}, we can construct a smooth Lagrangian family for each $M^X_{\sigma_X}(a,b)$ over an open subset of $\PP_0$ when $X$ is a general GM fourfold.

\begin{theorem}\label{thm-lagrangian-family}
Let $a,b$ be a pair of coprime integers. Then for any general GM fourfold $X$ and a stability condition $\sigma_X\in \Stab^{\circ}(\Ku(X))$ generic with respect to $a\Lambda_1+b\Lambda_2$, there exists an open subscheme $V\subset |\oh_X(H)|$ and a relative moduli space of stable objects $$p\colon L(a,b)\to V$$ with a morphism $$q\colon L(a,b)\to M^X_{\sigma_X}(a,b),$$ such that $L(a,b)$ is a quasi-projective manifold, $p$ is smooth and projective, and the following diagram is a Lagrangian family
\[\begin{tikzcd}
	{L(a,b)} & {M^X_{\sigma_X}(a,b)} \\
	V\subset \PP^8
	\arrow["q", from=1-1, to=1-2]
	\arrow["p"', from=1-1, to=2-1]
\end{tikzcd}\]
Furthermore, $q|_{p^{-1}(s)}$ is finite and unramified for each $s\in V$.
\end{theorem}

\begin{proof}
Let $V$ be an open subset of $\PP_0\subset |\oh_X(H)|$ such that for every $[Y]\in V$, the moduli space $M^Y_{\sigma_Y}(a,b)$ is smooth. Such $V$ exists by Proposition \ref{prop-structure-3fold-moduli} as $X$ is general. We denote the restriction of the universal hyperplane on $V$ by $\cH_V$. We define $\oh_{\cH_V}(1):=p_2^*\oh_X(H)|_{\cH_V}$, where $p_2\colon |\oh_X(H)|\times X\to X$ is the projection. Then $\oh_{\cH_V}(1)|_{\cH_s}\cong \oh_{\cH_s}(H)$ for any $s\in V$. In other words, $\pi\colon \cH_V\to V$ is a family of GM threefolds in the sense of \cite[Section 5.2]{bayer2022kuznetsov}. As in the proof of Proposition \ref{prop-structure-3fold-moduli}, we have a relative Kuznetsov component $\Ku(\cH_V)\subset \mathrm{D}_{\mathrm{perf}}(\cH_V)$ and a stability condition $\underline{\sigma}$ on $\Ku(\cH_V)$ over $V$ such that for any $s\in V$, $(\underline{\sigma})|_s$ is Serre-invariant.

Let $v\in \Knum(\Ku(\cH_V)/V)$ be the class such that $v_s=a\lambda_1+b\lambda_2\in \Knum(\Ku(\cH_s))$ for any $s\in V$. We define $$p\colon L(a,b):=M_{\underline{\sigma}}(\Ku(\cH_V)/V, v)\to V$$ to be the relative good moduli space. According to \cite[Theorem 21.24(3)]{BLMNPS21}, such moduli space exists and $p$ is proper. By Theorem \ref{thm-general-gm4}, the functor $\pr_X\circ j_*$ induces a morphism $q\colon L(a,b)\to M^X_{\sigma_X}(a,b)$.


From our construction and Theorem \ref{thm-lagrangian-general}, $p^{-1}(s)=M^{\cH_s}_{(\underline{\sigma})|_s}(a,b)$ is smooth projective and $q|_{p^{-1}(s)}$ is finite and unramified with the image Lagrangian for each $s\in V$. Applying \cite[Lemma 5.1]{ppzEnriques2023} and \cite[Lemma 21.12]{BLMNPS21} implies the smoothness of $p$. Thus, $L(a,b)$ is smooth as well. Finally, the projectivity of $p$ follows from the smoothness, \cite[Theorem 21.25]{BLMNPS21} and \cite[Corollary 3.4]{david:projective-criteria}, which implies that $L(a,b)$ is a quasi-projective manifold. By definition, this yields a smooth Lagrangian family of $M^X_{\sigma_X}(a,b)$.
\end{proof}


Indeed, the results above can be generalized to a non-Hodge-special GM sixfold $X$ for an open subset $V\subset |\oh_X(H))|=\PP^{10}$ if a family of Serre-invariant stability conditions on $\Ku(\cH_V)$ over $V$ is known. This can be done by establishing a generalization of \cite[Corollary 6.5]{kuznetsov2019categorical} for families of GM fivefolds. Since the proof is analogous to the GM fourfolds case, we state the corresponding results here for later use.

\begin{corollary}\label{cor-lagrangian-family-6fold}
Let $X$ be a non-Hodge-special GM sixfold and $\sigma_X$ be a stability condition on $\Ku(X)$. Assuming there exists an open subset $V\subset |\oh_X(H))|$ and a stability condition $\underline{\sigma}$ on $\Ku(\cH_V)$ over $V$ such that $(\underline{\sigma})|_s$ is a Serre-invariant stability condition on $\Ku(\cH_s)$ for any $s\in V$, then the claim of Theorem \ref{thm-lagrangian-family} holds for the moduli space $M^X_{\sigma_X}(a, b)$. 

\end{corollary}

A similar result can be obtained when considering a fixed GM fourfold within a family of GM fivefolds.

\begin{corollary}\label{cor-lagrangian-family-5fold}
Let $X$ be a non-Hodge-special GM fourfold and $\sigma_X$ be a stability condition on $\Ku(X)$. Given a smooth connected quasi-projective manifold $V$ and a family of GM fivefolds $\cH_V\to V$ containing $X$. Assume that there is a stability condition $\underline{\sigma}$ on $\Ku(\cH_V)$ over $V$ such that $(\underline{\sigma})|_s$ is a Serre-invariant stability condition on $\Ku(\cH_s)$ for any $s\in V$. Then the claim of Theorem \ref{thm-lagrangian-family} holds for the moduli space $M^X_{\sigma_X}(a, b)$. 

\end{corollary}

\subsection{Lagrangian covering families}\label{subsec-lag-cover}

In Theorem \ref{thm-lagrangian-family}, we construct a Lagrangian family for each $M^X_{\sigma_X}(a,b)$. A natural question arises regarding whether $q$ is dominant, indicating the presence of a Lagrangian covering family. A straightforward computation of dimension shows that if $q$ is dominant, then $\dim M^X_{\sigma_X}(a,b)\leq 16$. So up to sign, we have
\[(a,b)=(0, 1), (1,0), (1,\pm 1), (1, \pm 2), (2, \pm 1).\]
At this moment, we do not know an abstract method to demonstrate that the Lagrangian family in Theorem \ref{thm-lagrangian-family} constitutes a Lagrangian covering family. However, this can be established for the first three classes as follows, since we have a specific description of the corresponding moduli spaces: 

\begin{itemize}
    \item $(a,b)=(0,1)$ or $(1,0)$. By \cite{GLZ2021conics}, the corresponding moduli space on a very general GM fourfold is the double EPW sextic or the double dual EPW sextic. On the other hand, the corresponding moduli space on a GM threefold is studied in \cite{JLLZ2021gushelmukai}. We will see in \cite{FGLZcube} that Theorem \ref{thm-lagrangian-family} gives a Lagrangian covering family which agrees with the classical construction in \cite[Section 5.1]{iliev2011fano}.

    \item $(a,b)=(1,\pm 1)$. By \cite{kapustka2022epw}, the corresponding moduli space on a very general GM fourfold is the double EPW cube constructed in \cite{IKKR19}. In \cite{FGLZcube}, we will realize it as the MRC quotient of the Hilbert scheme of twisted cubics on a GM fourfold. Moreover, we will apply Theorem~\ref{thm-lagrangian-family} to construct Lagrangian covering families of double EPW cubes in \cite{FGLZcube}. These constructions are new in the literature.

    \item $(a,b)=(1,\pm 2)$ or $(2, \pm 1)$. In this case, the corresponding moduli spaces on GM fourfolds are $12$-dimensional hyperk\"ahler varieties and have not been extensively studied using classical methods. According to  Corollary \ref{cor-rational-embedding}, the moduli space for $(a,b)=(1,2)$ on a very general GM fourfold $X$ contains a subvariety birational to $X$. On the other hand, the moduli space for $(a,b)=(1,2)$ on a GM threefold is studied in \cite{jLz2021brillnoether}, which contains such threefold as a Brill--Noether locus. We expect that both moduli spaces can be constructed from Hilbert schemes of certain curves, similar to \cite{bayer2020desingularization} and \cite{li2018twisted}. If this is the case, then it is possible to demonstrate that the family in Theorem \ref{thm-lagrangian-family} in this scenario indeed constitutes a covering family.
\end{itemize}

A similar discussion as above for cubic fourfolds has been included in Appendix \ref{App-cubic}.

\section{Further applications}\label{subsec-application}
In this section, we delve into further applications of studying GM fourfolds through their hyperplane sections.

\subsection{Embedding Gushel--Mukai fourfolds into hyperk\"ahler manifolds}

Using Theorem \ref{thm-very-general-GM4}, we can construct a rational embedding $X\dashrightarrow M^X_{\sigma_X}(1,2)$ which confirms an expectation in \cite[Section 5.4.3]{perry2019stability}. We denote by $\CC_x$ the skyscraper sheaf at a point $x \in X$.  

\begin{corollary}\label{cor-rational-embedding}
Let $X$ be a general GM fourfold or non-Hodge-special special GM fourfold. Then for $\sigma_X\in \Stab^{\circ}(\Ku(X))$ generic with respect to $\Lambda_1+2\Lambda_2$, the functor $\pr_X$ induces a rational map
\begin{align*}
    X & \dashrightarrow M_{\sigma_X}^{X}(1,2) \\
    x & \mapsto \pr_X(\CC_x)
\end{align*}
which is generically an immersion.
\end{corollary}

\begin{proof}
For any point $x\in Y$ on a smooth GM threefold $j\colon Y\hookrightarrow X$, we know that $\pr_Y(\mathbb{C}_x)$ is of class $[\pr_Y(\mathbb{C}_x)]=\lambda_1+2\lambda_2$ and stable with respect to any Serre-invariant stability condition on $\Ku(Y)$ by \cite[Theorem 5.9]{jLz2021brillnoether}. Then Theorem \ref{thm-general-gm4} and Corollary \ref{cor-push-stable} imply that $\pr_X(j_*\pr_Y(\CC_x))$ is $\sigma_X$-stable. Moreover, by Proposition \ref{prop-pushforward}, we have $\pr_X(\CC_x)\cong \pr_X(j_*\pr_Y(\CC_x))$. 

Let $W\subset X$ be an open subset such that each $x\in W$ is contained in a smooth hyperplane section $Y$ and  
\begin{equation}\label{eq-T3-proj-point}
    \pr_Y(\mathbb{C}_x)\neq T_Y(\pr_Y(\mathbb{C}_x)). 
\end{equation}
If $X$ is ordinary, then \eqref{eq-T3-proj-point} holds for any point $x \in Y$ by \cite[Lemma 6.6, Theorem 1.1]{jLz2021brillnoether}. But if both $X$ and $Y$ are special, the involution $T_Y$ on $\Ku(Y)$ is induced by the geometric involution of the double cover on $Y$. Thus, \eqref{eq-T3-proj-point} holds when $x\in Y$ is not contained in the branch divisor of $X$, as the double cover structures on $X$ and $Y$ are compatible. 

Thus, we obtain a morphism $W\to M^X_{\sigma_X}(1,2)$, which is injective by next Lemma \ref{lem-proj-point-injective}, and unramified by Lemma \ref{prop_general-2} and \eqref{eq-T3-proj-point}. Hence, it is an immersion.

\end{proof}


\begin{lemma}\label{lem-proj-point-injective}
Let $X$ be a GM fourfold and $x, x'\in X$ be two closed points. Then $\pr_X(\CC_x)\cong \pr_X(\CC_{x'})$ if and only if $x=x'$.
\end{lemma}

\begin{proof}
Let $x\in X$ be a closed point. By definition, we have an exact triangle $\cU_X^{\vee \oplus 2}\to \CC_x\to \bL_{\cU^{\vee}_X} \CC_x$, which yields the exact triangle 
\begin{equation}\label{eq-prcx-1}
    \cQ^{\vee \oplus 2}_X\xra{a} I_x\to \bL_{\oh_X}\bL_{\cU^{\vee}_X} \CC_x[-1] =: F_x[1]
\end{equation}
after applying $\bL_{\oh_X}$. Note that $a$ corresponds to two linearly independent maps in $\Hom_X(\cU^{\vee}_X, \CC_x)=\CC^2$, thus its image is the zero locus of two linearly independent sections of $\cQ_X$, i.e.~$\Gr(0,3)\cap X$. Therefore, $a$ is surjective, and $F_x$ is a sheaf. 

Since $\bR_{\cU_X}\bR_{\oh_X(-H)}\cQ^{\vee}_X=\cQ^{\vee}_X$, applying $\bR_{\cU_X}\bR_{\oh_X(-H)}$ to \eqref{eq-prcx-1}, we obtain
\begin{equation}\label{eq-prcx-2}
    \cQ^{\vee \oplus 2}_X\to \bR_{\cU_X}\bR_{\oh_X(-H)}I_x\to \pr_X(\CC_x)[-1].
\end{equation}
Considering $\RHom_X(I_x, \oh_X(-H))=\CC[-3]$, 
there is an exact triangle
\begin{equation}\label{eq-prcx-22}
    \bR_{\cU_X}\bR_{\oh_X(-H)}I_x\to \bR_{\cU_X}I_x\to \bR_{\cU_X}\oh_X(-H)[3] =: K[2].
\end{equation}
Here $K
$ is the cokernel of the natural map $\oh_X(-H)\hookrightarrow \cU_X^{\oplus 5}=\cU_X\otimes H^0(\cU^{\vee}_X)$. Since  $\RHom_X(I_x, \cU_X)=\CC^2[-3]$, we have  the exact triangle $\cU_X^{\oplus 2}[2] \to  \bR_{\cU_X}I_x \to I_x$. Thus, taking cohomology from \eqref{eq-prcx-22} yields a long exact sequence
\begin{equation}
0\to \cH^{-2}(\bR_{\cU_X}\bR_{\oh_X(-H)}I_x)\to \cU_X^{\oplus 2}\xra{b} K\to \cH^{-1}(\bR_{\cU_X}\bR_{\oh_X(-H)}I_x)\to 0    
\end{equation}
and $\cH^{0}(\bR_{\cU_X}\bR_{\oh_X(-H)}I_x)\cong I_x$. Then, taking cohomology from \eqref{eq-prcx-2} implies that $$\cH^{-2}(\bR_{\cU_X}\bR_{\oh_X(-H)}I_x) = \cH^{-2}(\pr_X(\CC_x)[-1])$$ and gives the long exact sequence 
\begin{equation}\label{le}
    0 \to \cH^{-1}(\bR_{\cU_X}\bR_{\oh_X(-H)}I_x) \to  \cH^{-1}(\pr_X(\CC_x)[-1]) \to \cQ^{\vee \oplus 2}_X\xra{a} I_x\to  \cH^{0}(\pr_X(\CC_x)[-1]) \to 0. 
\end{equation}
Since $a$ is surjective, $ \cH^{0}(\pr_X(\CC_x)[-1]) = 0$ and its kernel is the sheaf $F$ defined in \eqref{eq-prcx-1}. Consequently, from \eqref{le}, we derive the long exact sequence
\begin{equation}\label{b}
    0\to \cH^{-2}(\pr_X(\CC_x)[-1])\to \cU_X^{\oplus 2}\xra{b} K \to \cH^{-1}(\pr_X(\CC_x)[-1])\to F_x\to 0.
\end{equation}
Now we consider two cases depending on the morphism $b$:

\medskip

\textbf{Case I.} First assume the map $b$ is injective, then $\cH^{-2}(\pr_X(\CC_x)[-1]) = 0$ and $\pr_X(\CC_x) = \cH^{-2}(\pr_X(\CC_x))[2]$. The injective map $b$ in \eqref{b} induces the following commutative diagram of exact triangles 
\[\begin{tikzcd}
	 & \cU_X^{\oplus 2}& \cU_X^{\oplus 2} \\
	\oh_X(-H) & \cU_X^{\oplus 5} & K \\
	\oh_X(-H) & \cU_X^{\oplus 3} & G_x
	\arrow[from=3-1, to=3-2]
	\arrow[shift right, no head, from=2-1, to=3-1]
 \arrow[no head, from=2-1, to=3-1]
	\arrow[from=2-1, to=2-2]
 \arrow[from=2-2, to=2-3]
 \arrow[from=1-2, to=2-2]
 \arrow[from=2-2, to=3-2]
	\arrow[from=3-2, to=3-3]
	\arrow[from=2-3, to=3-3]
	\arrow["b", from=1-3, to=2-3]
 \arrow[no head, from=1-2, to=1-3]
	\arrow[shift right, no head, from=1-2, to=1-3]
\end{tikzcd}\]
By dualizing the third row, we get $\mathcal{E}xt_X^i(G_x, \oh_X)=0$ for any $i>1$ and $\mathcal{E}xt_X^1(G_x, \oh_X)$ is supported at a point. Hence, $G_x$ is reflexive, and it is straightforward to check that $G_x$ is slope stable with $\ch_0(G_x) = 5$, $\ch_1(G_x) = -2H$, $\ch_2(G_x).H^2 = -2$ and $\ch_3(G_x).H = \frac{8}{3}$.  Similarly, one can check that the sheaf $F_x$ in \eqref{eq-prcx-1} is slope stable with $\ch_0(F_x) = 5$, $\ch_1(F_x) = -2H$, $\ch_2(F_x).H^2 = -2$ and $\ch_3(F_x).H = \frac{2}{3}$. This proves that $\Hom_X(G_x, F_{x'}) = 0$ for any closed point $x' \in X$. Now assume for a contradiction that there is a closed point $x' \neq x$ on $X$ such that $\pr_X(\CC_x) \cong \pr_X(\CC_{x'})$, then by \eqref{b} we have the exact triangle $G_x \to \pr_X(\CC_{x'})[-2] \to F_x$ of coherent sheaves. Thus the surjection $\pr_X(\CC_{x'})[-2] \twoheadrightarrow F_{x'}$ implies 
$\Hom_X(F_x, F_{x'}) \neq 0$
and so $F_x \cong F_{x'}$ which is not possible as they have different non-locally free loci.

\medskip

\textbf{Case II.} Now assume the map $b$ in \eqref{b} is not injective, then we claim either $\ker(b)=\cU_X$ or $b$ is zero.
Note that any map $\cU^{\oplus m}_X\to \oh_X$ corresponding to $4\leq m\leq 5$ linearly independent sections of $\cU^{\vee}_X$ is surjective, as its image is the ideal sheaf of the zero locus of these $m$ sections, i.e.~$\Gr(2, 5-m)\cap X$. Therefore, dualizing the exact sequence $0\to \oh_X(-H)\to \cU^{\oplus 5}_X\to K\to 0$, we find that $\mathcal{E}xt_X^i(K, \oh_X)=0$ for any $i>0$, hence $K$ is locally free. One can easily check that $K$ is slope stable with $\ch_0(K) = 9$ and $\ch_1(K) = -4H$. Hence if the map $b$ in \eqref{b} is non-zero, ordering of slopes forces $\ch_0(\ker b) =2$ and $\ch_1(\ker b) = -H$ and so $\ker b \cong \cU_X$ by the stability of $\cU_X$.

\medskip

As $\mathcal{E}xt_X^3(I_x, \oh_X)=\CC_x$, we know that $\mathcal{E}xt_X^2(F_x, \oh_X)=\CC_x$. We claim that $$\mathcal{E}xt_X^1(G, \oh_X)=\mathcal{E}xt_X^2(G, \oh_X)=0,$$ where $G:=\mathrm{cok}(b)$, which implies $\mathcal{E}xt_X^2(\cH^{-1}(\pr_X(\CC_x)[-1]), \oh_X)=\CC_x$. To this end, when $b$ is zero, we have $G=K$, hence the result follows from the fact that $K$ is a bundle. When $\ker(b)=\cU_X$, from the definition of $K$, we have an exact sequence $0\to \oh_X(-H)\to \cU_X^{\oplus 4}\to G\to 0$, where $\oh_X(-H)\to \cU_X^{\oplus 4}$ corresponds to $4$ linearly independent sections of $\cU^{\vee}_X$. Hence by dualizing, we know that $\mathcal{E}xt_X^i(G, \oh_X)=0$ for any $i>0$, and the claim follows.

Now, given a point $x' \neq x$ on $X$, if $\pr_X(\CC_x)\cong \pr_X(\CC_{x'})$, then we have $\cH^{-1}(\pr_X(\CC_x))\cong \cH^{-1}(\pr_X(\CC_{x'}))$, which implies
\[\CC_x=\mathcal{E}xt_X^2(\cH^{-1}(\pr_X(\CC_x)[-1]), \oh_X)\cong \mathcal{E}xt_X^2(\cH^{-1}(\pr_X(\CC_{x'})[-1]), \oh_X)=\CC_{x'},\]
leading to a contradiction.

\end{proof}


\begin{remark}
By Theorem \ref{thm-very-general-GM4}, we actually have a rational map
\[X\dashrightarrow M_{\sigma_X}^{X}([\pr_X(\CC_x)])\]
for a very general GM variety $X$ of dimension $n\geq 4$ and any Serre-invariant stability condition $\sigma_X$ on~$\Ku(X)$, where $M_{\sigma_X}^{X}([\pr_X(\CC_x)])$ is the moduli space of $\sigma$-semistable objects with class $[\pr_X(\CC_x)]$. However, we do not have a proof for the stability of $\pr_X(\CC_x)$ for $n=5,6$ as $[\pr_X(\CC_x)]$ is not primitive. We expect that $\pr_X(\CC_x)$ is stable for a general point $x\in X$, which implies that the rational map above is generically an immersion,  as in Corollary \ref{cor-rational-embedding}.
\end{remark}

\subsection{Projectivity of Bridgeland moduli spaces of semistable objects}

Using Theorem \ref{thm-very-general-GM4}, we can also achieve projectivity of moduli spaces of semistable objects in $\Ku(Y)$ of any class for a general GM threefold~$Y$, which improves \cite[Theorem 1.3(2)]{ppzEnriques2023} and confirms the expectation in \cite[Remark 1.4]{ppzEnriques2023} for general (ordinary or special) GM threefolds.

\begin{corollary}\label{cor-projective}
Let $a,b$ be a pair of integers and $Y$ be a general GM threefold. Then for any Serre-invariant stability condition $\sigma_Y$ on $\Ku(Y)$, the moduli space of $\sigma_Y$-semistable objects $M^Y_{\sigma_Y}(a,b)$ is a projective scheme.
\end{corollary}

\begin{proof}
Assume that $Y$ is a general GM threefold, in the sense that $Y$ is contained in a general ordinary GM fourfold $X$ in the sense of Theorem \ref{thm-general-gm4} or a non-Hodge-special GM fourfold. By Theorem \ref{thm-lagrangian-general}, we have a finite morphism $M^Y_{\sigma_Y}(a,b) \to M^X_{\sigma_X}(a,b)$ for $\sigma_X\in \Stab^{\circ}(\Ku(X))$. So it suffices to prove $M^X_{\sigma_X}(a,b)$ is a projective scheme.

We can assume that $\sigma_X\in \Stab^{\circ}(\Ku(X))$ is generic with respect to $a\Lambda_1+b\Lambda_2$. When $(a,b)\neq (\pm 2,0)$ or $(0,\pm 2)$, we know that $M^X_{\sigma_X}(a,b)$ is projective by the same argument as in \cite[Theorem 3.1]{sacca:moduli-ku} using the corresponding results in \cite{perry2019stability} instead of  \cite{BLMNPS21} (cf.~\cite[Section 8]{sacca:moduli-ku}). When $(a,b)= (\pm 2,0)$ or $(0,\pm 2)$, or in other words, the class is of OG10-type in the sense of \cite{sacca:moduli-ku}, the projectivity of $M^X_{\sigma_X}(a,b)$ also follows from the same argument in \cite[Section 3.7]{li2020elliptic} by replacing results in \cite{BLMNPS21} with \cite{perry2019stability}.
\end{proof}

\subsection{Lagrangian constant cycle subvarieties} Finally, we can construct a Lagrangian constant cycle subvariety for each Bridgeland moduli space of stable objects. We refer to \cite[Section 3]{ lin:constant-cycle-lag-fibration} for the definition and basic properties of constant cycle subvarieties. 

\begin{definition}\label{def-constant cycle}
For a projective hyperk\"ahler manifold $M$, a subvariety $Z\subset M$ is called a \emph{Lagrangian constant cycle subvariety} of $M$ if $Z$ is Lagrangian and all points of $Z$ are rationally equivalent in $M$.
\end{definition}

It is conjectured in \cite[Conjecture 0.4]{voisin:remark-coisotropic} that each projective hyperk\"ahler manifold admits a Lagrangian constant cycle subvariety. In the following, we focus on hyperk\"ahler manifolds arising from GM fourfolds. One of the crucial ideas is the following proposition, which follows from the same argument as in \cite{Marian-Zhao} and \cite[Proposition 3.4]{SY18}.

\begin{proposition}[{\cite[Theorem]{Marian-Zhao}, \cite[Proposition 3.4]{SY18}}]\label{prop-Marian-Zhao}
Let $X$ be a GM fourfold or sixfold and~$M$ be a Bridgeland moduli space of stable objects in $\Ku(X)$ such that $M$ is a projective hyperk\"ahler manifold. For closed points $[E]$ and $[F]$ in $M$, if $$\ch(E)=\ch(F)\in\mathrm{CH}^{*}(X)_{\mathbb{Q}},$$ then $$[E]=[F]\in\mathrm{CH}_0(M).$$
\end{proposition}


Recall that we have an involution $T_X$ on $\Ku(X)$ defined in Lemma~\ref{lem-involution} for each GM fourfold or sixfold $X$. According to \cite[Remark 5.8]{bayer2022kuznetsov} and \cite[Theorem 1.6]{ppzEnriques2023}, the functor $T_X$ 
induces an anti-symplectic involution on $M_{\sigma_X}^X(a,b)$ which we also denote by $T_X$, and the fixed locus is non-empty. We denote by~$\mathrm{Fix}(T_X)$ the fixed locus, then $\mathrm{Fix}(T_X)$ is a smooth Lagrangian subvariety of $M_{\sigma_X}^X(a,b)$ by~\cite[Lemma 1.1]{Beauville11}. 

\begin{theorem}\label{thm-CH0}
Let $X$ be a GM fourfold or sixfold and $a,b$ be a pair of coprime integers. Then for any two objects $E_1,E_2\in \Ku(X)$ with  $\ch(E_1)=\ch(E_2)\in \mathrm{H}^*(X, \QQ)$, we have
\begin{equation}\label{eq-chern}
\ch(E_1)+\ch(T_X(E_1))=\ch(E_2)+\ch(T_X(E_2))\in \mathrm{CH}^*(X)_{\QQ}.
\end{equation}
In particular, for any stability condition $\sigma_X\in \Stab^{\circ}(\Ku(X))$ that is generic with respect to $a\Lambda_1+b\Lambda_2$, the fixed locus $\mathrm{Fix}(T_X)$ is a Lagrangian constant cycle subvariety of $M_{\sigma_X}^X(a,b)$.
\end{theorem}

\begin{proof}
By Proposition \ref{prop-Marian-Zhao}, we only need to show \eqref{eq-chern}. Indeed, let $[E_1], [E_2]\in \mathrm{Fix}(T_X)$. As $T_X(E_i)\cong E_i$ for each $i=1,2$, by \eqref{eq-chern}, we have
\[\ch(E_1)=\ch(E_2)\in \mathrm{CH}^*(X)_{\QQ}.\]
Thus, Proposition \ref{prop-Marian-Zhao} implies $[E_1]=[E_2]\in \mathrm{CH}_0(M_{\sigma_X}^X(a,b))$.

To prove \eqref{eq-chern}, we first assume that $X$ is a GM fourfold. By \cite[Proposition 4.2]{liefu:cycle-gushel-mukai}, we know that the cycle map 
\[\mathrm{CH}^i(X)\to \mathrm{H}^{2i}(X, \ZZ)\]
is injective for each $i\neq 3$. Hence, $\ch(E_1)=\ch(E_2)\in \mathrm{H}^*(X, \QQ)$ implies 
\[\ch_i(E_1)=\ch_i(E_2)\in \mathrm{CH}^*(X)_{\QQ}\]
for each $i\neq 3$. Now for any smooth hyperplane section $j\colon Y\hookrightarrow X$, we have  $\mathrm{CH}^3(Y)\cong \mathrm{H}^6(Y,\ZZ) \cong \ZZ$ according to \cite[Theorem 4.1]{liefu:cycle-gushel-mukai}. Thus we obtain 
\[\ch(j^*E_1)=\ch(j^*E_2)\in \mathrm{CH}^*(Y)_{\QQ},\]
which implies $\ch(j_*j^*E_1)=\ch(j_*j^*E_2)\in \mathrm{CH}^*(X)_{\QQ}$ and $\ch(\pr_X(j_*j^*E_1))=\ch(\pr_X(j_*j^*E_2))\in \mathrm{CH}^*(X)_{\QQ}$. Therefore, the equality \eqref{eq-chern} follows from the exact triangle \eqref{eq-triangle-functor-2}.

Similarly, when $X$ is a GM sixfold, we have 
\[\ch_i(E_1)=\ch_i(E_2)\in \mathrm{CH}^*(X)_{\QQ}\]
for each $i\neq 4$ by \cite[Theorem 4.5]{liefu:cycle-gushel-mukai}. Since $\mathrm{CH}^4(Y)\cong \ZZ$ (cf.~\cite[Theorem 4.4]{liefu:cycle-gushel-mukai}), we still have~$\ch(\pr_X(j_*j^*E_1))=\ch(\pr_X(j_*j^*E_2))$ in $\mathrm{CH}^*(X)_{\QQ}$. Then \eqref{eq-chern} follows from \eqref{eq-triangle-functor-2} as well.
\end{proof}

\begin{remark}
By the modular interpretation in \cite[Theorem 1.1]{GLZ2021conics} and \cite[Corollary 1.3]{kapustka2022epw}, Theorem \ref{thm-CH0} and \cite[Proposition 5.16]{perry2019stability} imply that the fixed loci of natural involutions on double (dual) EPW sextics (cf.~\cite{o2010double}) and double EPW cubes (cf.~\cite{IKKR19}) associated with very general GM fourfolds are Lagrangian constant cycle subvarieties. For double EPW sextics, this is proved in \cite{zhang2023one}.

\end{remark}

\begin{appendix}

\section{Cubic threefolds/fourfolds}\label{App-cubic}

In this section, we delve into potential generalizations of our result in Theorem \ref{thm-very-general-GM4} for GM varieties to the case of cubic threefolds/fourfolds. Subsequently, we will discuss its potential applications.

For a cubic fourfold $X$, its semi-orthogonal decomposition is given by $$\D^b(X)=\langle\Ku(X),\oh_X,\oh_X(H),\oh_X(2H)\rangle=\langle\oh_X(-H),\Ku(X),\oh_X,\oh_X(H)\rangle,$$
where $H$ is the hyperplane class of $X$, satisfying $S_{\Ku(X)}=[2]$. We define the projection functor  $$\pr_X:=\bR_{\oh_X(-H)}\bL_{\oh_X}\bL_{\oh_X(H)}.$$
There is a rank two lattice in the numerical Grothendieck group $\Knum(\Ku(X))$ generated by $\Lambda_1$ and $\Lambda_2$ with $$\ch(\Lambda_1)=3-H-\frac{1}{2}H^2+\frac{1}{6}H^3+\frac{1}{8}H^4,\quad \ch(\Lambda_2)=-3+2H-\frac{1}{3}H^3,$$ 
over which the Euler pairing is of the form
\begin{equation}
\left[               
\begin{array}{cc}   
-2 & 1 \\  
1 & -2\\
\end{array}
\right].
\end{equation}
When $X$ is non-Hodge-special, we have $\Knum(\Ku(X))=\langle\Lambda_1, \Lambda_2 \rangle$.

For a cubic threefold $Y$, we have a semi-orthogonal decomposition
\[\D^b(Y)=\langle \Ku(Y), \oh_Y, \oh_Y(H)\rangle.\]
In this case, $S_{\Ku(Y)}=\bL_{\oh_Y}\circ (-\otimes \oh_Y(H))[1]$ and $S^3_{\Ku(Y)}=[5]$. Moreover, $\Knum(\Ku(Y))$ is a rank two lattice generated by $\lambda_1$ and $\lambda_2$ with 
$$\ch(\lambda_1)=2-H-\frac{1}{6}H^2+\frac{1}{6}H^3 ,\quad \ch(\lambda_2)=-1+H-\frac{1}{6}H^2-\frac{1}{6}H^3$$
and the Euler pairing is 
\begin{equation}
\left[               
\begin{array}{cc}   
-1 & 1 \\  
0 & -1\\
\end{array}
\right].
\end{equation}
We define the projection functor by $\pr_Y:=\bL_{\oh_Y(H)}\bL_{\oh_Y}$. 
Let $\sigma_Y$ be a Serre invariant stability condition on $\Ku(Y)$. We denote by $M^Y_{\sigma_Y}(a,b)$ the moduli space of S-equivalence classes of $\sigma_Y$-semistable objects in $\Ku(Y)$ with class $a\lambda_1+b\lambda_2$.

\begin{conjecture}\label{main-conjecture-1}
Let $X$ be a cubic fourfold or a GM fourfold and $j\colon Y\hookrightarrow X$ be its smooth hyperplane section. Let $\sigma_Y$ be a Serre-invariant stability condition on $\Ku(Y)$. 
If $\mathrm{gcd}(a,b)=1$, then there is a stability condition $\sigma_X\in \Stab^{\circ}(\Ku(X))$ generic with respect to $a\Lambda_1+b\Lambda_2$ and an open smooth subscheme $U\subset M^Y_{\sigma_Y}(a,b)$ such that $\pr_X(j_*E)$ is $\sigma_X$-semistable for any $[E]\in U$.
\end{conjecture}

Note that when $\gcd(a,b)=1$, $M^Y_{\sigma_Y}(a,b)$ is smooth by \cite{PY20}. From now on, let $X$ be a cubic fourfold and $j \colon Y\hookrightarrow X$ be a smooth hyperplane section. Moreover, assume that $\sigma_Y$ is a Serre invariant stability condition on $\Ku(Y)$ and $\sigma_X$ belongs to the family of stability conditions $\Stab^{\circ}(\Ku(X))$ constructed in \cite[Theorem 29.1]{BLMNPS21}. According to \cite[Theorem 29.2]{BLMNPS21}, for any pair of coprime integers $a,b$, and stability condition $\sigma_X\in \Stab^{\circ}(\Ku(X))$ which is generic with respect to $a\Lambda_1+b\Lambda_2$, the moduli space $M^X_{\sigma_X}(a,b)$ is a projective hyperk\"ahler manifold.

\begin{lemma}\label{pullback-triangle}
   For any object $E\in \Ku(Y)$, we have an exact triangle
    \begin{equation} \label{eq-pullback-triangle-1}
        S^{-1}_{\Ku(Y)}(E)[2]\to j^*\pr_X(j_*E)\to E.
    \end{equation}
In particular, if $E\in \Ku(Y)$ is $\sigma_Y$-semistable for a Serre-invariant stability condition $\sigma_Y$ on $\Ku(Y)$, then $j^*\pr_X(j_*E)$ is not $\sigma_Y$-semistable and \eqref{eq-pullback-triangle-1} is the Harder--Narasimhan filtration of $j^*\pr_X(j_*E)$ with
    \[\phi_{\sigma_Y}(E)+1>\phi_{\sigma_Y}(S^{-1}_{\Ku(Y)}(E)[2])>\phi_{\sigma_Y}(E).\]
\end{lemma}
\begin{proof}
As in the proof of Lemma \ref{lem-exact-triangle}(1), we only need to determine $\mathbf{T}_{j^*, \pr_X\circ j_*}(E)$. From \cite[Proposition 4.16, Corollary 4.19(1)]{kuznetsov:serre-dimension}, we see $\mathbf{T}_{j^*, \pr_X\circ j_*}\cong S^{-1}_{\Ku(Y)}\circ [3]$. Hence, \eqref{eq-pullback-triangle-1} is verified as desired. The rest follows from \eqref{eq-pullback-triangle-1}, \cite[Proposition 3.4(a), (d)]{FeyzbakhshPertusi2021stab}, and Lemma \ref{lem-T-fix-phase}.
\end{proof}

As a result of Lemma \ref{pullback-triangle}, one observes that unlike Theorem \ref{thm-very-general-GM4} for GM varieties, Conjecture \ref{main-conjecture-1} may not be true for $U = M^Y_{\sigma_Y}(a,b)$ for the cubic threefold $Y$. Nevertheless, the same argument as in Lemma~\ref{prop_general-2} implies the following.  

\begin{lemma} \label{prop_general}
    Let $E\in \Ku(Y)$ be a $\sigma_Y$-semistable object. 

\begin{enumerate} 
    \item \emph{(Injectivity of tangent maps)} The natural map
    \[\Ext^1_Y(E,E)\to \Ext^1_X(\pr_X(j_*E), \pr_X(j_*E))\]
    is injective.

    \item \emph{(Fibers)} Let $E'\neq E\in \Ku(Y)$ be another $\sigma_Y$-semistable object such that 
    $$\phi_{\sigma_Y}(E)=\phi_{\sigma_Y}(E'),\quad \text{and } \Hom_Y(E, E')=0,$$ then 
    \[\Hom_X(\pr_X(j_*E), \pr_X(j_*E'))=0.\] 
\end{enumerate}    
\end{lemma}

As a result, we deduce the following immersion via the same argument as in Theorem \ref{thm-lagrangian-general}. 

\begin{theorem} \label{thm-lagrangian-general-1}
Let $a, b$ be a pair of integers with $\gcd(a, b) =1$ such that Conjecture \ref{main-conjecture-1} holds. Then the functor $\pr_X\circ j_*$ induces a rational map defined over $U$
\[r \colon \overline{U} \dashrightarrow M^X_{\sigma_X}(a,b),\]
such that the image $L$ is Lagrangian in $M^X_{\sigma_X}(a,b)$. Furthermore, $r|_U$ is an 
immersion and $\overline{U}$ is birational onto $L$.
\end{theorem}
\begin{proof}
According to Lemma \ref{prop_general}, $r|_U$ 
is injective at the level of underlying spaces. Hence, $r|_U$ is an immersion and $\overline{U}$ is birational to $L$. The rest follows via the same argument as in Theorem \ref{thm-lagrangian-general}.



\end{proof}

\subsection{Lagrangian covering families}

If an analog of Corollary \ref{cor-push-stable} holds for a cubic fourfold $X$ and it yields a Lagrangian covering family, then $\dim M^X_{\sigma_X}(a,b)\leq 10$. Up to sign, we have
\[(a,b)=(0, 1), (1,0), (1,\pm 1), (2,\pm 1), (1, \pm 2).\]

\begin{itemize}
    \item $(a,b)=(0, 1), (1,0), (1,\pm 1)$. In this case, the functor $\bL_{\oh_X}\circ (-\otimes \oh_X(H))$ is an auto-equivalence of $\Ku(X)$ that switches these classes up to sign. Hence, we only need to consider the case $(a,b)=(1,1)$. According to \cite[Theorem 1.1]{li2018twisted}, the moduli space for $(a,b)=(1,1)$ on a cubic fourfold is the Fano variety of lines. On the other hand, the moduli space for $(a,b)=(1,1)$ on a cubic threefold is the Fano surface of lines (cf.~\cite{PY20}). We will see in \cite{FGLZcube} that Theorem \ref{thm-lagrangian-family} yields a Lagrangian covering family, aligning with the classical construction in \cite{voisin:torelli-cubic-fourfold}.

    \item $(a,b)=(2,\pm 1), (1, \pm 2)$. In this case, the functor $\bL_{\oh_X}\circ (-\otimes \oh_X(H))$ also switches these classes up to sign, allowing us to focus solely on $(a,b)=(2,1)$. By \cite[Theorem 1.2]{li2018twisted}, the moduli space for $(a,b)=(2,1)$ on a cubic fourfold is the LLSvS eightfold constructed in \cite{LLSvS17}. On the other hand, the moduli space for $(a,b)=(2,1)$ on a cubic threefold is considered in \cite{bayer2020desingularization}. We will see in \cite{FGLZcube} that Theorem \ref{thm-lagrangian-family} provides a Lagrangian covering family, consistent with the classical construction in \cite{shinder2017geometry}.
\end{itemize}

\end{appendix}

\bibliography{lagrangian}

\bibliographystyle{alpha}

\end{document}